\documentclass[aap]{imsart}

\RequirePackage{amsthm,amsmath,amsfonts,amssymb}
\RequirePackage[numbers]{natbib}
\RequirePackage[colorlinks,citecolor=blue,urlcolor=blue]{hyperref}
\RequirePackage{graphicx}

\startlocaldefs
\usepackage{accents,stackengine,scalerel} 
\usepackage{yfonts}

\theoremstyle{plain}
\newtheorem{lemma}{Lemma}[section]
\newtheorem{proposition}[lemma]{Proposition}
\newtheorem{corollary}[lemma]{Corollary}

\newtheorem{theorem}[lemma]{Theorem}

\theoremstyle{definition}
\newtheorem{definition}[lemma]{Definition}
\newtheorem*{remark}{Remark}

\newcommand{\cev}[1]{\accentset{\leftharpoonup}{#1}}
\newcommand{\vecc}[1]{\accentset{\rightharpoonup}{#1}}

\newcommand{\bC}{\mathbb{C}}
\newcommand{\bD}{\mathbb{D}}
\newcommand{\bE}{\mathbb{E}}

\newcommand{\bN}{\mathbb{N}}
\newcommand{\bP}{\mathbb{P}}
\newcommand{\bR}{\mathbb{R}}

\newcommand{\cC}{\mathcal{C}}

\newcommand{\cE}{\mathcal{E}}

\newcommand{\cI}{\mathcal{I}}

\newcommand{\cM}{\mathcal{M}}

\newcommand{\ff}{\mathbf{f}}

\newcommand{\fN}{\mathbf{N}}

\newcommand{\fX}{\mathbf{X}}

\newcommand{\fW}{\mathbf{W}}

\newcommand{\ed}{\mbox{$ \ \stackrel{d}{=}$ }}

\newcommand{\besq}{{\tt BESQ}}
\newcommand{\ekp}{{\tt EKP}}
\newcommand{\pd}{{\tt PD}}
\newcommand{\skewer}{\ensuremath{\normalfont\textsc{skewer}}}
\newcommand{\sskewer}{\ensuremath{\normalfont\textsc{sSkewer}}}
\newcommand{\skewerbar}{\ensuremath{\overline{\normalfont\textsc{skewer}}}}
\newcommand{\clade}{\ensuremath{\normalfont\textsc{clade}}}

\newcommand{\cskewer}{\ensuremath{\normalfont c\textsc{skewer}}}
\newcommand{\fskewer}{\ensuremath{\normalfont f\textsc{skewer}}}
\newcommand{\cskewerbar}{\ensuremath{\overline{\normalfont c\textsc{skewer}}}}
\newcommand{\fskewerbar}{\ensuremath{\overline{\normalfont f\textsc{skewer}}}}

\newcommand{\PRM}{\mathtt{PRM}}

\def\Concat{ \mathop{ \raisebox{-2pt}{\Huge$\star$} } }

\def\concat{\star}

\newcommand{\Leb}{\mathrm{Leb}}

\newcommand{\ind}{\mathbf{1}}
\newcommand{\pdip}{\mathtt{PDIP}}
\endlocaldefs

\begin{document}

\begin{frontmatter}
\title{Up-down ordered Chinese restaurant processes with\\ two-sided immigration, emigration and diffusion limits}
\runtitle{Chinese restaurant processes and their diffusion limits}

\begin{aug}
\author[A]{\fnms{Quan}~\snm{Shi}\ead[label=e1]{quan.shi@amss.ac.cn} 
},
\author[B]{\fnms{Matthias}~\snm{Winkel}\ead[label=e2]{winkel@stats.ox.ac.uk}}

\address[A]{SKLMS, Academy of Mathematics and Systems Science, Chinese Academy of Sciences, China; \printead{e1}}

\address[B]{Department of Statistics, University of Oxford, UK; \printead{e2}}
\end{aug}

\begin{abstract}
We establish scaling limit theorems for the up-down ordered Chinese restaurant processes (oCRPs) of Rogers and Winkel as processes in a space of interval partitions. As previously conjectured, the limits are self-similar diffusions previously constructed directly in the continuum. We extend the oCRP model and the results to a three-parameter family ${\rm oCRP}^{(\alpha)}(\theta_1,\theta_2)$, $\alpha\in(0,1)$, $\theta_1,\theta_2\ge 0$.  We use the scaling limit approach to extend existing stationarity results to the full three-parameter family, identifying an extended family of Poisson--Dirichlet interval partitions. Their ranked sequence of interval lengths has Poisson--Dirichlet distribution with parameters $\alpha\in(0,1)$ and 
$\theta:=\theta_1+\theta_2-\alpha\ge-\alpha$, including for the first time the usual range of $\theta>-\alpha$ rather than being restricted to $\theta\ge 0$. This has applications to Fleming--Viot processes, nested interval partition evolutions and tree-valued Markov processes, notably relying on the extended parameter range.
\end{abstract}



\end{frontmatter}

\section{Introduction}
The primary purpose of this paper is to study the weak convergence of a family of properly rescaled continuous-time Markov chains on integer compositions \cite{Gnedin97} and the limiting diffusions. 
Our results should be compared with the scaling limits of natural up-down Markov chains on branching graphs, which have received substantial focus in the literature \cite{BoroOlsh09, Petrov09, Petrov13}. In this language, our models take place on the branching graph of integer compositions and on its boundary, which was represented in \cite{Gnedin97} as a space of interval partitions. This paper establishes a proper scaling limit connection between discrete models \cite{RogeWink20} and their
continuum analogues \cite{Paper1-1,Paper1-2,IPPAT} in the generality of \cite{ShiWinkel-1}. 

Specifically, for $n\in \bN=\{1,2,\ldots\}$, \emph{a composition of $n$} is a sequence of positive integers that sum to $n$. 
Denote by $\cC_n$ the set of all compositions of $n$, with $\cC_0:= \{\emptyset\}$ by convention. 
We consider a directed graph whose vertices are given by the collection of all compositions $\cC= \bigcup_{n\ge 0}\cC_n$. 
For ease of presentation, it is convenient to regard each composition of $n$ as the allocation of $n$ customers to a row of ordered tables in a restaurant. 
Then, for any vertex $\beta\in \cC_n$, there is a directed edge from $\beta$ to another composition $\lambda$, denoted by $\beta\nearrow \lambda$, if and only if 
$\lambda\in \cC_{n+1}$ and $\lambda$ can be obtained from $\beta$ by adding one customer, either to an occupied table or to a new table inserted to the array at a specific position. Note that there may be different choices of such operations that result in the same $\lambda$, and we define the number of ways  to be the \emph{multiplicity} of the edge, denoted by $\kappa(\beta,\lambda)$.  For example, $\kappa((1,1,2,3), (1,1,1, 2,3))=3$.  
If there is an edge $\beta\nearrow \lambda$, then $\beta$ can be obtained from $\lambda$ by removing one customer and we also say $\lambda\searrow\beta$. 
The graph of compositions is in the class of \emph{branching graphs} which are widely studies in algebraic combinatorics, representation theory and probability. See e.g.\ \cite{Petrov13} and reference therein. Other prototypes of branching graphs include the Young graph and the Schur graph of partitions and the Pascal triangle.

We are interested in a family of continuous-time Markov chains $(C(t),t \ge 0)$ on the graph of compositions, parametrised by $\alpha\in [0,1)$ and $\theta\ge 0$. This family
was studied in \cite{RogeWink20}. 
We refer to this process as a Poissonised  ordered up-down  \emph{$(\alpha,\theta)$-Chinese restaurant process}, denoted by $\mathrm{PCRP}^{(\alpha)}(\theta)$, in view of the following interpretation. 
In the language of  the well-known Chinese restaurant processes due to Dubins and Pitman (see e.g.\@ \cite{CSP}, also \cite{James2006,PitmWink09}), this model describes the evolution of the numbers of customers at  a row of tables ordered from left to right. If $C(t)= (n_1, \ldots , n_k)\in \cC_n$, then it means that, at time $t\ge 0$, there are $n$ customers in total, sitting at $k$ different tables, and the $i$-th occupied table enumerated from left to right has $n_i\in \bN$ customers.  
The generator $\mathcal{A}$ is given by 
\[
\mathcal{A} f (\beta) = n\!\!\sum_{\lambda\in\mathcal{C}_{n-1}}\!\! p^{\downarrow} (\beta, \lambda) f(\lambda)+   (n\!+\!\theta)\!\!\sum_{\lambda\in\mathcal{C}_{n+1}}\!\! p^{\uparrow} (\beta,\lambda) f(\lambda)- (2n\!+\!\theta)f(\beta), \, \beta\in \cC_n, n\ge 1, 
\]
and $\mathcal{A} f (\emptyset) = \theta f( (1)) - \theta f(\emptyset)$, 
where $p^{\uparrow}$ and $p^{\downarrow}$ are stochastic kernels on $\mathcal{C}$ specified as follows.  
The up kernel $p^{\uparrow}$ captures the arrival of  a new customer, who either takes a seat at an existing table or starts a new table. Specifically, as illustrated in Figure~\ref{fig:oCRP-alpha}, for each $\beta\in\mathcal{C}_n$, the kernel $(n+\theta)p^\uparrow(\beta,\cdot)$ assigns
\begin{figure}[t]
	\centering
	\includegraphics[width=0.8\linewidth]{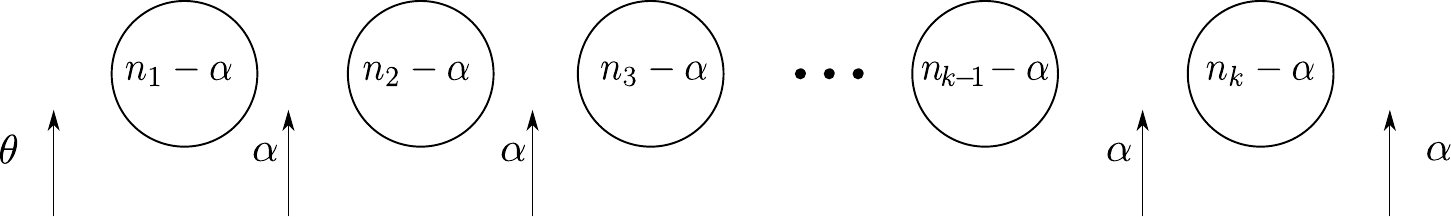}
	\caption{The up transition rates $(n+\theta)p^{\uparrow}((n_1,\ldots,n_k),\cdot)$
	}
	\label{fig:oCRP-alpha}
\end{figure}
\begin{itemize}
	\item rate $m-\alpha$ to the arrival of a new customer to any occupied table with $m\in\bN$ customers;
	\item rate $\theta$ to the arrival of a new customer to a new table at the very left; 
	\item rate $\alpha$ to the arrival at a new table to the right of each occupied table. 
\end{itemize}
The down kernel $p^{\downarrow}$ on $\cC$ corresponds to removing uniformly one customer: $np^\downarrow(\beta,\cdot)$ assigns
\begin{itemize}
	\item rate 1 to the departure of each of the $n$ customers in $\beta\in\mathcal{C}_n$.
\end{itemize} 
%
%
%
For $\beta\in \cC_n$, let $g(\beta):= \sum_{\lambda_0=\emptyset\nearrow \lambda_1\nearrow \ldots\nearrow \lambda_n=\beta} \prod_{i=0}^{n-1}\kappa(\lambda_i, \lambda_{i+1})$, where the sum is over all directed paths from $\emptyset$ to $\beta$. 
The description yields that the down kernel satisfies the following (as can be read from \cite[Pages 308--309]{RivRiz})   
\[
p^{\downarrow}(\lambda,\beta)=\frac{\kappa(\beta,\lambda) g(\beta)}{g(\lambda)}\ind\{\lambda\searrow\beta\}, \qquad n\ge 0, \beta\in \cC_{n}, \lambda\in \cC_{n+1}. 
\]
Moreover, the up kernel is compatible with a two-parameter family of composition structures $(\Pi^{(\alpha,\theta)}_n,n\ge 0)$ with $\alpha\in [0,1)$ and $\theta\ge 0$ studied in \cite[Section~8]{GnedPitm05} and \cite{PitmWink09}, where each $\Pi^{(\alpha,\theta)}_n$ is a probability measure on $\cC_n$, and they satisfy
\[
\Pi^{(\alpha,\theta)}_{n+1}(\beta)  = \sum_{\lambda\colon \lambda\nearrow\beta}\Pi^{(\alpha,\theta)}_{n}(\lambda)p^\uparrow(\lambda,\beta)  \qquad\mbox{for all } n\ge 0,\beta\in\mathcal{C}_{n+1}.
\]
That $(\Pi^{(\alpha,\theta)}_n, n\ge 0)$ is a \emph{composition structure} \cite{Gnedin97} refers to the fact that it satisfies
\begin{equation}\label{eq:cs}
	\Pi^{(\alpha,\theta)}_{n}(\beta)  = \sum_{\lambda\colon \lambda\searrow\beta}\Pi^{(\alpha,\theta)}_{n+1}(\lambda)p^\downarrow(\lambda,\beta), \qquad \mbox{for all }n\ge 0,\beta\in\mathcal{C}_n.
\end{equation}
In the language of branching graphs, a composition structure $(\Pi^{(\alpha,\theta)}_n, n\ge 1)$ is also called a \emph{coherent} system \cite[Definition~8]{Petrov13}. 
Note that, by \cite[Proposition~6]{PitmWink09} only if $\alpha=\theta$ we have in addition that
\[
\Pi^{(\alpha,\theta)}_n(\beta)p^{\uparrow}(\beta,\lambda)={\Pi^{(\alpha,\theta)}_{n+1}(\lambda)}p^{\downarrow}(\lambda,\beta), \qquad\mbox{for all }n\ge 0, \beta\in \cC_{n}, \lambda\in \cC_{n+1}. 
\]
So the up-down Markov chains that we study go beyond the setting of \cite[Definition~9]{Petrov13}.  


Motivated by \cite{Ald-web,Pal13,Paper0,Paper1-1}, Rogers and Winkel \cite{RogeWink20} related the class of Markov chains $\mathrm{PCRP}^{(\alpha)}(\theta)$ to splitting trees \cite{Lambert10} and conjecture the existence of diffusion limits. They support their conjecture by several convergence results of one-dimensional stochastic processes: the L\'evy processes encoding splitting trees and the integer-valued up-down chains capturing the evolutions of customer numbers in the restaurant or at a given table. Inspired by the discrete considerations, \cite{Paper1-1,Paper1-2,IPPAT,ShiWinkel-1} proposed the analogous continuum models, so-called \emph{self-similar interval-partition evolutions} (SSIP-evolutions), in increasing generality. SSIP-evolutions and their "de-Poissonised" processes (see Definition~\ref{defn:dePoi}) have been used to construct the Aldous diffusion \cite{Paper-AldousDiffusion} as a diffusive evolution of the Brownian continuum random tree \cite{AldousCRT1}.   
Although SSIP-evolutions and their de-Poissonisations have been conjectured to be the scaling limits of certain Markov chains on integer compositions, there has so far been no proof in the literature. 

To fill this gap, we prove in our first main result (see Theorem~\ref{thm:crp-ip} for a formal statement in a more general context with more general initial values) that SSIP-evolutions are the scaling limits of $\mathrm{PCRP}^{(\alpha)}(\theta)$,  confirming a conjecture in \cite{RogeWink20}. 
Specifically, let $(C(t),t\ge 0)$ be a $\mathrm{PCRP}^{(\alpha)}(\theta)$ starting from $\emptyset\in\mathcal{C}_1$.
Then  
\[
\left(\frac{1}{n} C(2 n t), t\ge 0 \right)\underset{n\to \infty}{\longrightarrow} (\beta(t),t\ge 0), \qquad \text{in distribution,}
\]
where $(\beta(t),t\ge 0)$ is an SSIP-evolution. 
This convergence is under the Skorokhod topology of c\`adl\`ag functions on a certain metric space $(\mathcal{I}_H, d_H)$, which can be viewed \cite{Gnedin97} as a subspace of the boundary of $\cC$. More precisely, $\cI_{H}$ is the collection of all \emph{interval partitions} $\gamma$ of $[0,M]$ for all $M\ge 0$, where each $\gamma$ is a set of disjoint open intervals in $[0,M]$ such that its partition points $G(\gamma):=[0,M]\setminus\bigcup_{U\in\gamma}U$ has Lebesgue measure zero. 
Furthermore, $d_H$ is the Hausdorff metric applied to the (closed) sets of partition points: 
for every $\gamma,\gamma'\in \cI_{H}$, 
\[
d_{H} (\gamma, \gamma')
:= \inf \bigg\{r\ge 0\colon G(\gamma)\subseteq \bigcup_{x\in G(\gamma')} (x-r,x+r),~
G(\gamma')\subseteq \bigcup_{x\in G(\gamma)} (x-r,x+r) \bigg\}.
\]
Though $(\cI_H, d_H)$ is not complete, the induced topological space is Polish \cite[Theorem~2.3]{Paper1-0}.

Recently, \cite{RivRiz} obtained the scaling limits of a closely related family of discrete-time up-down ordered Chinese restaurant processes, in which at each time exactly one customer arrives and then one customer leaves uniformly, such that the number of customers remains constant. But neither limit theorem implies the other.  
Their method, as for other branching graphs \cite{BoroOlsh09, Petrov09, Petrov13}, is by analysing the generator of Markov processes, which is quite different from our paper.

The discrete approximation established in this work in turn permits us to understand more properties of SSIP-evolutions. In particular, we prove the pseudo-stationarity (Proposition~\ref{prop:ps-theta1theta2-nokill}) for the generalised models introduced in \cite{ShiWinkel-1} and below. 
Also, our convergence theorem should give us access to the generator of the limiting SSIP-evolutions, which is at present only partially known \cite{Paper1-3}.

\subsection{More Contributions of this work}

\subsubsection{Coverage of the full range of parameters with $\theta\in [-\alpha, 0)$}
Keeping track of only the sizes of parts, and not their order, every composition
structure induces a partition structure, that is, a sequence of sampling consistent
partitions of integers. 
The aforementioned  composition 
structures $\Pi^{(\alpha,\theta)}$ correspond to the Ewens--Pitman partition structures. 
Note that the two-parameter model of partition structures extends to $\theta>-\alpha$. 
In previous studies of the two-parameter model of composition structures, only $\theta\ge 0$ is covered, in a setting where there exists a regenerative random order \cite{GnedPitm05,PitmWink09}. 
We fill the gap $\theta\in (-\alpha,0)$, the \emph{emigration} case, by a natural extension to a three-parameter model introduced in Section~\ref{sec:PDIP}. 
The scaling limit result in Theorems~\ref{thm:crp-ip} is in fact established in this general setting. 

\subsubsection{Fleming--Viot superprocesses}
In \cite{FVAT}, we introduced a family of Fleming--Viot superprocesses parametrised by $\alpha\in (0,1)$, $\theta\ge 0$, generalising the labeled infinitely-many-neutral-alleles model introduced by Ethier and Kurtz \cite{EthiKurt93,EthKurtzBook}. 
These Fleming--Viot superprocesses take values on the space $(\cM^a_1,d_{\cM})$ of all purely atomic probability measures on $[0,1]$ endowed with the Prokhorov distance. Our construction in \cite{FVAT} is closely related to interval partition evolutions. 
We can now extend this model to the case $\theta\in [-\alpha,0)$ and identify the desired stationary distribution, the two-parameter Pitman--Yor process \cite{PitmanYor97,IshwJame01}, here exploiting the connection with an $\mathrm{SSIP}$-evolution in the emigration case. This is discussed in more detail in Section~\ref{sec:FV}. 

\subsubsection{Convergence of excursion measures}

We construct a sigma-finite measure of excursions of SSIP-evolutions and prove that SSIP-excursions appear as the scaling limits of excursions on compositions (Theorem~\ref{thm:Theta}).  

\subsubsection{Nested interval partition evolutions}
For $0\le \bar{\alpha}\le \alpha<1$ and $\bar\theta> 0$, consider a random mass partition with Poisson--Dirichlet distribution $\mathtt{PD}^{(\bar\alpha)} (\bar\theta)$ on the space 
$\nabla_{\infty}$ of nonnegative decreasing sequences with sum 1. If we split each fragment independently into $\mathtt{PD}^{(\alpha)} (-\bar\alpha)$ proportions, then the finer mass partition has distribution $\mathtt{PD}^{(\alpha)} (\bar\theta)$ \cite[(5.24)]{CSP}. Informally speaking, the two-parameter family $\mathtt{PD}^{(\alpha)} (\theta)$ is nested.  
The ordered analogue of nested mass partitions, nested interval partitions, appears in the study of coalescents \cite{FouLamSch21} and implicitly in the flow-of-bridges construction of continuous-state branching processes \cite{BerLeG03}.  

Nested Markov chains on partitions have been considered in non-parametric Bayesian statistics \cite{Blei10}. 
For the continuum limit, both on partitions and compositions, this aspect has so far been absent in the literature. 
In Section \ref{sec:nestedPCRP}, we extend Theorem \ref{thm:crp-ip} to the setting of  nested Markov chains on compositions and obtain nested SSIP-evolutions as scaling limits. At each time, the values of the nested limiting diffusions form a family of nested interval partitions. We also deduce the pseudo-stationarity for the nested SSIP-evolutions via discrete approximation. Our approach relies on the study of the emigration case and the SSIP-excursions mentioned above.   

Our nested interval partition evolutions provide important insights into understanding (multifurcating) continuum-tree-valued processes, as we will see below. 


\subsection{Further motivations from the study of random trees}

An initial motivation of this work was from studies of diffusions on a space of continuum trees.
Aldous \cite{Aldous00} introduced a Markov chain on the space of binary trees with $n$ leaves, by removing a leaf uniformly and reattaching it to a random edge. This Markov chain has the uniform distribution as its stationary distribution. 
As $n\!\to\! \infty$, Aldous conjectured that the limit of this Markov chain is a diffusion on continuum trees with stationary distribution given by the \emph{Brownian continuum random tree (CRT)}, i.e.\@ the universal scaling limit of random discrete trees with finite vertex degree variance. 

Among different approaches 
\cite{LohrMytnWint18,Paper-AldousDiffusion} investigating this problem, \cite{LohrMytnWint18} shows the existence of a process by using a martingale problem on a new space of trees, which they call algebraic measure trees. Roughly speaking, this new space retains knowledge of the branch points as mass splits, but does not include the metric structure. 
On the other hand, \cite{Paper-AldousDiffusion} provides a more detailed description of the Aldous diffusion, which grants access to sample path properties of the diffusion. 
\cite{Paper-AldousDiffusion} describes the evolution via a consistent system of spines endowed with lengths and subtree masses, which relies crucially on interval partition evolutions. 
We stress that \cite{Paper-AldousDiffusion} only gives the continuum construction, while the full convergence of the Markov chains on discrete trees is still an open problem. 
Our scaling limit result provides a building block for a complete solution of this problem. This is a new contribution of our work to this domain.

More significantly, the current work provides tools towards understanding dynamics on multifurcating trees with possibly unbounded degrees, while both approaches in \cite{LohrMytnWint18,Paper-AldousDiffusion} mainly focus on binary trees and cannot be obviously extended to the multifurcating case. 
Specifically, our aforementioned nested interval partition evolutions can be related to the spinal decompositions of multifurcating trees developed in \cite{HPW}.  
Sample a leaf uniformly at random and consider its path to the root, called the \emph{spine}. We say that two vertices $x,y$ of the tree are equivalent, if the paths from $x$ and $y$ to the root first meet the spine at the same point. Then equivalence classes are bushes of spinal subtrees rooted at the same spinal branch point. In each equivalence class, by deleting the branch point on the spine, the subtrees form smaller connected components.   
The collection of equivalence classes is called \emph{the coarse spinal decomposition}, and the collection of all subtrees is called \emph{the fine spinal decomposition}. 
Some variants of our aforementioned nested interval partition evolutions can be used to represent the mass evolution of  the nested coarse and fine spinal decompositions in continuum multifurcating tree-valued dynamics. 
The order structure provided by the interval partition evolutions also plays a crucial role. At the coarse level, the equivalence classes of spinal bushes are naturally ordered by the distance of the spinal branch points to the root. At the fine level, a total order of the subtrees in the same equivalence class aligns with the \emph{semi-planar structure} introduced in a study of related Markov chains on discrete trees \cite{Soerensen}. This is used to explore the evolutions of sizes of subtrees in a bush at a branch point. 

In future work, we aim to construct \emph{stable Aldous diffusions}, with stationary distributions given by \emph{stable L\'evy trees} with parameter $\rho\in (1,2)$ \cite{DuLG05,DuquLeGall02}, which are the infinite variance analogues of the Brownian CRT. With our techniques one could further consider more general classes of \emph{continuum fragmentation trees}, including the \emph{alpha-gamma models} \cite{CFW} or a two-parameter Poisson--Dirichlet family \cite{HPW}.
In Section~\ref{sec:trees}, we give an example of a Markov chain related to random trees that converges to our nested SSIP-evolutions.

\subsection{Organisation of the paper} 
In Section \ref{sec:PDIP} we generalise the two-parameter ordered Chinese Restaurant Processes to a three-parameter model and establish 
connections to interval partitions and composition structures. We state the scaling limit results, Theorems \ref{thm:crp-ip} and ~\ref{thm:Theta},  as well as properties of the limiting interval partition evolution, in Section~\ref{sec:mainresult}. 
In Section \ref{sec:crp}, we prove Theorem \ref{thm:crp-ip} in the two-parameter setting, building on \cite{Paper1-1,Paper1-2,IPPAT,RogeWink20}. In Section \ref{sec:pcrp-ssip}, we study the three-parameter setting, both for processes absorbed in $\emptyset$ and for recurrent extensions, which we obtain by constructing excursion measures of interval partition  evolutions. This section concludes with proofs of all results stated in Section~\ref{sec:mainresult}. We finally turn to applications in Section \ref{sec:appl}.

\section{Chinese restaurant processes}\label{sec:PDIP}
Throughout the rest of the paper, we fix $\alpha\in (0,1)$. 
In this section we will first introduce a three-parameter family of composition structures in Section~\ref{sec:oCRP} and discuss their limiting distributions in Section~\ref{sec:ocrp-pdip}. Then we will present the corresponding Markov chains on compositions and their scaling limits in Section~\ref{sec:mainresult}.

\subsection{A three-parameter family of composition structures}\label{sec:oCRP}

\begin{definition}[$\mathrm{oCRP}^{(\alpha)}(\theta_1,\theta_2)$ in discrete-time]\label{defn:oCRP}
	Let $\theta_1, \theta_2\ge 0$.  
	We start with a single customer sitting at a table and new customers arrive one-by-one. 
	When there are already $n\ge 1$ customers, the $(n+1)$-st customer is located by the following \emph{$(\alpha,\theta_1,\theta_2)$-seating rule}:   
	\begin{itemize}
		\item If a table has $m$ customers, then the new customer comes to join this table with probability $(m- \alpha)/(n +\theta)$, where $\theta:= \theta_1 + \theta_2 -\alpha$.  
		\item The new customer may also start a new table:  with probability $\theta_1/(n+\theta)$ (resp.\ $\theta_2/(n+\theta)$), she starts a new table at the very left (resp.\ right); between each pair of two neighbouring occupied tables (if there are two or more occupied tables), the new table is located there with probability $\alpha/(n+\theta)$. 
	\end{itemize}
	At each step $n\ge 1$, the numbers of customers at  the tables (ordered from left to right) form a composition $C(n)\in \cC_n$. We will refer to $(C(n), n\ge 1)$ as an \emph{ordered Chinese restaurant process with parameters $\alpha$, $\theta_1$ and $\theta_2$}, or $\mathrm{oCRP}^{(\alpha)}(\theta_1, \theta_2)$. We also denote the distribution of $C(n)$ by
	$\mathtt{oCRP}^{(\alpha)}_n(\theta_1,\theta_2)$. 
\end{definition}

In the degenerate case $\theta_1=\theta_2=0$, an $\mathrm{oCRP}^{(\alpha)}(0,0)$ is simply a deterministic process $C(n)= (n), n\ge 1$, because we start with a single customer at a single table and can never start a second occupied table in this case (the rates to start new tables at the very left and right vanish and no space between two occupied tables ever exists).  
If we do not distinguish the location of the new tables, but build the partition of $\mathbb{N}$ that has $i$ and $j$ in the same block if the $i$-th and $j$-th customer sit at the same table, then this gives rise to the well-known (unordered) $(\alpha, \theta)$-Chinese restaurant process with $\theta:= \theta_1 + \theta_2 -\alpha$; see e.g.\@ \cite[Chapter~3]{CSP}.

When $\theta_1=\alpha$, this model encompasses the family of composition structures studied in \cite[Section~8]{GnedPitm05} and \cite{PitmWink09}. 
In particular, they show that an $\mathrm{oCRP}^{(\alpha)}(\alpha, \theta_2)$ is a 
\emph{regenerative composition structure} in the following sense.

\begin{definition}[Composition structure~\cite{GnedPitm05,Gnedin97}]\label{defn:structure}
	A Markovian sequence of random compositions $(C(n), n\ge 1)$ with $C(n)\in \cC_n$
	is a \emph{composition structure}, if the following property is satisfied: 
	\begin{itemize}
		\item Weak sampling consistency: for each $n\ge 1$, if we first distribute $n+1$ identical customers into an ordered series of tables according to $C(n+1)$ and then remove one customer uniformly at random (deleting any empty table if necessary), then the resulting composition $\widetilde{C}(n)$ has the same distribution as $C(n)$.
	\end{itemize}
	Moreover, a composition structure is called \emph{regenerative}, if it further satisfies 	
	\begin{itemize}
		\item Regeneration:  
		for every $n \ge m$, conditionally on the left-most block of 
		$C(n)$ having size $m$, the remainder is a composition in $\cC_{n-m}$ with 
		the same distribution as $C(n-m)$.  
	\end{itemize}
	For $n \ge m$, let $r(n,m)$ be the probability that the first block of 
	$C(n)$ has size $m$. Then $(r(n,m),1\le m\le n)$ is called the \emph{decrement matrix} of $(C(n),n\ge 1)$. 
\end{definition}

\begin{lemma}[{\cite[Proposition~6]{PitmWink09}}]\label{lem:rcs}
	For $\theta_2\ge 0$, an $\mathrm{oCRP}^{(\alpha)}(\alpha, \theta_2)$  $(C(n), n\ge 1)$
	starting from $(1)\in \cC$ is a regenerative composition structure
	with decrement matrix 
	\[
	r_{\theta_2}(n,m):= \binom{n}{m} \frac{(n-m)\alpha+ m\theta_2}{n} \frac{\Gamma(m-\alpha)\Gamma(n-m+\theta_2)}{\Gamma(1-\alpha)\Gamma(n+\theta_2)}, \quad 1\le m\le n.
	\]
	For every $(n_1, n_2, \ldots n_k)\in \cC_n$, we have
	\[
	\bP\big(C(n)= (n_1, n_2, \ldots n_k)\big) = \prod_{i=1}^{k} r_{\theta_2}(N_{i:k}, n_i),\quad 	\text{where}~ N_{i:k}:= \sum_{j=i}^{k} n_j. 
	\] 
\end{lemma}

To understand the distribution of $(C(n), n\ge 1)\sim\mathrm{oCRP}^{(\alpha)}(\theta_1,\theta_2)$, 
let us consider the richer structure that captures a decomposition as follows. 
Recall that there is an initial table with one customer at time $1$. 
Let us distinguish this initial table from other tables. At time $n\ge 1$, we record the size of the initial table as $N_0^{(n)}$, the composition of the table sizes to the left of the initial table by $C_-^{(n)}$, and the composition to the right of the initial table by $C_+^{(n)}$. 
Then there is the identity $C(n)= C_-^{(n)} \concat( N_0^{(n)}) \concat C_+^{(n)}$, where $\concat$ denotes the \emph{concatenation}: for $\beta_1= (n_1, n_2, \ldots n_k)\in \cC_{m_1}$  and $\beta'= (n'_1, n'_2, \ldots n'_{\ell})\in \cC_{m_2}$,   
$\beta_1\concat \beta_2:= (n_1, n_2, \ldots n_k,n'_1, n'_2, \ldots n'_{\ell})\in \cC_{m_1+m_2}$. 

Let $(N_-^{(n)}, N_+^{(n)}):= (\|C_-^{(n)}\|, \|C_+^{(n)}\|)$, where we write $\|\beta\|=n$ if $\beta\in\mathcal{C}_n$, $n\geq 0$.
Then $(N_-^{(n)}, N_0^{(n)}, N_+^{(n)})$ can be described as a P\'olya urn model with three colours. 
More precisely, when the current numbers of balls of the three colours are $(n_-, n_0, n_+)$, 
we next add a ball whose colour is chosen according to probabilities proportional to $n_-\!+ \theta_1$, $n_0 -\alpha$ and $n_+\!+ \theta_2$.   
Starting from the initial state $(0,1,0)$, we get $(N_-^{(n)}, N_0^{(n)}, N_+^{(n)})$ after adding $n\!-\!1$ balls.
In other words, the vector $(N_-^{(n)}, N_0^{(n)}\!-\!1, N_+^{(n)})$ has \emph{Dirichlet-multinomial distribution with
	parameters $n\!-\!1$ and $(\theta_1, 1\!-\!\alpha, \theta_2)$}; i.e.\@ for $n_-, n_0, n_+\!\in\! \bN_0$ with $n_0\ne 0$ and $n_-\!+n_0+n_+\! = n$
\begin{eqnarray}
	p_n(n_-,n_0,n_+)&:=& \bP\left( (N_-^{(n)}, N_0^{(n)}, N_+^{(n)}) = (n_-, n_0, n_+)\right) \notag \\
	&=& \frac{\Gamma(1-\alpha + \theta_1 +\theta_2) }{\Gamma (1-\alpha)\Gamma (\theta_1) \Gamma(\theta_2)}
	\frac{(n-1)!\Gamma(n_0-\alpha)\Gamma(n_- + \theta_1)\Gamma(n_+ + \theta_2)}{\Gamma(n-\alpha + \theta_1 +\theta_2)(n_0-1)! n_-! n_+!}. \label{eq:dmn}
\end{eqnarray}
Furthermore, conditionally given $(N_-^{(n)},N_0^{(n)},N_+^{(n)})$, the compositions $C_-^{(n)}$ and $C_+^{(n)}$  are independent with distribution
$\mathtt{oCRP}^{(\alpha)}_{N_-^{(n)}}(\theta_1, \alpha)$ 
and $\mathtt{oCRP}^{(\alpha)}_{N_+^{(n)}}(\alpha,\theta_2)$ respectively, for which there is
an explicit description in Lemma~\ref{lem:rcs}, up to an elementary left-right reversal. 
While the statement of the following proposition is for the 
 $\mathrm{oCRP}^{(\alpha)}(\theta_1,\theta_2)$ as originally defined, we will use this decomposition as an auxiliary process in its proof.

\begin{proposition}\label{prop:down}
	Let $\theta_1,\theta_2\ge 0$ and $(C(n), n\ge 1)$ an $\mathrm{oCRP}^{(\alpha)}(\theta_1,\theta_2)$. 
	Then for every $(n_1, \ldots, n_k)\in \cC$ with $n = n_1+\cdots+n_k$, we have
	\[
	\bP\big(C(n)\! =\!  (n_1, \ldots, n_k)\big)=\sum_{i=1}^k\! \bigg(\! 
	p_n \big(N_{1:i-1}, n_{i}, N_{i+1:k}\big)    \prod_{j=1}^{i-1} r_{\theta_1}\!\big(N_{1:j}, n_j\big)  
	\!\!\prod_{j=i+1}^{k}\!\! r_{\theta_2}\!\big(N_{j:k},n_j\big)\!
	\bigg)\!, 
	\]
	where $N_{i:j} =n_i+\cdots+n_j$, $p_n$ is given by \eqref{eq:dmn} and $r_{\theta_1}$, $r_{\theta_2}$ are as in Lemma \ref{lem:rcs}.  
	Furthermore, $(C(n), n\ge 1)$ is a composition structure in the sense that it is weakly sampling consistent. 
\end{proposition}

\begin{proof}
	The claimed expression for $\bP(C(n)=(n_1,\ldots,n_k))$ is an immediate consequence of the decomposition explained above. 
          Let us use shorthand notation $q_n(n_1,\ldots,n_k)$ for this probability function, and in this expression, also abbreviate the 
          $\mathtt{oCRP}^{(\alpha)}_{N_{1:i-1}}(\theta_1, \alpha)$ and $\mathtt{oCRP}^{(\alpha)}_{N_{i+1:k}}(\alpha,\theta_2)$ probabilities captured in the
          two products by $r_-(n_1,\ldots,n_{i-1})$ and $r_+(n_{i+1},\ldots,n_k)$, respectively, so that
          \begin{equation}\label{eqn:distn}
          q_n(n_1,\ldots,n_k)=\sum_{i=1}^k  
	p_n \big(N_{1:i-1}, n_{i}, N_{i+1:k}\big)    r_-(n_1,\ldots,n_{i-1})r_+(n_{i+1},\ldots,n_k).
	\end{equation}
         To prove the weak sampling consistency, we have to verify
          \begin{equation}\label{eqn:qn}
             \begin{split}
             q_n(n_1,\ldots,n_k)=&\sum_{\ell=1}^kq_{n+1}(n_1,\ldots,n_{\ell-1},n_\ell+1,n_{\ell+1},\ldots,n_k)\frac{n_\ell+1}{n+1}\\[-0.4cm]
					&\qquad+\sum_{\ell=0}^kq_{n+1}(n_1,\ldots,n_\ell,1,n_{\ell+1},\ldots,n_k)\frac{1}{n+1}.
             \end{split}
          \end{equation}
          First consider the decomposition 
	$C(n+1)=C_-^{(n+1)}\concat(N_0^{(n+1)})\concat C_+^{(n+1)}$. 
	By removing one customer uniformly at random, 
	we obtain in the obvious way a triple $(\widetilde{C}_-^{(n)},\widetilde{N}_0^{(n)},\widetilde{C}_+^{(n)})$, with 
	the exception for the case when $N_0^{(n+1)}=1$ and this customer is removed by the down-step:  
	in the latter situation, to make sure that $\widetilde{N}_0^{(n)}$ is strictly positive, we choose the new marked table to be the nearest to the left with probability proportional to $\|C_-^{(n+1)}\|$, 
	and the nearest to the right with the remaining probability, proportional to $\|C_+^{(n+1)}\|$, and 
	we further decompose according to this new middle table to define $\big(\widetilde{C}_-^{(n)},\widetilde{N}_0^{(n)},\widetilde{C}_+^{(n)}\big)$.  

	Therefore, for $n_-, n_0, n_+\in \bN_0$ with $n_0\ne 0$ and $n_-+n_0+n_+ = n$,
	the  probability of the event  $\big\{\big(\|\widetilde{C}_-^{(n)}\|,\widetilde{N}_0^{(n)},\|\widetilde{C}_+^{(n)}\|\big)= (n_-, n_0, n_+) \big\}$ is 
	\begin{align}
		&\hspace{-0.4cm}p_{n+1}(n_-\!\!+\!1,n_0,n_+)\frac{n_-\!+\!1}{n\!\!+\!1} + p_{n\!+\!1}(n_-,n_0,n_+\!+\!1)\frac{n_+\!+\!1}{n\!+\!1}
		+ p_{n\!+\!1}(n_-,n_0\!+\!1,n_+)\frac{n_0\!+\!1}{n\!+\!1}\hspace{-0.4cm}\nonumber\\
		&+p_{n+1}(n_-\!\!+\!n_0,1,n_+) \frac{n_-\!\!+\!n_0}{n(n\!+\!1)}   r_{\theta_1}\!(n_-\!+\!n_0, n_0)\label{eqn:pn}\\
		&+p_{n+1}(n_-,1, n_0\!+\!n_+) \frac{n_0\!+\!n_+}{n(n\!+\!1)} r_{\theta_2}\!(n_0\!+\!n_+, n_0), \nonumber
	\end{align}
	where the meaning of each term should be clear. We will refer to the five terms as 1--5.
  Using \eqref{eq:dmn}  together with the expressions for $r_{\theta_1}$ and $r_{\theta_2}$ from Lemma \ref{lem:rcs}, we rewrite the sum as 
	\begin{align*}
	&	 p_n(n_-,n_0,n_+)\frac{n }{(n+1)(n-\alpha+\theta_1+\theta_2)}\times  \\
	&	 \left( n_- +\theta_1+n_+ +\theta_2 +(n_0-\alpha)\frac{n_0+1}{n_0} +
	\frac{1}{n_0}\frac{n_-\alpha+n_0\theta_1}{n}  +
	\frac{1}{n_0}\frac{n_+\alpha+n_0\theta_2}{n}  
	\right). 
	\end{align*}
	Straightforward calculation shows that this is equal to $p_n(n_-,n_0,n_+)$.

          Next, note that the weak sampling consistency of $\mathtt{oCRP}^{(\alpha)}_{N_{1:i-1}}(\theta_1, \alpha)$ and 
          $\mathtt{oCRP}^{(\alpha)}_{N_{i+1:k}}(\alpha,\theta_2)$ given by Lemma \ref{lem:rcs} can be expressed as two sums, which we will refer to as A and B
          \begin{equation}\label{eqn:rminus}
             \begin{split} \quad r_-(n_{1},\ldots,n_{i-1})=&\sum_{\ell=1}^{i-1} r_-(n_1,\ldots,n_{\ell-1},n_\ell+1,n_{\ell+1},\ldots,n_{i-1})\frac{n_\ell+1}{N_{1:i-1}+1}\ \ \ \,\quad\mbox{(A)}\\[-0.4cm]
                                                                                                         &\qquad+\sum_{\ell=0}^{i-1}r_-(n_{1},\ldots,n_\ell,1,n_{\ell+1},\ldots,n_{i-1})\frac{1}{N_{1:i-1}+1}\, \ \ \ \,\quad\mbox{(B)}
             \end{split}
          \end{equation}
          and another two sums, which we will refer to as C and D
          \begin{equation}\label{eqn:rplus}
             \begin{split} \ \ \ \  r_+(n_{i+1},\ldots,n_k)=&\sum_{\ell=i+1}^k r_+(n_{i+1},\ldots,n_{\ell-1},n_\ell+1,n_{\ell+1},\ldots,n_k)\frac{n_\ell+1}{N_{i+1:k}+1}\ \ \ \ \,\mbox{(C)}\\[-0.4cm]
                                                                                                         &\qquad\ \ \ +\sum_{\ell=i}^kr_+(n_{i+1},\ldots,n_\ell,1,n_{\ell+1},\ldots,n_k)\frac{1}{N_{i+1:k}+1}\ \ \ \ \,\mbox{(D)}
             \end{split}
          \end{equation}
         Finally, consider \eqref{eqn:qn}. We use \eqref{eqn:distn} on the LHS and replace $p_n(N_{1:i-1},n_i,N_{i+1:k})$ by  
         \eqref{eqn:pn}, with $n_-=N_{1:i-1}$, $n_0=n_i$ and $n_+=N_{i+1:k}$. For each $1\le i\le k$, this is a sum over five terms 1$_i-$5$_i$. To the $r_-(n_1,\ldots,n_{i-1})$ in the first of 
         these, with a numerator $N_{1:i-1}+1$, we apply \eqref{eqn:rminus}, and to the $r_+(n_{i+1},\ldots,n_k)$ in the second of these, with a numerator $N_{i+1:k}+1$, we apply 
         \eqref{eqn:rplus}. This yields two double sums over $1\!\le\! i\!\le\! k$ and subsets of $0\!\le\!\ell\!\le\! k$ whose terms we denote by 
         $1^{\mathrm{(A)}}_{i\ell}$, $1\le\ell<i\le k$, $1^{\mathrm{(B)}}_{i\ell}$, $0\le\ell<i\le k$, $2^{\mathrm{(C)}}_{i\ell}$, $1\le i<\ell\le k$, $2^{\mathrm{(D)}}_{i\ell}$, $1\le i\le\ell\le k$, and 
         we also have $3_i$, $4_i$ and $5_i$, $1\le i\le k$, which we can also write as $3_\ell$, $4_\ell$ and $5_\ell$, $1\le\ell\le k$.
      
         On the RHS of \eqref{eqn:qn}, we apply \eqref{eqn:distn} to $q_{n+1}(n_1,\ldots,n_{\ell-1},n_\ell+1,n_{\ell+1},\ldots,n_k)\frac{n_{\ell}}{n+1}$ in the first sum and identify this as 
         $\sum_{i=1}^{\ell-1} 2^{\mathrm{(C)}}_{i\ell} + 3_{\ell} + \sum_{i=\ell+1}^{k} 1^{\mathrm{(A)}}_{i\ell}$. 
          Similarly, we apply \eqref{eqn:distn} to 
         $q_{n+1}(n_1,\ldots,n_\ell,1,n_{\ell+1},\ldots,n_k)\frac{1}{n+1}$   in the second sum and identify this as 
         $\sum_{i=1}^{\ell} 2^{\mathrm{(D)}}_{i\ell} + 4_{\ell}\ind_{\{\ell > 0\}} +5_{\ell+1}\ind_{\{\ell < k\}}  + \sum_{i=\ell+1}^{k} 1^{\mathrm{(B)}}_{i\ell}$. 
         To justify this identity, we observe that, for $1\le\ell\le k-1$, 
         $4_{\ell}+5_{\ell+1}=p_{n+1}(N_{1:\ell},1,N_{\ell+1:k})r_-(n_1,\ldots,n_\ell)r_+(n_{\ell+1},\ldots,n_k)$, since  $r_{\theta_1}(N_{1:\ell},n_\ell)r_-(n_1,\ldots,n_{\ell-1})=r_-(n_1,\ldots,n_\ell)$, $r_{\theta_2}(N_{\ell+1: k},n_{\ell+1})r_+(n_{\ell+2},\ldots,n_k)$ $=r_+(n_{\ell+1},\ldots,n_k)$ and $N_{1:\ell}+N_{\ell+1:k}=n$; the cases $\ell=0$ and $\ell=k$ are similarly covered by $5_1$ and $4_k$, respectively. This completes the proof of \eqref{eqn:qn} and hence of this proposition.
\end{proof} 

\begin{remark} The last part of the proof, after identifying \eqref{eqn:pn} with $p_n(n_-,n_0,n_+)$, can be seen as a formalised version of the following conceptual sampling consistency argument.
	
	The triple description \eqref{eqn:distn} yields that, conditionally on $\|C_-^{(n+1)}\|=n_-$, $C_-^{(n+1)}$ has distribution $\mathtt{oCRP}^{(\alpha)}_{n_-}(\theta_1,\alpha)$. 
	Conditionally on $(\|\widetilde{C}_-^{(n)}\|,\widetilde{N}_0^{(n)},\|\widetilde{C}_+^{(n)}\|)= (n_-, n_0, n_+)$, we still have 
	that $\widetilde{C}_-^{(n)}\sim \mathtt{oCRP}^{(\alpha)}_{n_-}(\theta_1, \alpha)$. 
	This can be checked by looking at each situation: 
	in the down-step, if the customer is removed from the marked table (with size $\ge 2$) or the right part, then $\widetilde{C}_-^{(n)} = C_-^{(n+1)}$, 
	which has distribution $\mathtt{oCRP}^{(\alpha)}_{n_-}(\theta_1,\alpha)$;  
	if the customer is removed from the left part, then this is a consequence of the weak sampling consistency of $C_-^{(n+1)}$ given by Lemma~\ref{lem:rcs}; 
	if the marked table has one customer and she is removed, then 
	the claim holds because Lemma~\ref{lem:rcs} yields that $C_-^{(n+1)}$ is regenerative.  
	Similarly, $\widetilde{C}_+^{(n)}\sim \mathtt{oCRP}^{(\alpha)}_{n_+}(\theta_2)$. 
	Summarising, $(\widetilde{C}_-^{(n)},\widetilde{N}_0^{(n)},\widetilde{C}_+^{(n)})$ has the same distribution 
	as $(C_-^{(n)},N_0^{(n)},C_+^{(n)})$. 
\end{remark}

\subsection{The three-parameter family $\mathtt{PDIP}^{(\alpha)}(\theta_1,\theta_2)$}\label{sec:ocrp-pdip}

Our next aim is to study the asymptotics of an $\mathtt{oCRP}^{(\alpha)}_{n}(\theta_1,\theta_2)$, as $n\rightarrow\infty$. 
To this end, recall the space $(\mathcal{I}_H,d_H)$ of interval partitions introduced in the introduction. 
For an interval partition $\beta\in\mathcal{I}_H$, we refer
to the intervals $U\in\beta$ as \em blocks \em and to their lengths ${\rm Leb}(U)$ as their \em masses\em. We similarly refer to $\|\beta\|:=\sum_{U\in\beta}{\rm Leb}(U)$ as the 
\em total mass \em of $\beta$. 
For $c>0$ and $\beta\in \cI_H$, define a \emph{scaling map} by
\[
c \beta:= \{(ca,cb)\colon (a,b)\in \beta\}.
\] 
The \emph{left-right reversal} of $\beta$ is 
\begin{equation}\label{eq:reversal}
	{\rm rev}(\beta):=\{ (\|\beta\|-b, \|\beta\|-a):~ (a,b)\in\beta\}\in \cI_H.
\end{equation}
We also introduce the \emph{concatenation} of a family of interval partitions $(\beta_a)_{a\in \mathcal{A}}$, indexed by a totally ordered set $(\mathcal{A}, \preceq)$: 
\[
\Concat_{a\in \mathcal{A}} \beta_a 
:= \left\{ 
\big(x+ S_{\beta}(a-), y + S_{\beta}(a-)\big) \colon a\in \mathcal{A}, (x,y)\in \beta_{a}
\right\}, ~\text{where}~ S_{\beta}(a-):= \sum_{b\prec a} \|\beta_b\|.
\]
When $\mathcal{A} = \{1,2\}$, we denote this by $\beta_1 \concat \beta_2$.
Then each composition $(n_1,n_2,\ldots,n_k)\in \cC$ is identified with the interval partition 
$\Concat_{i=1}^k \{(0, n_i)\} \in \cI_H$. 
We therefore shall view $\cC$ as a subset of $\cI_H$. 
Let $\cI_{H,1}\subset \cI_{H}$ be the space of partitions of the unit interval $[0,1]$. 

\begin{proposition}\label{prop:crp-pdip}
	For $\alpha\in (0,1)$, $\theta_1,\theta_2\ge 0$, let $(C(n),n\ge 1)\sim \mathrm{oCRP}^{(\alpha)}(\theta_1,\theta_2)$.
	Then $\frac{1}{n} C(n)$ converges a.s.\@ to a random interval partition $\bar{\gamma}$ on $\cI_{H,1}$, under the metric $d_{H}$, as $n\to \infty$. 
\end{proposition}

When $\theta_1=\alpha$, this is known from \cite[Proposition~6]{PitmWink09} and the limit $\bar{\gamma}$ is called a \emph{regenerative $(\alpha,\theta_2)$ interval partition} in \cite{GnedPitm05,PitmWink09}. 

\begin{definition}[$\mathtt{PDIP}^{(\alpha)}( \theta_1, \theta_2)$]\label{defn:pdip} 
	The limiting random interval partition $\bar{\gamma}$ in Proposition~\ref{prop:crp-pdip} is called a \emph{Poisson--Dirichlet$(\alpha,\theta_1,\theta_2)$ interval partition}, whose law on $\cI_{H,1}$ is denoted by $\mathtt{PDIP}^{(\alpha)}( \theta_1,\theta_2)$. 	 
	When $\theta_2=\alpha$, we write $\mathtt{PDIP}^{(\alpha)}(\theta_1):= \mathtt{PDIP}^{(\alpha)}( \theta_1,\alpha)$. 
\end{definition}

The case $\theta_1=\theta_2 =0$ is degenerate with $\mathtt{PDIP}^{(\alpha)}( 0, 0) = \delta_{\{(0,1)\}}$.
The family of $\mathtt{PDIP}^{(\alpha)}( \theta_1, \theta_2)$ notably extends the subfamilies of \cite{GnedPitm05,PitmWink09,ShiWinkel-1}, whose ranked sequence of interval lengths in the Kingman simplex
\[
\nabla_{\infty} := \bigg\{ (x_1, x_2, \ldots) \colon x_1\ge x_2\ge \cdots \ge 0,\, \sum_{i\ge 1} x_i = 1 \bigg\}
\]
are members of the two-parameter family ${\tt PD}^{(\alpha)}(\theta)$, $\alpha\in(0,1)$, $\theta\ge 0$ of \emph{Poisson--Dirichlet distributions}.  Here, we include new cases of interval partitions, for which $\theta\in(-\alpha,0)$, completing the usual range of the two-parameter family of ${\tt PD}^{(\alpha)}(\theta)$ of  \cite[Definition~3.3]{CSP}.

\begin{proposition}\label{prop:pdip-pd}
	Let $\theta_1 ,\theta_2 \ge 0$ with $\theta_1+\theta_2 >0$. 
	The ranked interval lengths of a $\mathtt{PDIP}^{(\alpha)}( \theta_1, \theta_2)$ have $\pd^{(\alpha)}(\theta)$ distribution on $\nabla_{\infty}$ with $\theta:= \theta_1 + \theta_2 - \alpha>-\alpha$. 
\end{proposition}

\begin{proof}
	Let $C^{(n)} \sim \mathtt{oCRP}^{(\alpha)}_{n}(\theta_1, \theta_2)$,  
	then it follows immediately from 
	its construction that 
	$C^{(n)}$ ranked in decreasing order is an unordered $(\alpha, \theta)$-Chinese restaurant process with $\theta:= \theta_1 + \theta_2 -\alpha >-\alpha$.  
	As a consequence of Proposition~\ref{prop:crp-pdip},  the ranked interval lengths of a $\mathtt{PDIP}^{(\alpha)}( \theta_1, \theta_2)$ have the same distribution as the limit of $(\alpha, \theta)$-Chinese restaurant processes, which is 
	known to be $\pd^{(\alpha)}(\theta)$. 
\end{proof}

To prove Proposition~\ref{prop:crp-pdip}, we extend the parameters of Dirichlet distributions to 
every $\alpha_1, \ldots ,\alpha_m \ge 0$ with $m\ge 1$:  say $\alpha_{i_1},\ldots, \alpha_{i_k} >0$
and $\alpha_j =0$ for any other $j\le m$, let $(B_{i_1},\ldots ,B_{i_k})\sim \mathtt{Dir} (\alpha_{i_1}, \ldots, \alpha_{i_k})$ and $B_j:=0$ for all other $j$. 
Then we define $\mathtt{Dir} (\alpha_1, \ldots, \alpha_m)$ to be the law of $(B_1,\ldots,B_m)$. 
By convention $\mathtt{Dir} (\alpha)= \delta_1$ for any $\alpha\ge 0$.

\begin{proof}[Proof of Proposition~\ref{prop:crp-pdip}]
	We use the triple-description of $C(n)$ in the 
	proof of Proposition~\ref{prop:down}. Consider independent $(R_1(n), n\ge 1)\sim \mathrm{oCRP}^{(\alpha)}(\theta_1,\alpha)$,  $(R_2(n), n\ge 1)\sim \mathrm{oCRP}^{(\alpha)}(\alpha,\theta_2)$, and a P\'olya urn model with three colours $(N_1^{(n)}, N_0^{(n)}, N_2^{(n)})$, $n\ge 1$. 
	Then we can write $C(n) = R_1(N_1^{(n)}) \star \{(0,N_0^{(n)} )\}\star  R_2(N_2^{(n)})$. 
	
	The asymptotics of a P\'olya urn yield that $\frac{1}{n} (N_1^{(n)}, N_0^{(n)}, N_2^{(n)})$ converges a.s.\@ to some $(B_1, B_0, B_2)\sim \mathtt{Dir} (\theta_1, 1-\alpha,\theta_2)$. 
	By \cite[Proposition~6]{PitmWink09}, there exist independent $\bar{\gamma}_1\sim \pdip^{(\alpha)} (\theta_1)$ and $\bar{\gamma}_2\sim \pdip^{(\alpha)} (\theta_2)$, 
	such that  $\frac{1}{n}  R_1 (n) \to \bar{\gamma}_1$
	and  $\frac{1}{n}  R_2(n) \to \mathrm{rev}(\bar{\gamma}_2)$ a.s.\@ as $n\to\infty$. 
	Therefore,	$\frac{1}{n}  C(n)$ converges a.s.\@ to 
	$(B_1 \bar{\gamma}_1) \concat \{(0,B_0)\}\concat  (B_2\mathrm{rev}(\bar{\gamma}_2))$. 
\end{proof}

By the proof of Proposition~\ref{prop:crp-pdip}, we immediately have the following decomposition. 
\begin{corollary}\label{cor:pdip} 
	Let $\theta_1, \theta_2 \ge 0$. Consider   
	$(B_1, B_0, B_2)\sim \mathtt{Dir} (\theta_1, 1\!-\!\alpha,\theta_2)$, $\bar{\gamma}_1\sim \pdip^{(\alpha)} (\theta_1)$, and $\bar{\gamma}_2\sim \pdip^{(\alpha)} (\theta_2)$, independent of each other. 
	Let $\bar{\gamma}= B_1  \bar{\gamma}_1 \concat \{(0, B_0)\} \concat \mathrm{rev} (B_2  \bar{\gamma}_2)$. 
	Then $\bar{\gamma}\sim \mathtt{PDIP}^{(\alpha)}(\theta_1, \theta_2)$.
\end{corollary}

With independent $B\sim \mathtt{Beta}( 1\!-\!\alpha\!+\!\theta_1, \theta_2)$, $\bar\gamma\sim \mathtt{PDIP}^{(\alpha)}(\theta_1, 0)$, and $\bar\beta\sim \mathtt{PDIP}^{(\alpha)}(\alpha,\theta_2)$, it follows readily from Corollary~\ref{cor:pdip} that 
\begin{equation*}\label{eq:pdip-decomp-bis}
	B\bar{\gamma} \concat (1-B) \bar{\beta} \sim \mathtt{PDIP}^{(\alpha)}(\theta_1,\theta_2). 
\end{equation*} 
This extends \cite[Corollary 8]{PitmWink09} to the three-parameter case. 
When $\theta_1\ge \alpha$, we also have a different decomposition as follows. 

\begin{corollary}\label{cor:pdipdec}
	Suppose that $\theta_1> \alpha$. 
	With independent $B'\sim \mathtt{Beta}( \theta_1\!-\!\alpha, \theta_2)$, $\bar\gamma\sim \mathtt{PDIP}^{(\alpha)}(\theta_1, 0)$, and $\bar\beta\sim \mathtt{PDIP}^{(\alpha)}(\alpha,\theta_2)$, we have 
	\begin{equation}\label{eq:pdip-decomp}
		B'\bar{\gamma} \concat (1-B') \bar{\beta} \sim \mathtt{PDIP}^{(\alpha)}(\theta_1,\theta_2), 
	\end{equation}
	and $\bar{\gamma}\overset{d}{=}V^\prime\bar{\gamma}_1\concat\{(0,1-V^\prime)\}$ for independent
	$V^\prime\sim{\tt Beta}(\theta_1,1-\alpha)$ and $\bar{\gamma}_1\sim{\tt PDIP}^{(\alpha)}(\theta_1)$.
\end{corollary}

\begin{proof}
	Consider an $\mathrm{oCRP}^{(\alpha)}(\theta_1,\theta_2)$. 
	For the initial table, we colour it in red with probability $(\theta_1-\alpha)/(\theta_1-\alpha+\theta_2)$. If it is not coloured in red, then each time a new table arrives at the very left of the composition, we flip an unfair coin with success probability $1-\alpha/\theta_1$ and colour this new table in red at the first success.  
	In this way, we separate the composition at every step into two parts: the tables to the left of the red table (with the red table included), and everything to the right of the red table. 
	It is easy to see that the sizes of the two parts follow a P\'olya urn such that the asymptotic proportions follow  $\mathtt{Dir} (\theta_1-\alpha,\theta_2)$. Moreover, conditionally on the sizes of the two parts, they are independent 
	$\mathrm{oCRP}^{(\alpha)}(\theta_1,0)$ and $\mathrm{oCRP}^{(\alpha)}(\alpha,\theta_2)$ respectively. 
	Now the claim follows from Proposition~\ref{prop:crp-pdip} and Corollary~\ref{cor:pdip}.
\end{proof}

An ordered version \cite{Pitman97} of \emph{Kingman's paintbox processes} is described as follows.  Let $\gamma\in \cI_{H,1}$ and $(Z_i, i\ge 1)$ be i.i.d.\@ uniform random variables on $[0,1]$. 
Then customers $i$ and $j$ sit at the same table, if and only if $Z_i$ and $Z_j$ fall in the same block of $\gamma$. Moreover, the tables are ordered by their corresponding intervals. 
For any $n\ge 1$, the first $n$ variables $(Z_i, i\in [n])$ give rise to a \emph{composition of the set $[n]$}, i.e.\ an ordered family of  disjoint subsets of $[n]$:  
\begin{equation}\label{eqn:C*}
	C_{\gamma}^*(n)= \{B_U(n)\colon B_U(n)\!\ne\! \emptyset, U\!\in\! \gamma\}, ~\text{where}~ B_U(n) := \{j\!\le\! n \colon Z_j \!\in\! U\}. 
\end{equation}
Let $P_{n,\gamma}$ be the distribution of the random composition of $n$ induced by $C^*_{\gamma}(n)$. 

The following statement shows that the composition structure induced by an ordered CRP is a mixture of \emph{ordered Kingman paintbox processes}.

\begin{proposition}\label{prop:ocrp-pdip}
	The probability measure $\pdip^{(\alpha)}(\theta_1,\theta_2)$ is the unique probability measure on $\cI_{H}$, such that there is the identity 
	\[
	\mathtt{oCRP}_n^{(\alpha)}(\theta_1,\theta_2) (A) 
	= \int_{\cI_H} P_{n,\gamma} (A) ~ \pdip^{(\alpha)}(\theta_1,\theta_2)(d \gamma), \qquad \forall n\ge 1,\; \forall A \subseteq \cC_n. %
	\]
\end{proposition}
\begin{proof}
	Since an $\mathrm{oCRP}^{(\alpha)}(\theta_1,\theta_2)$ is a composition structure by Proposition~\ref{prop:down} and since renormalised $\mathtt{oCRP}_n^{(\alpha)}(\theta_1,\theta_2)$ converges weakly to $\pdip^{(\alpha)}(\theta_1,\theta_2)$ as $n\to \infty$ by Proposition~\ref{prop:crp-pdip}, 
	the statement follows from \cite[Corollary~12]{Gnedin97}. 
\end{proof}

\begin{remark}
	If we label the customers by $\bN$ in an  $\mathrm{oCRP}^{(\alpha)}(\theta_1,\theta_2)$ defined in  Definition~\ref{defn:oCRP}, then we also naturally obtain a composition of the set $[n]$ when $n$ customers have arrived. However, it does not have the same law as the $C^*_{\bar{\gamma}}(n)$ obtained from the paintbox in \eqref{eqn:C*} with $\bar{\gamma}\sim\pdip^{(\alpha)}(\theta_1,\theta_2)$, though we know from Proposition~\ref{prop:ocrp-pdip} that their induced integer compositions of $n$ have the same law. Indeed, $C^*_{\bar{\gamma}}(n)$ obtained by the paintbox is \emph{exchangeable} \cite{Gnedin97}, but it is easy to check that an $\mathrm{oCRP}^{(\alpha)}(\theta_1,\theta_2)$ with general parameters is not, the only exceptions being for $\theta_1=\theta_2=\alpha$. 
\end{remark}

\subsection{The scaling limits of PCRP}\label{sec:mainresult}
We consider a class of ordered Chinese restaurant processes with departures, parametrised by $\alpha\in (0,1)$ and $\theta_1,\theta_2\ge 0$.  
This model extends the model ${\rm PCRP}^{(\alpha)}(\theta)$, $\alpha\in(0,1)$, $\theta\ge 0$, of \cite{RogeWink20} introduced in the introduction.
In the extended model, when there is at least one customer in the restaurant, arrivals are according to the following rule, illustrated in Figure~\ref{fig:PCRP}:   
\begin{itemize}
	\item for each occupied table, say there are $m\ge 1$ customers, 
	a new customer comes to join this table at rate $m- \alpha$;
	\item at rate $\theta_1$, a new customer enters to start a new table to the left of the leftmost table; 
	\item at rate $\theta_2$,  a new customer begins a new table to the right of the rightmost table; 
	\item between each pair of two neighbouring occupied tables, a new customer enters and begins a new table there at rate $\alpha$. 
\end{itemize}
Furthermore, the chain jumps from the null vector $\emptyset$ to state $(1)$ at rate $\theta:= \theta_1+\theta_2-\alpha$ if $\theta>0$, and $\emptyset$ is an absorbing state if $\theta\le 0$. 

We refer to the arrival of a customer as an \emph{up-step}. 
In addition, each customer leaves at rate $1$ (a \emph{down-step}). 

At every time $t\ge 0$, let $C(t)$ be the vector of customer numbers at occupied tables, 
listed from left to right. 
In this way we have defined a continuous-time Markov chain $(C(t), t\ge 0)$.
This process is referred to as a \emph{Poissonised up-down ordered Chinese restaurant process (PCRP) with parameters 
	$\alpha$, $\theta_1$ and $\theta_2$ and initial state $C(0)\in \cC$}, denoted by $\mathrm{PCRP}^{(\alpha)}_{C(0)}(\theta_1,\theta_2)$. 
\begin{figure}[t]
	\centering
	\includegraphics[width=\linewidth]{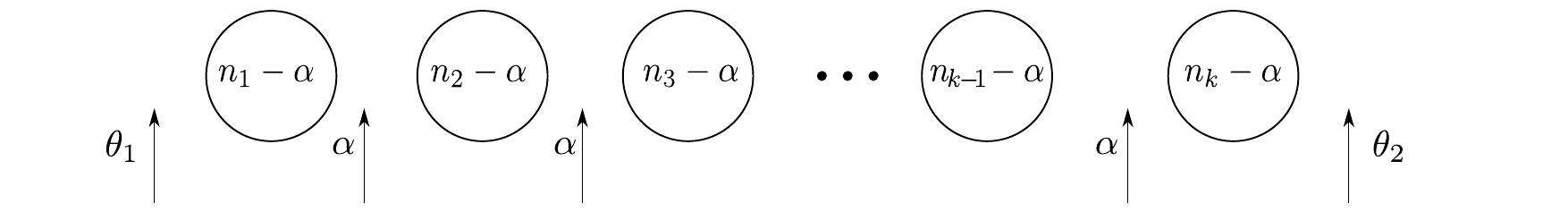}
	\caption{The rates at which new customers arrive in a $\mathrm{PCRP}^{(\alpha)}(\theta_1,\theta_2)$.
	}
	\label{fig:PCRP}
\end{figure}

When $\theta_2= \alpha$, a $\mathrm{PCRP}^{(\alpha)} (\theta_1, \alpha)=\mathrm{PCRP}^{(\alpha)}(\theta_1)$ was studied in \cite{RogeWink20}. Notably, our generalisation includes cases $\theta=\theta_1+\theta_2-\alpha\in[-\alpha,0)$, which did not arise in \cite{James2006,PitmWink09,RogeWink20}. 
Though we focus on the range $\alpha\in (0,1)$ in this paper, our model is clearly well-defined for $\alpha=0$ and it is straightforward to deal with this case; we include a discussion in Section~\ref{sec:zero}.

Our main result is a limit theorem in distribution 
in the space of $\cI_H$-valued c\`adl\`ag functions $\bD(\bR_+, \cI_H)$, endowed with 
the $J_1$-Skorokhod topology (see e.g.\@ \cite{Billingsley} for background).

\begin{theorem}\label{thm:crp-ip} Let $\alpha\in(0,1)$ and $\theta_1,\theta_2\ge 0$. 
	For $n\in \bN$, let  $(C^{(n)}(t),\, t\ge 0)$ be a $\mathrm{PCRP}^{(\alpha)}(\theta_1,\theta_2)$ starting from $C^{(n)}(0)= \gamma^{(n)}$.
	Suppose that  the initial interval partitions $ \frac{1}{n}  \gamma^{(n)}$ converge in distribution to 
	$\gamma\in \cI_H$ as $n\to \infty$, under $d_H$.
	Then there exists an $\cI_H$-valued path-continuous Hunt process $(\beta(t), t\ge 0)$ starting from $\beta(0)= \gamma$, such that 
	\begin{equation}\label{mainthmeq}
		\Big(\frac{1}{n}  C^{(n)}(2 n t),\, t\ge 0 \Big)
		\underset{n\to \infty}{\longrightarrow}  (\beta(t),\, t\ge 0) , \quad \text{in distribution in $\bD(\bR_+,\cI_H)$.} 
	\end{equation}
	Moreover, set $\zeta^{(n)} =\inf\{ t\ge 0\colon C^{(n)}(t) =\emptyset \}$ and $\zeta =\inf\{ t\ge 0\colon \beta(t) =\emptyset \}$ to be the respective first hitting times of $\emptyset$. If $\gamma\neq\emptyset$, then \eqref{mainthmeq} holds jointly with $\zeta^{(n)}/2n\to \zeta$, in distribution.
\end{theorem}
We call the limiting diffusion $(\beta(t),\, t\ge 0)$ on $\cI_H$ an  \emph{$(\alpha,\theta_1,\theta_2)$-self-similar interval partition evolution}, or $\mathrm{SSIP}^{(\alpha)}(\theta_1, \theta_2)$-evolution. These processes are indeed \emph{self-similar with index $1$}, in the language of self-similar Markov processes \cite{Lamperti72}, see also \cite[Chapter~13]{KyprianouBook}:
 if $(\beta(t), t\ge 0)$ is an 
$\mathrm{SSIP}^{(\alpha)}(\theta_1, \theta_2)$-evolution, then $(c \beta(c^{-1} t),\, t\ge 0)$ is an $\mathrm{SSIP}^{(\alpha)}(\theta_1, \theta_2)$-evolution starting from 
$c \beta(0)$, for any $c>0$.
We note that the definition of $\mathrm{PCRP}^{(\alpha)}(\theta_1,\theta_2)$ and the convergence to an 
$\mathrm{SSIP}^{(\alpha)}(\theta_1,\theta_2)$-evolution in Theorem~\ref{thm:crp-ip} are non-trivial even when $\theta_1=\theta_2=0$.

$\mathrm{SSIP}^{(\alpha)}(\theta_1, \theta_2)$-evolutions generalise the interval partition evolutions constructed in the literature \cite{Paper1-2,IPPAT,ShiWinkel-1}. 
From the proof of Theorem~\ref{thm:crp-ip}, we will see that an $\mathrm{SSIP}^{(\alpha)}( \theta_1, \alpha)$-evolution is an \emph{$(\alpha, \theta_1)$-self-similar interval partition evolution} in the sense of \cite{Paper1-2,IPPAT}, which we will also refer to as an $\mathrm{SSIP}^{(\alpha)}(\theta_1)$-evolution. 
In Section~\ref{sec:SSIP}, we will also explain a connection with \cite{ShiWinkel-1}.

While the construction in \cite{Paper1-1,Paper1-2,IPPAT,ShiWinkel-1} is purely in the continuum, via an approach developed in \cite{Paper1-1} by using the marked L\'evy processes, Theorem \ref{thm:crp-ip} is the first scaling limit result with an $\mathrm{SSIP}$-evolution as its limit.  For this smaller class with $\theta_2 =\alpha$, \cite{RogeWink20} provides a study of the family $\mathrm{PCRP}^{(\alpha)}( \theta_1, \alpha)$ and conjectures the existence of diffusion limits, which is thus confirmed by our Theorem~\ref{thm:crp-ip}. 
We conjecture that the convergence in Theorem~\ref{thm:crp-ip} can be extended to the case where $G(\frac{1}{n}  \gamma^{(n)})$ converges in distribution, with respect to the Hausdorff metric, to a compact set of positive Lebesgue measure; then the limiting process is a generalised interval partition evolution in the sense of \cite[Section~4]{IPPAT}. 

When $\theta=\theta_1\!+\!\theta_2\!-\!\alpha\in(0,1)$, a $\mathrm{PCRP}^{(\alpha)}(\theta_1,\theta_2)$ is reflected at $\emptyset$.
When $\theta= \theta_1\!+\!\theta_2\!-\!\alpha \le 0$, a $\mathrm{PCRP}^{(\alpha)}(\theta_1,\theta_2)$ is absorbed at $\emptyset$, and 
if the initial interval partitions $ \frac{1}{n}  \gamma^{(n)}$ converge in distribution to $\emptyset\in \cI_H$ as $n\to \infty$ under $d_H$, 
then the limiting process in Theorem~\ref{thm:crp-ip} is the constant process that stays in $\emptyset$. 
In both cases we refine the discussion and establish the convergence of rescaled PCRP excursions to a non-trivial limit in the following sense. 
\begin{theorem}\label{thm:Theta}
	Let $\alpha\in(0,1)$, $\theta_1, \theta_2\ge 0$ and suppose that $\theta= \theta_1\!+\!\theta_2\!-\!\alpha \in (-\alpha, 1)$. 
	Let $(C(t),t\!\ge\! 0)$ be a $\mathrm{PCRP}^{(\alpha)}(\theta_1,\theta_2)$ starting from state $(1)$ and denote by $\mathrm{P}^{(n)}$ the law of the process
	$\left( C^{(n)}(t):= \frac{1}{n}  C (2 n t\wedge\zeta(C)) ,\, t\ge 0 \right)$, where $\zeta(C):=\inf\{t\ge 0\colon C(t)=\emptyset\}$.  
	Then the following convergence holds vaguely under the Skorokhod topology: 
	\[
	\frac{ \Gamma(1+\theta)}{1-\theta} n^{1-\theta}  \mathrm{P}^{(n)} \underset{n\to \infty}{\longrightarrow} \Theta,  
	\]
	where the limit $\Theta$ is a $\sigma$-finite measure on the space of continuous excursions on $\cI_{H}$. 
\end{theorem}
A description of the limit $\Theta$ is given in Section~\ref{sec:exc}.
We refer to  $\Theta$ as the excursion measure of an $\mathrm{SSIP}^{(\alpha)} (\theta_1, \theta_2)$-evolution, which plays a crucial role in the study of nested interval partition evolutions in Section~\ref{sec:nested2}. 

In Section \ref{sec:crp}, we prove Theorem \ref{thm:crp-ip} in the two-parameter setting $\theta_2=\alpha$, building on \cite{Paper1-1,Paper1-2,IPPAT,RogeWink20}. In Section \ref{sec:pcrp-ssip}, we study the three-parameter setting and complete the proofs of Theorems \ref{thm:crp-ip} and \ref{thm:Theta} in Section~\ref{sec:results}. 
In the following, we state some properties of the limiting interval partition evolutions.

When $\theta_1=\theta_2=0$, the PCRP starting from $(1)$ only ever has a single table, and the convergence to an excursion measure is well-known. While PCRP and SSIP-evolutions are recurrent for $\theta=1$, SSIP-evolutions no longer return to $\emptyset$, as in the transient cases $\theta>1$, so there cannot be excursion measures in these cases. This follows from the connections to squared Bessel processes, to which we turn now.

Specifically, a  squared Bessel process $Z=(Z(t),\,t\ge 0)$  starting from $Z(0)=m\ge 0$ and with ``dimension'' parameter $\delta\in \bR$ is the unique strong solution of the following equation: 
\[
Z(t) = m +\delta t + 2 \int_0^t \sqrt{|Z(s)|} d B(s),
\]
where $(B(t),\,t\ge 0)$ is a standard Brownian motion. 
We refer to \cite{GoinYor03} for general properties of squared Bessel processes. 
Let  $\zeta(Z):= \inf \{ t\ge 0\colon Z(t)=0 \}$ be the first hitting time of zero. To allow $Z$ to re-enter $(0,\infty)$ where possible after hitting $0$, 
we define the \emph{lifetime of $Z$} by
\begin{equation}\label{eq:besq-zeta}
	\overline{\zeta}(Z):= 
	\begin{cases}
		\infty,	& \text{if}~ \delta>0,\\
		\zeta(Z), &\text{if}~ \delta\le 0.	\\
	\end{cases}
\end{equation}
We write $\besq_m(\delta)$ for the law of a squared Bessel process $Z$ with dimension $\delta$ starting from $m$, in the case $\delta\le 0$ absorbed in $\emptyset$ at the end of its (finite) lifetime $\overline{\zeta}(Z)$. When $\delta\le 0$, by our convention $\besq_0(\delta)$ is the law of the constant zero process.

In an $\mathrm{SSIP}^{(\alpha)}(\theta_1, \theta_2)$-evolution, informally speaking, 
each block evolves as ${\tt BESQ}(-2\alpha)$, independently of other blocks \cite{Paper1-1,Paper1-2}. Meanwhile, there is always immigration of rate $2\alpha$ between ``adjacent blocks'', rate $2\theta_1$ on the left \cite{IPPAT} and rate $2\theta_2$ on the right \cite{ShiWinkel-1}. 
Moreover, the total mass process $(\|\beta(t)\|,\,t\ge 0)$ of any $\mathrm{SSIP}^{(\alpha)}(\theta_1,\theta_2)$-evolution  $(\beta(t),\,t\ge 0)$  is ${\tt BESQ}_{\|\beta(0)\|}(2\theta)$ with 
$\theta:=\theta_1+\theta_2-\alpha$. We discuss this more precisely in Section~\ref{sec:SSIP}. 
We refer to $|2\theta|$ 
as the total \em immigration rate \em if $\theta>0$, and as the total \em emigration rate \em if $\theta<0$.

There are \emph{pseudo-stationary} $\mathrm{SSIP}^{(\alpha)}(\theta_1, \theta_2)$-evolutions, that have fluctuating total mass but stationary interval length proportions, in the sense  \cite{Paper1-2} of the following proposition. 
Recall from Definition~\ref{defn:pdip} that the family of $\mathtt{PDIP}^{(\alpha)}( \theta_1, \theta_2)$ gives the limiting block sizes in their left-to-right order of the three-parameter family of composition structures.

\begin{proposition}[Pseudo-stationarity]\label{prop:ps-theta1theta2-nokill}
	For $\alpha\in(0,1)$ and $\theta_1, \theta_2\ge 0$, consider independently $\bar\gamma\sim \mathtt{PDIP}^{(\alpha)}( \theta_1,\theta_2)$ and a $\besq(2 \theta)$-process $(Z(t),\,t\ge 0)$ with any initial distribution and parameter $\theta=\theta_1+\theta_2-\alpha$. 
	Let $(\beta(t),\, t\ge 0)$ be an $\mathrm{SSIP}^{(\alpha)}(\theta_1, \theta_2)$-evolution starting from $\beta(0)=Z(0)   \bar\gamma$. 
	Fix any $t\ge 0$, then $\beta(t)$ has the same distribution as $Z(t) \bar\gamma$. 
\end{proposition}

As in the case $\theta_2=\alpha$ studied in \cite{Paper1-2,IPPAT}, we define an associated family of $\cI_{H,1}$-valued evolutions via time-change and renormalisation (``de-Poissonisation'').
\begin{definition}[De-Poissonisation and $\mathrm{IP}^{(\alpha)} (\theta_1, \theta_2)$-evolutions]
	\label{defn:dePoi}
	Consider $\gamma\in \cI_{H,1}$,
	let $\boldsymbol{\beta}:= (\beta(t),\, t\ge 0)$ be an $\mathrm{SSIP}^{(\alpha)} (\theta_1, \theta_2)$-evolution starting from $\gamma$ and define a time-change function $\tau_{\boldsymbol{\beta}}$ by 	
	\begin{equation}\label{eq:tau-beta}
		\tau_{\boldsymbol{\beta}}(u):= \inf \left\{ t\ge 0\colon \int_0^t \|\beta(s)\|^{-1} d s>u \right\}, \qquad u\ge 0.	
	\end{equation}
	Then the process on $\cI_{H,1}$ obtained from $\boldsymbol{\beta}$ via the following \emph{de-Poissonisation}
	\[
	\overline{\beta}(u):= \big\| \beta(\tau_{\boldsymbol{\beta}}(u)) \big\|^{-1}  \beta(\tau_{\boldsymbol{\beta}}(u)),\qquad u\ge 0,
	\]
	is called a \emph{Poisson--Dirichlet $(\alpha,\theta_1,\theta_2)$-interval partition evolution} starting from $\gamma$, abbreviated as $\mathrm{IP}^{(\alpha)} (\theta_1, \theta_2)$-evolution. 
\end{definition}

\begin{theorem}\label{thm:dP}
	Let $\alpha\!\in\!(0,1)$, $\theta_1, \theta_2\ge 0$. 
	An $\mathrm{IP}^{(\alpha)} (\theta_1, \theta_2)$-evolution is a  Hunt process on 
	$(\cI_{H,1},d_{H})$ with continuous paths. It is continuous in the initial state and has a stationary distribution
	$\mathtt{PDIP}^{(\alpha)}( \theta_1, \theta_2)$. 
\end{theorem}
In the case $\theta_2=\alpha$, Proposition~\ref{prop:ps-theta1theta2-nokill} and Theorem~\ref{thm:dP} have been proved in \cite{Paper1-2,IPPAT}. The proofs of the more general results here will be given in Section~\ref{sec:results}.

For the two-parameter case (with $\theta_2=\alpha$), \cite{RivRiz} obtained the scaling limits of a family of discrete-time up-down ordered Chinese restaurant processes, in which the number of customers remains constant by coupling each $\mathrm{oCRP}^{(\alpha)}( \theta_1, \alpha)$-arrival with a departure. 
It is conjectured that the limits of \cite{RivRiz} are $\mathrm{IP}^{(\alpha)}(\theta_1,\alpha)$-evolutions, and we further conjecture that this extends to the three-parameter setting of  Definition~\ref{defn:dePoi}.

Define $\mathcal{H}$ to be the commutative unital algebra of functions on $\nabla_{\infty}$ generated by 
$q_k(x) = \sum_{i\ge 1} x_i^{k+1}$, $k\ge 1$, and $q_0(x)=1$.   
For every $\alpha\in (0,1)$ and $\theta>-\alpha$, define an operator $\mathcal{B}_{\alpha,\theta}\colon \mathcal{H} \to \mathcal{H}$ by
\[
\mathcal{B}_{\alpha, \theta}:= \sum_{i\ge 1} x_i \frac{\partial^2}{\partial x_i^2} 
-\sum_{i,j\ge 1} x_ix_j\frac{\partial^2}{\partial x_i \partial x_j}
- \sum_{i\ge 1} (\theta x_i +\alpha ) \frac{\partial}{\partial x_i}. 
\]
It has been proved in \cite{Petrov09} that there is a Markov process on $\nabla_{\infty}$ 
whose (pre-)generator on $\mathcal{H}$ is $\mathcal{B}_{\alpha, \theta}$, which shall be referred to as the 
Ethier--Kurtz--Petrov diffusion with parameter $(\alpha, \theta)$, for short ${\tt EKP}(\alpha,\theta)$-diffusion;  
moreover, ${\tt PD}^{(\alpha)}(\theta)$ is the unique invariant probability measure 
for ${\tt EKP}(\alpha, \theta)$. In \cite{Paper1-3}, the following connection will be established. 
\begin{itemize}
	\item Let $\alpha\in(0,1)$, $\theta_1 ,\theta_2 \ge 0$ with $\theta_1+\theta_2 >0$. 
	For an $\mathrm{IP}^{(\alpha)} (\theta_1, \theta_2)$-evolution $(\overline{\beta}(u),\, u\ge 0)$, 
	list the lengths of intervals of $\overline{\beta}(u)$ in decreasing order in a sequence $W(u)\in \nabla_{\infty}$.  
	Then the process $(W(u/2), u\ge 0)$ is an $\ekp(\alpha, \theta)$-diffusion with $\theta:= \theta_1+\theta_2-\alpha>-\alpha$.  
\end{itemize}

\section{Proof of Theorem~\ref{thm:crp-ip} when $\theta_2=\alpha$}\label{sec:crp}
We first recall in Section~\ref{sec:pre} the construction and some basic properties of the two-parameter family of $\mathrm{SSIP}^{(\alpha)}(\theta_1)$-evolutions from \cite{Paper1-1,Paper1-2, IPPAT}, 
and then prove that they are the diffusion limits of the corresponding $\mathrm{PCRP}^{(\alpha)}(\theta_1,\alpha)$, in Sections~\ref{sec:cv-0alpha} and \ref{sec:cv-thetaalpha} for $\theta_1=0$ and for $\theta_1\ge 0$ in general, respectively, thus proving Theorem~\ref{thm:crp-ip} for the case $\theta_2=\alpha$. 
The proofs rely on a representation of $\mathrm{PCRP}^{(\alpha)}(\theta_1,\alpha)$ by Rogers and Winkel \cite{RogeWink20} that we recall in Section~\ref{sec:PCRP} and an in-depth investigation of a positive-integer-valued Markov chain in Section~\ref{sec:ud-chain}. 

\subsection{Preliminaries: $\mathrm{SSIP}^{(\alpha)}(\theta_1)$-evolutions}\label{sec:pre}

In this section, we recall the \emph{scaffolding-and-spindles} construction and some basic properties of an $(\alpha,\theta_1)$ self-similar interval partition evolution, $\mathrm{SSIP}^{(\alpha)}(\theta_1)$. The material is collected from \cite{Paper1-1, Paper1-2, IPPAT}.

Let $\cE$ be the space of non-negative c\`adl\`ag excursions away from zero. Then 
for any $f\in \cE$, we have $\zeta (f):= \inf\{t>0\colon f(t)=0\}=\sup \{ t\ge 0\colon f(t)>0\}$.  
We will present the construction of SSIP-evolutions via the following \emph{skewer} map introduced in \cite{Paper1-1}. 
\begin{definition}[Skewer] \label{def:skewer}
	Let $N= \sum_{i\in I} \delta(t_i,f_i)$ be a point measure on $\bR_+\times \cE$ and $X$ a c\`adl\`ag process such that 
	\[\sum_{ \Delta X(t)> 0} \delta(t, \Delta X(t)) = \sum_{i\in I} \delta(t_i, \zeta(f_i)).\] 
	The \emph{skewer} of the pair $(N,X)$ at level $y$ is (when well-defined) the interval partition
	\begin{equation}
		\skewer(y,N,X) :=     
		\{ (M^y(t-),M^y(t)) \colon M^y(t-)<M^y(t), t\ge 0 \},
	\end{equation}
	where $M^y(t) = \int_{[0,t]\times\cE} f\big( y- X(s-) \big)  N(ds,df)$. 
	Denote the process by 
	\[
	\skewerbar(N,X):= (\skewer(y,N,X), y\ge 0). 
	\]
\end{definition}

Let $\theta\in (-1,1)$. We know from \cite{GoinYor03} that $\besq( 2\theta)$ has an exit boundary at zero. 
Pitman and Yor \cite[Section~3]{PitmYor82} construct a $\sigma$-finite excursion measure $\Lambda^{(2\theta)}_{\mathtt{BESQ}}$ associated with $\besq(2\theta)$ on the space $\cE$, 
such that 
\begin{equation}\label{eq:besq-exc}
	\Lambda_{\tt BESQ}^{(2\theta)}(\zeta>y):=\Lambda^{(2\theta)}_{\mathtt{BESQ}} \left\{f\in\cE\colon \zeta(f)> y\right\} =\frac{2^{\theta-1}}{\Gamma(2\!-\!\theta)} y^{-1+\theta}, \qquad y>0,
\end{equation}
and under $\Lambda^{(2\theta)}_{\mathtt{BESQ}}$, conditional on $\{ \zeta=y \}$ for $0<y<\infty$, 
the excursion is a squared Bessel bridge from $0$ to $0$ of length $y$, see \cite[Section 11.3]{RevuzYor}.    
\cite[Section~3]{PitmYor82} offers several other equivalent descriptions of $\Lambda^{(2\theta)}_{\mathtt{BESQ}}$; see also \cite[Section~2.3]{Paper1-1}.

\begin{figure}
	\centering
	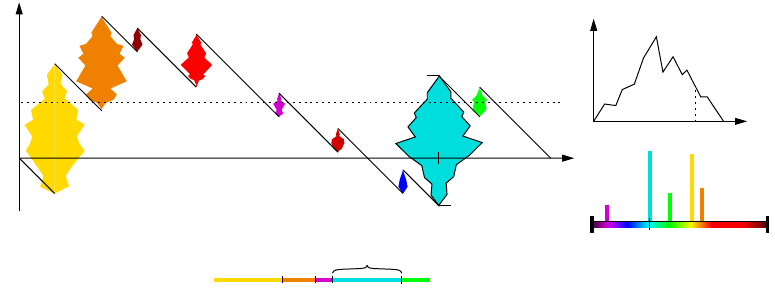
	\caption{A scaffolding with marks (atom size evolutions as spindle-shapes and allelic types from a colour spectrum coded by $[0,1]$) and the skewer and superskewer (see Definition~\ref{def:superskewer}) at level $y$, not to scale.\label{fig:scaf-marks}}
\end{figure}
For $\alpha\in (0,1)$, let $\fN$ be a Poisson random measure on $\bR_+\times \cE$ with intensity $c_{\alpha}\mathrm{Leb} \otimes \Lambda^{(-2\alpha)}_{\mathtt{BESQ}}$, denoted by $\PRM(c_{\alpha}\mathrm{Leb} \otimes \Lambda^{(-2\alpha)}_{\mathtt{BESQ}})$, where
\begin{equation}\label{eq:nu-Lambda}
	c_{\alpha}:=2 \alpha(1\!+\!\alpha)/\Gamma(1\!-\!\alpha). 
\end{equation} 
Each atom of $\fN$, which is an excursion function in $\cE$, shall be referred to as a \emph{spindle}, in view of illustration of $\fN$ as in Figure~\ref{fig:scaf-marks} . 
We pair $\fN$ with a \emph{scaffolding function} $\xi_{\fN}:=(\xi_{\fN}(t), t\ge 0)$ defined by
\begin{equation}\label{eq:scaffolding}
	\xi_{\fN}(t):=  \lim_{z\downarrow 0} \bigg( 
	\int_{[0,t]\times \{g\in \cE\colon \zeta(g)>z\}} \zeta(f)  \fN(ds,df) - \frac{(1+\alpha)t}{(2z)^{\alpha}\Gamma(1-\alpha)\Gamma(1+\alpha)}
	\bigg).
\end{equation}
This is a spectrally positive stable L\'evy process of index $(1+\alpha)$, with L\'evy measure $c_{\alpha}\Lambda^{(-2\alpha)}_{\mathtt{BESQ}} (\zeta \in  d y)$ and Laplace exponent 
$
(2^{-\alpha}q^{1+\alpha} /\Gamma(1+\alpha),  q\ge 0). 
$

For $x>0$, let $\ff\sim \besq_x(- 2\alpha)$, independent of $\fN$.  
Write $\mathtt{Clade}_x(\alpha)$ for the law of a \emph{clade of initial mass $x$}, which is a random point measure on $\bR_+\times \cE$ defined by
\begin{equation*}\label{eq:clade}
	\clade (\ff, \fN):= \delta(0,\ff)  + \fN\,\big|_{\left(0, T_{-\zeta(\ff)} (\xi_\fN) \right]\times \cE}, 
	~\text{where}~ T_{-y} (\xi_\fN):= \inf\{t\ge 0 \colon \xi_{\fN}(t)= -y\}. 
\end{equation*}

\begin{definition}[$\mathrm{SSIP}^{(\alpha)}(0)$-evolution]\label{defn:ip0}
	For $\gamma \in \cI_H$, let $(\fN_U, U\in \gamma)$ be a family of independent clades, with each $\fN_U \sim \mathtt{Clade}_{\mathrm{Leb}(U)} (\alpha)$. 
	An $\mathrm{SSIP}^{(\alpha)}(0)$-evolution starting from $\gamma \in \cI_H$ is a process distributed as $\boldsymbol{\beta}= (\beta(y), y\ge 0)$ defined by 
	\[
	\beta(y):=  \Concat_{U\in \gamma} \skewer (y, \fN_{U}, \xi_{\fN_{U}}), \qquad y\ge  0. 
	\]
\end{definition}


We now turn to the case $\theta_1>0$. 
Let $\fN$ be a $\mathtt{PRM}(c_{\alpha}\mathrm{Leb}\otimes \Lambda^{(-2\alpha)}_{\mathtt{BESQ}})$ and $\fX_{\alpha}= \xi_{\fN}$  its scaffolding. 
Define the \em modified  scaffolding \em process 
\begin{equation}\label{eq:X-alphatheta}
	\fX_{\theta_1}(t) := \fX_{\alpha}(t) + \left(1 - \alpha/\theta_1 \right)\ell(t) \quad \text{where} \quad \ell(t) := -\inf_{u\leq t}\fX_{\alpha}(u) \quad \text{for }t\ge 0.
\end{equation}
For any $y\ge 0$, let 
\[
T^{-y}_{\alpha}:=T_{-y}(\mathbf{X}_\alpha)=\inf\{t\ge 0\colon \fX_{\alpha}(t)= -y \} = \inf\{t\ge 0\colon \ell(t)\ge y \}. 
\]
Notice that 
$\inf_{u\le t} \fX_{\theta_1} (u)= -(\alpha/\theta_1) \ell (t)$, 
then we have the identity 
\begin{equation}\label{eq:Talphatheta}
	T^{-y}_{\theta_1} 
	:=T_{-y}(\fX_{\theta_1})=  \inf\{t\ge 0\colon \fX_{\theta_1}(t)=- y \} 
	= T^{-(\theta_1/\alpha) y}_{\alpha}. 
\end{equation}
For each $j\in \bN$, define an interval-partition-valued process 
\[
\cev{\beta}_j(y) := 
\skewer (y, \fN\big|_{[0,T_{\theta_1}^{-j})}, j+\fX_{\theta_1}\big|_{[0,T_{\theta_1}^{-j})}), \qquad y\in [0,j].
\]
For any $z>0$, the shifted process  
$(z +\fX_{\alpha} (T_{\alpha}^{-z}+t), t\ge 0)$
has the same distribution as $\fX_{\alpha}$, by the strong Markov property of $\fX_{\alpha}$. 
As a consequence, $(-z+\ell(t+T_{\alpha}^{-z}), t\ge 0)$ has the same law as $\ell$. 
Combing this and \eqref{eq:Talphatheta}, we deduce that, 
for any $k\ge j$, the following two pairs have the same law:
\[
\left( \big(L_{T_{\theta_1}^{j-k}} \fN \big)\Big|_{[0, T_{\theta_1}^{-k}-T_{\theta_1}^{j-k})}, 
k+{\big( L_{T_{\theta_1}^{j-k}}\fX_{\theta_1}}\big)\Big|_{[0,T_{\theta_1}^{-k}-T_{\theta_1}^{j-k})} 
\right)
\stackrel{d}{=} 
\left({\fN}\big|_{[0,T_{\theta_1}^{-j})}, j+{\fX_{\theta_1}}\big|_{[0,T_{\theta_1}^{-j})} 
\right),
\]
where $L$ stands for the shift operator and we have also used the Poisson property of $\fN$. 
This leads to 
$(\cev{\beta}_j(y),\, y\in[0,j] )  \stackrel{d}{=} (\cev{\beta}_k(y),\, y\in[0,j] )$. 
Thus, by Kolmogorov's extension theorem, there exists a process $(\cev{\beta}(y), y\ge 0)$ such that 
\begin{equation}\label{eq:backarrow}
	\left(\cev{\beta}(y),\, y\in[0,j]\right)  \stackrel{d}{=} \left(\cev{\beta}_j(y),\, y\in[0,j] \right)\quad\mbox{for every $j\in \bN$.}
\end{equation}

\begin{definition}[$\mathrm{SSIP}^{(\alpha)}(\theta_1)$-evolution]~\label{defn:alphatheta}
	For $\theta_1>0$, let $(\cev{\beta}(y), y\ge 0)$ be as in \eqref{eq:backarrow} and $(\vecc{\beta}(y), y\ge 0)$ 
	an independent $\mathrm{SSIP}^{(\alpha)}(0)$-evolution starting from $\gamma\in \cI_H$. 
	Then $(\beta(y) = \cev{\beta}(y)\concat \vecc{\beta}(y),\, y\ge 0)$
	is called an \emph{$\mathrm{SSIP}^{(\alpha)}(\theta_1)$-evolution} starting from $\gamma$. 
\end{definition}

In \cite{IPPAT}, an $\mathrm{SSIP}^{(\alpha)}(\theta_1)$-evolution is defined in a slightly different way that more explicitly handles the Poisson random measure of excursions of $\fX_\alpha$ above the minimum. Indeed, the passage from $\alpha$ to $\theta_1$ in \cite{IPPAT} is by changing the intensity by a factor of $\theta_1/\alpha$. 
The current correspondence can be easily seen to have the same effect.

\begin{proposition}[{\cite[Proposition~1.3]{IPPAT}}]\label{prop:alphatheta}
	For $\theta_1 \!\ge\! 0$, an $\mathrm{SSIP}^{(\alpha)}(\theta_1)$-evolution  
	is a path-continuous Hunt process and its total mass process  
	is a $\besq(2\theta_1)$. 
\end{proposition}

We refer to \cite{IPPAT} for the transition kernel of an $\mathrm{SSIP}^{(\alpha)}(\theta_1)$-evolution.

\subsection{Poissonised ordered up-down Chinese restaurant processes}\label{sec:PCRP}

For $\theta >-1$, let $Z:= (Z(t), t\ge 0)$ be a continuous-time Markov chain on $\bN_0$, whose non-zero transition rates are
\begin{equation*}
	Q_{i,j} (\theta)=	\begin{cases}  
		i  +\theta,&  i\ge 1, j = i+1; \\
		i,& i\ge 1, j = i-1, \\
		\theta\vee 0,& i=0, j=1.
	\end{cases} 
\end{equation*}
In particular, $0$ is an absorbing state when $\theta\le 0$.  
For $k\in \bN_0$, we define
\begin{equation}
	\label{eq:pi}
	\pi_k(\theta) \colon \text{ the law of the process } Z \text{ starting from }  Z(0)=k. 
\end{equation} 
Let $\zeta(Z):= \inf\{t> 0 \colon Z(t) = 0\}$ be its first hitting time of zero. 

Let $\alpha\in (0,1)$ and $\theta_1,\theta_2\ge 0$.  
Recall from the introduction that a Poissonised ordered up-down Chinese restaurant process (PCRP) with parameters $\alpha$, $\theta_1$ and $\theta_2$, starting from $C\in \cC$, is denoted it by $\mathrm{PCRP}^{(\alpha)}_C (\theta_1, \theta_2)$.

When $\theta_2= \alpha$, a $\mathrm{PCRP}^{(\alpha)} (\theta_1, \alpha)$ is well-studied by Rogers and Winkel \cite{RogeWink20}. 
They develop a representation of a PCRP by using scaffolding and spindles, in a similar way to the construction of an $\mathrm{SSIP}^{(\alpha)} (\theta_1)$-evolution. 
Their approach draws on connections with \emph{splitting trees} and results 
of the latter object developed in \cite{GeigKers97,Lambert10}. 

Let $\mathbf{D}\sim\PRM(\alpha \cdot \mathrm{Leb} \otimes \pi_1(-\alpha))$ and define its scaffolding function by 
\begin{equation}\label{eq:scaffolding-D}
	J_{\mathbf{D}} (t) :=- t + \int_{[0,t] \times \cE}  \zeta(f) \mathbf{D} (ds, df) ,\qquad t\ge 0. 
\end{equation}
Let $Z\sim \pi_m(-\alpha)$ with $m\in \bN$, independent of $\mathbf{D}$. For $y>0$, set $T_{-y} (J_\mathbf{D})= \inf\{t\ge 0 \colon J_{\mathbf{D}}(t)= -y\}$. 
Then a \emph{discrete clade with initial mass $m$} is a random point measure on $\bR_+\times \cE$ defined by
\begin{equation}\label{eq:clade-D}
	\clade^D (Z, \mathbf{D}):= \delta(0,Z)  + \mathbf{D}\mid_{\big(0, T_{-\zeta(Z)} (J_\mathbf{D}) \big]\times \cE},  
\end{equation}
Write $\mathtt{Clade}^D_m(\alpha)$ for the law of $\clade^D (Z, \mathbf{D})$. 

Recall that we view $\cC$ as a subspace of $\cI_{H}$, with each composition $(n_1, \ldots, n_k)\in \cC$ identified with the interval partition  $\{ (s_{i-1},s_{i}), 1\le i\le k \}$, where $s_i=n_1+\cdots+n_i$. 

\begin{lemma} [{\cite[Theorem 1.2]{RogeWink20}}]\label{lem:pcrp0}
	For $\gamma \in \cC$, let $(\mathbf{D}_U, U\!\in\! \gamma)$ be an independent family with each $\mathbf{D}_U \!\sim\! \mathtt{Clade}^D_{\mathrm{Leb}(U)} (\alpha)$. 
	Then the process 
	$
	\big(\Concat_{U\in \gamma}\skewer(y, \mathbf{D}_{U}, J_{\mathbf{D}_{U}}), y\ge 0\big)
	$
	is a $\mathrm{PCRP}^{(\alpha)}_{\gamma}(0, \alpha)$. 
\end{lemma}

To construct a $\mathrm{PCRP}^{(\alpha)} (\theta_1,\alpha)$ with $\theta_1>0$,  define for $t\ge 0$
\begin{equation}\label{eqnstar}
	J_{\theta_1, \mathbf{D}}(t):= J_{\mathbf{D}}(t) + \left(1-\frac{\alpha}{\theta_1}\right) \ell(t), \quad \text{where}~ \ell (t):= -\inf_{u\le t} J_{\mathbf{D}}(u). 
\end{equation}
Then $\inf\{t\!\ge\! 0\colon\! J_{\mathbf{D}}(t) \!=\!-z\}\!=\!\inf\{t\!\ge\! 0\colon\! J_{\theta_1,\mathbf{D}}(t)\!=\!-(\alpha/\theta_1)z\}\!=:\!T^{-(\alpha/\theta_1)z}_{\theta_1}$ for $z\ge 0$. Set
\[
\cev{C}_j(y) := \skewer\left(y, \mathbf{D}\big|_{[0,T_{\theta_1}^{-j})}, j+ J_{\theta_1,\mathbf{D}}\big|_{[0,T_{\theta_1}^{-j})}\right), \qquad y\in [0,j],\quad j\in\bN.
\]
Then, for any $k>j$, we have 

\begin{align*}
	&\left(  \big(L_{T_{\theta_1}^{j\!-\!k}} \mathbf{D} \big)\big|_{[0,T^{-k}_{\theta_1}-T^{j\!-\!k}_{\theta_1})},
	\; k+ \big(L_{T_{\theta_1}^{j\!-\!k}}J_{\theta_1,\mathbf{D}}\big)\big|_{[0, T^{-k}_{\theta_1}-T^{j\!-\!k}_{\theta_1})} \right)\\
& \ed\left(
	\mathbf{D}\mid_{[0,T^{-j}_{\theta_1})},\;
	 j+ J_{\theta_1,\mathbf{D}}\big|_{[0,T^{-j}_{\theta_1})}
	\right).
\end{align*}

As a result, 
$(\cev{C}_k(y),\, y\in [0,j])\ed (\cev{C}_j(y),\, y\in [0,j])$. 
Then by Kolmogorov's extension theorem there exists 
a c\`adl\`ag process $(\cev{C}(y),\, y\ge 0)$
such that 
\begin{equation}\label{eq:discrbackarrow}
	(\cev{C}(y),\, y\in [0,j])\ed (\cev{C}_j(y),\, y\in [0,j]) \quad\text{for all}~ j\in \bN. 
\end{equation}

\begin{theorem}[{\cite[Theorem 2.5]{RogeWink20}}]\label{thm:jccp}
	For $\theta_1 > 0$, let $(\cev{C}(y),\, y\ge 0)$ be the process defined in \eqref{eq:discrbackarrow}.  
	For $\gamma\in \cC$, 
	let $(\vecc{C}(y),\,y\ge 0)$ be a $\mathrm{PCRP}^{(\alpha)} (0,\alpha)$ starting from $\gamma$.  
	Then the $\cC$-valued process 
	$(C(y):= \cev{C}(y)\concat \vecc{C}(y),\,y\ge 0)$ 
	is a $\mathrm{PCRP}^{(\alpha)} (\theta_1,\alpha)$ starting from $\gamma$. 
\end{theorem}

\subsection{Study of the up-down chain on positive integers}\label{sec:ud-chain}
For $\theta>-1$ and $n, k\in \bN$, define a probability measure:  
\begin{equation}\label{eq:pi-n}
	\pi_k^{(n)}(\theta) \text{ is the law of the process } 
	\left( (n^{-1} Z(2 n y),\, y\ge 0 \right), ~\text{where}~ Z \sim \pi_{k} (\theta)\text{ as in } \eqref{eq:pi}. 
\end{equation}
In preparation of proving Theorem~\ref{thm:crp-ip}, we present the following convergence concerning scaffoldings and spindles.  
\begin{proposition}\label{prop:cv-prm}
	For $n\in \bN$, let $\fN^{(n)}$ be a Poisson random measure on $\bR_+\times \cE$ with intensity 
	$ \mathrm{Leb}\otimes ( 2 \alpha n^{1+ \alpha}\cdot\pi_1^{(n)} (-\alpha))$, 
	and define its scaffolding $\xi^{(n)}:= (\xi^{(n)}(t))_{t\ge 0}$, where
	\begin{equation}\label{eq:scaffolding-n}
		\xi^{(n)}(t):= J_{\fN^{(n)}}^{(n)} (t) :=  -   n^{\alpha} t + \int_{[0,t] \times \cE}  \zeta(f)  \fN^{(n)} (ds, df) ,\qquad t\ge 0. 
	\end{equation}
	Write
	$\ell^{(n)}:= \left( \ell^{(n)}(t):= -\inf_{s\in [0,t]} \xi^{(n)}(s) , t\ge 0\right)$.  
	Let $\fN\sim \PRM( c_{\alpha} \cdot\mathrm{Leb} \otimes \Lambda^{(-2\alpha)}_{\mathtt{BESQ}})$, where $\Lambda^{(-2\alpha)}_{\mathtt{BESQ}}$ is the excursion measure associated with $\besq(- 2\alpha)$ 
	introduced in Section~\ref{sec:pre} and $c_{\alpha}=2 \alpha(1\!+\!\alpha)/\Gamma(1\!-\!\alpha)$ as in \eqref{eq:nu-Lambda}. Define its scaffolding  $\xi_{\fN}$ as in \eqref{eq:scaffolding}, 
	and $\ell_{\fN}= (\ell_{\fN}(t)= -\inf_{s\in [0,t]} \xi_{\fN}(s), t\ge 0)$.
	Then the joint distribution of the triple  $(\fN^{(n)}, \xi^{(n)}, \ell^{(n)})$ converges to $(\fN, \xi_{\fN}, \ell_{\fN})$ in distribution, under the product of vague and Skorokhod topologies. 
\end{proposition}

Note that, for $n\in \bN$, we can construct the $\mathbf{N}^{(n)}$ in Proposition~\ref{prop:cv-prm} from a Poisson random measure $\mathbf{D}\sim \PRM(\alpha\mathrm{Leb} \otimes \pi_1 (-\alpha))$, by setting  
\[
\mathbf{N}^{(n)} = 
\sum_{(s,f)\text{ atom of } \mathbf{D}}
\delta \left({\textstyle\frac12}n^{-(1+\alpha)} s,  n^{-1} f (2 n \,\cdot\, ) \right). 
\]
This suggests a direct relation between $\fN^{(n)}$ and a rescaled PCRP, which will be specified in \eqref{relation1}--\eqref{relation2}.

The up-down chain defined in \eqref{eq:pi-n} plays a central role in the proof of Proposition~\ref{prop:cv-prm}. Let us first record a convergence result obtained in \cite[Theorem 1.3--1.4]{RogeWink20}. 
Similar convergence in a general context of discrete-time Markov chains converging to positive self-similar Markov processes has been established in \cite{BertKort16}.

\begin{lemma}[{\cite[Theorem 1.3--1.4]{RogeWink20}}]\label{lem:ud} 
	Fix $a>0$ and $\theta>-1$. For every $n\in \bN$, let $Z^{(n)} \sim \pi_{\lfloor n a\rfloor}^{(n)} (\theta)$. 
	Then the following convergence holds in the space $\bD(\bR_+, \bR_+)$ of c\`adl\`ag functions endowed with the Skorokhod topology: 
	\[
	Z^{(n)} \underset{n\to \infty}{\longrightarrow} Z\sim \besq_a (2 \theta)\quad \text{in distribution}.
	\]
	Moreover, if $\theta\in (-1, 0]$, then the convergence holds jointly with the convergence of first hitting times of 0.
\end{lemma}

For our purposes, we study this up-down chain in more depth and obtain the following two 
convergence results. Their proofs are postponed to Appendix~\ref{sec:ud-proof}.

\begin{lemma}\label{lem:ud-bis} 
	In Lemma~\ref{lem:ud}, the joint convergence of first hitting time of 0 also holds when $\theta\in (0,1)$.
\end{lemma}

Recall that for $\theta\ge 1$, the first hitting time of 0 by ${\tt BESQ}(2\theta)$ is infinite.

\begin{proposition}\label{prop:vague}
	Let $\theta\in (-1,1)$ and $Z\sim\pi_1(\theta)$. Denote by $\widetilde{\pi}_1^{(n)}(\theta)$ the distribution of 
	$\big(\frac{1}{n}Z(2nt\wedge\zeta(Z)),\,t\ge 0\big)$. 
	Then the following convergence holds vaguely
	\[
	\frac{\Gamma(1\!+\!\theta)}{1\!-\!\theta} n^{1-\theta} \cdot \widetilde{\pi}_1^{(n)}(\theta) \underset{n\to \infty}{\longrightarrow}  \Lambda^{(2\theta)}_{\mathtt{BESQ}}
	\]
	on the space of c\`adl\`ag excursions equipped with the Skorokhod topology. 
\end{proposition}

\begin{proof}[Proof of Proposition~\ref{prop:cv-prm}]
	Proposition~\ref{prop:vague} shows that the intensity measure of the Poisson random measure $\fN^{(n)}$ 
	converges vaguely as $n\to \infty$. Then the weak convergence of $\fN^{(n)}$, under the vague topology,  follows from \cite[Theorem~4.11]{KallenbergRM}. 
	The weak convergence $\xi^{(n)} \to \xi_{\fN}$ has already been proved by Rogers and Winkel \cite[Theorem 1.5]{RogeWink20}. 
	
	Therefore, both sequences $(\fN^{(n)},n\in \bN)$ and $(\xi^{(n)}, n\in \bN)$ are tight (see e.g.\ \cite[VI~3.9]{JacodShiryaev}),  
	and the latter implies the tightness of $(\ell^{(n)},n\in \bN)$. 
	We hence deduce immediately the tightness of the triple-valued sequence $((\fN^{(n)}, \xi^{(n)},\ell^{(n)}), n\in \bN)$. 
	As a result, we only need to prove that, for any subsequence $(\fN^{(n_i)},\xi^{(n_i)},\ell^{(n_i)})$ that
	converges in law, the limiting distribution is the same as $(\fN,\xi_{\fN}, \ell_{\fN})$.
	By Skorokhod representation, we may assume that $(\fN^{(n_i)},\xi^{(n_i)},\ell^{(n_i)})$ converges a.s.\@ to $(\fN,\widetilde{\xi}, \widetilde{\ell})$, 
	and it remains to prove that $\widetilde{\xi}=\xi_{\fN}$ and $\widetilde{\ell}=\ell_{\fN}$ a.s..  
	
	For any $\varepsilon>0$, since a.s.\@ $\fN$ has no spindle of length equal to $\varepsilon$, 
	the vague convergence of $\fN^{(n_i)}$ implies that, 
	a.s.\@ for any $t\ge 0$, we have the following weak convergence of finite point measures: 
	\[
	\sum_{s\le t} \mathbf{1}\{|\Delta \xi^{(n_i)} (s)|>\varepsilon\} \delta\left(s, \Delta \xi^{(n_i)}(s)\right)
	\quad \Longrightarrow \quad   
	\sum_{s\le t} \mathbf{1}\{|\Delta \xi_{\fN} (s)|>\varepsilon\} \delta\left(s, \Delta \xi_{\fN} (s)\right). 
	\]
	The subsequence above also converges a.s.\@ to $\sum_{s\le t} \mathbf{1}\{|\Delta \xi (s)|>\varepsilon\} \delta\big(s, \Delta \widetilde{\xi}(s)\big)$, since $\xi^{(n_i)}\rightarrow\widetilde{\xi}$ in
	$\bD(\bR_+,\bR)$.  
	By the L\'evy--It\^o decomposition, this is enough to conclude that 
	$\widetilde{\xi} = \xi_{\fN}$ a.s.. 
	
	For any $t\ge 0$, since $\xi^{(n_i)}\to \xi_{\fN}$ a.s.\@ and $\xi_{\fN}$ is a.s.\@ continuous at $t$, we have $(\xi^{(n_i)}(s), s\in [0,t]) \to (\xi_{\fN}(s), s\in [0,t])$ in $\bD([0,t], \bR)$ a.s.. Then $\inf_{s\in [0,t]} \xi^{(n)}(s) \to  \inf_{s\in [0,t]}\xi_{\fN}(s)$ a.s., because it is a continuous functional (w.r.t.\@ the Skorokhod topology). 
	In other words, $\widetilde{\ell}(t)=\ell_{\fN}(t)$ a.s.. By the continuity of the process $\ell_{\fN}$ we  have 
	$\widetilde{\ell}=\ell_{\fN}$ a.s., completing the proof.  
\end{proof}

\begin{lemma}[First passage over a negative level]\label{lem:T-}
	Suppose that $(\fN^{(n)}, \xi^{(n)})$ as in Proposition~\ref{prop:cv-prm} converges a.s.\@ to $(\fN, \xi_{\fN})$ as $n\to \infty$. Define
	\begin{equation}
		T_{-y}^{(n)} := T_{-y} (\xi^{(n)}):= \inf\{t\ge 0 \colon \xi^{(n)}(t)= -y\}, \qquad y>0,
	\end{equation}
	and similarly $T_{-y}:= T_{-y} (\xi_{\fN})$.
	Let $(h^{(n)})_{n\in \bN}$ be a sequence of positive numbers with $\lim_{n\to \infty} h^{(n)}=h>0$. 
	Then 
	$T^{(n)}_{- h^{(n)}} $ converges to $T_{-h}$ a.s..  
\end{lemma}

\begin{proof}
	Since the process $\xi^{(n)}$ is a spectrally positive L\'evy process with some Laplace exponent  
	$\Phi^{(n)}$, we know from \cite[Theorem~VII.1]{BertoinLevy} that 
	$(T^{(n)}_{(-y)+},y\ge 0)$ is a subordinator with Laplace exponent $(\Phi^{(n)})^{-1}$.  
	On the one hand, the convergence of $\Phi^{(n)}$ leads to $(T^{(n)}_{(-y)+},y\ge 0)\to (T_{(-y)+}, y\ge 0)$ in distribution under the Skorokhod topology. 
	Since $\xi_{\fN}$ is a.s.\ continuous at $T_{-h}$, we have
	$T^{(n)}_{-h^{(n)}}\to T_{-h}$ in distribution.

	On the other hand, we deduce from the convergence of the process $\xi^{(n)}$ in $\bD(\bR_+,\bR)$ that, for any $\varepsilon>0$, a.s.\@ there exists $N_1\in \bN$ such that for all $n>N_1$, 
	\[
	|\xi^{(n)}(T_{-h}) - (-h)|<\varepsilon. 
	\] 
	We may assume that $|h^{(n)} - h|<\varepsilon$ for all $n\in \bN$. 
	As a result, a.s.\@ for any $n> N_1$ and $y'<h^{(n)}-2\varepsilon$, 
	we have $T^{(n)}_{-y'}< T_{-h}$. 
	Hence, by the arbitrariness of $\varepsilon$ and the left-continuity of $T^{(n)}_{-y}$ with respect to $y$, we have
	$
	\limsup_{n\to \infty} T^{(n)}_{-h^{(n)}}\le T_{-h}$ a.s.. Recall that $T^{(n)}_{-h^{(n)}}\to T_{-h}$ in distribution, it follows that $T^{(n)}_{-h^{(n)}}\to T_{-h}$ a.s..
\end{proof}
%
\subsection{The scaling limit of a $\mathrm{PCRP}^{(\alpha)} (0,\alpha)$}\label{sec:cv-0alpha}

\begin{theorem}[Convergence of $\mathrm{PCRP}^{(\alpha)} (0,\alpha)$]
	\label{thm:crp-ip-0}
	For $n\in \bN$, let  $(C^{(n)}(y), y\ge 0)$ be a $\mathrm{PCRP}^{(\alpha)}(0,\alpha)$ starting from 
	${C^{(n)}(0)}\in \cC$ 
	and $(\beta(y), y\ge 0)$ be an $\mathrm{SSIP}^{(\alpha)}(0)$-evolution starting from $\beta(0)\in \cI_H$. 
	Suppose that  the interval partition $\frac{1}{n}  C^{(n)}(0)$ converges in distribution to $\beta(0)$ as $n\to \infty$, under $d_H$.
	Then the rescaled process $(\frac{1}{n}  C^{(n)}(2 n y), y\ge 0)$ converges in distribution to $(\beta(y), y\ge 0)$ as $n\to \infty$ in the Skorokhod sense and hence uniformly. 
\end{theorem}

We start with the simplest case. 

\begin{lemma}\label{lem:cv-clade}
	The statement of Theorem~\ref{thm:crp-ip-0} holds, if $\beta(0)= \{(0,b)\}$ and $C^{(n)} (0) = \{(0,b^{(n)})\}$, where $ \lim_{n\to \infty} n^{-1} b^{(n)}= b> 0$. 
\end{lemma}

To prove this lemma, we first give a representation of the rescaled PCRP. 
Let $(\mathbf{N}^{(n)},\xi^{(n)})$ be as in Proposition~\ref{prop:cv-prm}. 
For each $n\in \bN$, we define a random point measure on $\bR_+\times \cE$ by
\begin{equation}\label{relation1}
	\mathbf{N}^{(n)}_{\mathrm{cld}}:= \delta(0,\ff^{(n)})  + \mathbf{N}^{(n)}\Big|_{(0, T_{-\zeta(\ff^{(n)})} (\xi^{(n)})]\times \cE}, 
\end{equation}
where $\ff^{(n)}\sim \pi_{b^{(n)}}^{(n)}(-\alpha)$, independent of $\mathbf{N}^{(n)}$,  and $T_{-y} (\xi^{(n)}):= \inf\{t\ge 0 \colon \xi^{(n)}(t)=  -y\}$. 
Then we may assume that $\mathbf{N}^{(n)}_{\mathrm{cld}}$ is obtained from $\mathbf{D}^{(n)} \sim \mathtt{Clade}^{D}_{b^{(n)}}(\alpha)$ defined in \eqref{eq:clade-D} such that 
for each atom $\delta(s,f)$ of $\mathbf{D}^{(n)}$, $\mathbf{N}^{(n)}_{\mathrm{cld}}$ has an atom $\delta(\frac12n^{-(1+\alpha)} s,  n^{-1} f (2 n \cdot ) )$. 
Let $\xi^{(n)}_{\mathrm{cld}}$ be the scaffolding associated with $\mathbf{N}^{(n)}_{\mathrm{cld}}$ as in \eqref{eq:scaffolding-n}. 
As a consequence, we have the identity
\begin{equation}\label{relation2}
\beta^{(n)}(y):= \skewer\left(y,\mathbf{N}^{(n)}_{\mathrm{cld}}, \xi^{(n)}_{\mathrm{cld}}\right)
= \frac{1}{n} \skewer\left(2 n y,\mathbf{D}^{(n)}, \xi_{\mathbf{D}^{(n)}}\right), \quad y\ge 0,
\end{equation}
where $\xi_{\mathbf{D}^{(n)}}$ is defined as in \eqref{eq:scaffolding-D}. 
By Lemma~\ref{lem:pcrp0}, we may assume that $\beta^{(n)}(y) = \frac{1}{n}  C^{(n)}(2 n y)$ with $C^{(n)}$ 
a $\mathrm{PCRP}^{(\alpha)}(0,\alpha)$ starting from $C^{(n)} (0) = \{(0,b^{(n)})\}$.

\begin{proof}[Proof of Lemma~\ref{lem:cv-clade}]
	With notation as above, we shall prove that the rescaled process $\boldsymbol{\beta}^{(n)}:=(\beta^{(n)}(y),\,y\ge 0)$ converges to 
	an $\mathrm{SSIP}^{(\alpha)}(0)$-evolution $\boldsymbol{\beta}:=(\beta(y),\,y\ge 0)$ starting from $\{(0,b)\}$. 
	By Definition~\ref{defn:ip0} we can write
	$\boldsymbol{\beta}= \skewerbar(\mathbf{N}_{\mathrm{cld}}, \xi_{\mathrm{cld}})$, with 
	$\mathbf{N}_{\mathrm{cld}} = \clade (\ff, \fN)$ and $\xi_{\mathrm{cld}}$ its associated scaffolding, 
	where $\ff\sim \besq_{b} (-2\alpha)$ and $\fN$ is a Poisson random measure on $[0,\infty)\times \cE$ with intensity 
	$c_\alpha\mathrm{Leb}\otimes \Lambda^{(-2\alpha)}_{\mathtt{BESQ}}$, independent of $\ff$. 
	Using Proposition~\ref{prop:cv-prm} and Lemma~\ref{lem:ud}, we have 
	$(\fN^{(n)}, \xi^{(n)})\to (\fN, \xi)$ and
	$(\ff^{(n)}, \zeta(\ff^{(n)}))\to (\ff, \zeta(\ff))$ in distribution, independently. 
	Then it follows from Lemma~\ref{lem:T-} that this convergence also holds jointly with $T_{-\zeta(\ff^{(n)})}(\xi^{(n)})\to T_{-\zeta(\ff)}(\xi)$. 
	As a consequence, we have $(\fN_{\mathrm{cld}}^{(n)}, \xi_{\mathrm{cld}}^{(n)})\to (\fN_{\mathrm{cld}}, \xi_{\mathrm{cld}})$ in distribution. 
	
	With notation as above, consider the triple-valued sequence $(\fN^{(n)},\xi^{(n)}, \|\boldsymbol{\beta}^{(n)}\|)_{n\in \bN}$.  
	For each element in the triple, we know its tightness from Proposition~\ref{prop:cv-prm} and Lemma~\ref{lem:ud}, 
	then the triple-valued sequence is also tight. 
	Therefore, we can extract a subsequence  $\left(\fN_{\mathrm{cld}}^{(n_i)},\xi_{\mathrm{cld}}^{(n_i)},\|\boldsymbol{\beta}^{(n_i)}\|\right)_{i\in \bN}$ that converges in distribution to a limit process $(\fN_{\mathrm{cld}},\xi_{\mathrm{cld}},\widetilde{M})$. 
	Using Skorokhod representation, we may assume that this convergences holds a.s.. 
	We shall prove that $\boldsymbol{\beta}^{(n_i)}$ converges to $\boldsymbol{\beta}$ a.s., from which the lemma 
	follows. 
	
	We stress that the limit $\widetilde{M}$  
	has the same law as the total mass process $\|\boldsymbol{\beta}\|$, 
	but at this stage it is not clear if they are indeed equal. 
	We will prove that $\widetilde{M}= \|\boldsymbol{\beta}\|$ a.s..

	To this end, let us consider the contribution of the spindles with lifetime longer than $\rho>0$. 
	On the space
	${\bR_+\times \{f\in\cE\colon \zeta(f)>\rho \}}$,	$\fN_{\mathrm{cld}}$ has a.s.\@ a finite number of atoms, say enumerated in chronological order by $(t_j ,f _j)_{ j\le K}$ with $K\in \bN$. 
	Since $\fN_{\mathrm{cld}}$ has no spindle of length exactly equal to $\rho$, 
	by the a.s.\@ convergence $\fN^{(n_i)}_{\mathrm{cld}}\to \fN_{\mathrm{cld}}$, we may assume that  
	each $\fN^{(n_i)}_{\mathrm{cld}}$ also has $K$ atoms $(t^{(n_i)}_j ,f^{(n_i)}_j)_{ j\le K}$ on  
	${\bR_+\times \{f\in\cE\colon \zeta(f)>\rho \}}$, 
	and, for every $j\le K$, that  
	\begin{equation}\label{eq:tjfj}
		\lim_{i\to \infty} t^{(n_i)}_j = t_j, \quad \lim_{i\to \infty} \sup_{t\ge 0} \left|f^{(n_i)}_j (t)- f_j(t)\right| =0, \quad
		\text{and}~ \lim_{i\to \infty} \zeta(f^{(n_i)}_j) = \zeta (f_j) \quad \text{a.s.}. 
	\end{equation}
	Note that $ \zeta(f^{(n_i)}_j)=\Delta \xi^{(n_i)}_{\mathrm{cld}} (t^{(n_i)}_j) $.
	Since $\xi^{(n_i)}_{\mathrm{cld}} \to \xi_{\mathrm{cld}}$ in $\bD(\bR_+, \bR)$, we deduce that 
	\begin{equation}\label{eq:xicld-tj}
		\lim_{i\to \infty} \xi^{(n_i)}_{\mathrm{cld}} (t^{(n_i)}_j -)= \xi_{\mathrm{cld}}(t_j-) \quad \text{a.s..} 
	\end{equation}

	By deleting all spindles whose lifetimes are smaller than $\rho$, we obtain from $\boldsymbol{\beta}^{(n_i)}$ an interval partition evolution  
	\[
	\beta^{(n_i)}_{>\rho}(y) := \left\{ \left(M^{(n_i)} _{k-1} (y, \rho), M^{(n_i)} _{k} (y, \rho)\right), 1\le k\le K\right\} , \qquad y\ge 0,
	\] 
	where  $M^{(n_i)} _k (y, \rho) = \sum_{j\in [k]} f^{(n_i)}_j \left(y - \xi^{(n_i)}_{\mathrm{cld}} (t^{(n_i)}_j -)\right)$. 
	We similarly define $M _k (y, \rho)$ and $\beta_{>\rho}(y)$ from $\boldsymbol{\beta}$. 	
	By \eqref{eq:xicld-tj} and \eqref{eq:tjfj}, for all $k\le K$,
	\[
	\lim_{n\to \infty} \sup_{y\ge 0} \left| M^{(n_i)} _k (y, \rho)  - M _k (y, \rho)\right| =0\quad \text{a.s..}
	\]
	It follows that
	\begin{equation}\label{eq:cv>rho}
		\lim_{i\to \infty} \sup_{y\ge 0} d_H \left( \beta^{(n_i)}_{>\rho}(y) ,\beta_{>\rho}(y) \right) =0\quad \text{a.s..}
	\end{equation}
	In particular, for all $y,\rho>0$, 
	\[
	\widetilde{M} (y)= \lim_{i\to \infty}
	\|\beta^{(n_i)} (y)\|  \ge \lim_{i\to \infty}
	\|\beta^{(n_i)}_{>\rho} (y)\| = \|\beta_{>\rho}(y)\|, \quad \text{a.s..}
	\] 
	Then monotone convergence leads to, for all $y>0$, 
	\[
	\widetilde{M} (y)\ge  \lim_{\rho\downarrow 0} \|\beta_{>\rho}(y)\| = \|\beta (y)\|, \quad \text{a.s..}
	\]
	Moreover, since $\widetilde{M}$ and $\|\boldsymbol{\beta}\|$ also have the same law,  
	we conclude that $\widetilde{M}$  and $\|\boldsymbol{\beta}\|$   are indistinguishable. 

	Next, we shall show that $\boldsymbol{\beta}_{> \rho}$ approximates arbitrarily closely to $\boldsymbol{\beta}$ as $\rho \to 0$. 	 	
	Write $M_{\le\rho}:= \|\boldsymbol{\beta}\|-\|\boldsymbol{\beta}_{> \rho}\|$. 
	Then a.s.\ 
	$\lim_{\rho \to 0} M_{\le \rho}(y)=0$ for each rational $y>0$. 
	Noticing the monotonicity in $\rho$, we deduce by Dini's theorem the uniform convergence $\sup_{y\ge 0} M_{\le\rho}(y)\to 0$ a.s., as $\rho\to 0$. 

%
		For any $\varepsilon>0$, we can thus find a certain $\rho>0$ such that $\sup_{y\ge 0} M_{\le\rho}(y)<\varepsilon$. 
	With this specified $\rho$, since $d_{H} (\beta(y), \beta_{>\rho} (y))\le M_{\le \rho} (y)$ for each $y\ge 0$, 
	we have 
	\begin{equation}\label{eq:dH>rho}
		\sup_{y\ge 0} d_{H} \left(\beta(y), \beta_{>\rho} (y)\right)<\varepsilon. 
	\end{equation}
	Using \eqref{eq:cv>rho} and the uniform convergence $\|\boldsymbol{\beta}^{(n_i)}\| \to \widetilde{M}=\|\boldsymbol{\beta}\|$,  we deduce that the process $ M^{(n_i)}_{\le \rho} := \|\boldsymbol{\beta}^{(n_i)}\|-\|\boldsymbol{\beta}^{(n_i)}_{> \rho}\|$ converges to 
	$M_{\le\rho}$ uniformly. Then, for all $n$ large enough, 
	we also have 
	\[
	\sup_{y\ge 0}d_{H} \left(\beta^{(n_i)}(y), \beta^{(n_i)}_{>\rho} (y)\right)\le \sup_{y\ge 0} M^{(n_i)}_{\le \rho}(y)<2\varepsilon
	\] 
	Combining this inequality with \eqref{eq:cv>rho} and \eqref{eq:dH>rho}, 	 
	we deduce that 
	\[
	\limsup_{i\to \infty} \sup_{y\ge 0} d_{H}\left(\beta^{(n_i)}(y), \beta(y)\right) \le 3 \varepsilon. 
	\]
	As $\varepsilon$ is arbitrary, we conclude that $\boldsymbol{\beta}^{(n_i)}$ converges to $\boldsymbol{\beta}$ a.s.\@ under the uniform topology, completing the proof.  
\end{proof}

To extend to a general initial state, let us record the following result that characterises the convergence under $d_H$. 
\begin{lemma}[{\cite[Lemma~4.4]{IPPAT}}]\label{lem:dH} Let $\beta,\,\beta_n\in\mathcal{I}_H$, $n\ge 1$. Then $d_H(\beta_n,\beta)\rightarrow 0$ as $n\rightarrow\infty$
	if and only if 
	\begin{equation}\label{crit1}\forall_{(a,b)\in\beta}\ \exists_{n_0\ge 1}\ \forall_{n\ge n_0}\ \exists_{(a_n,b_n)\in\beta_n}\ a_n\rightarrow a\ \mbox{and}\ b_n\rightarrow b
	\end{equation}
	and
	\begin{equation}\label{crit2}\forall_{(n_k)_{k\ge 1}\colon n_k\rightarrow\infty}\ \forall_{(c_k,d_k)\in\beta_{n_k},\,k\ge 1\colon d_k\rightarrow d\in(0,\infty],\,c_k\rightarrow c\neq d}\ (c,d)\in\beta.
	\end{equation}
\end{lemma}

\begin{proof}[Proof of Theorem~\ref{thm:crp-ip-0}]
	For the case $\beta(0)= \emptyset$, by convention $\beta(y) =\emptyset$ for every $y\ge 0$. 
	Then the claim is a simple consequence of the convergence of the total mass processes. 
	
	So we may assume that $\beta(0)\ne \emptyset$. By Definition~\ref{defn:ip0}, we can write  
	$\boldsymbol{\beta} = \Concat_{U\in \beta(0)} \boldsymbol{\beta}_U$, 
	where each process $\boldsymbol{\beta}_U:=(\beta_U(y), y\ge 0)\sim \mathtt{Clade}_{\mathrm{Leb}(U)}(\alpha)$, independent of the others. 
	We claim that a.s., for any $\varepsilon>0$ we can find at most a finite (random) number of intervals, say $U_1, U_2, \ldots, U_{K_\varepsilon}\in \beta(0)$, listed from left to right, 
	such that 
	\begin{equation}\label{eq:R_k}
		\sup_{y\ge 0} R_{K_\varepsilon}(y) <\varepsilon, \quad \text{where}~R_k(y) := \|\beta(y)\| - \sum_{i=1}^k \|\beta_{U_i} (y)\|, \; k\ge 1. 
	\end{equation}
	Indeed, let $(U_i, i\ge 1)$ be a sequence containing all intervals of $\beta(0)$, ranked by decreasing order of length. 
	Since $\|\boldsymbol{\beta}\|\sim \besq_{\|\beta(0)\|}(0)$, $\|\boldsymbol{\beta}_{U_i}\|\sim 
	\besq_{\mathrm{Leb}(U_i)}(0)$ for each $i\ge 1$, and, for each $k\ge 1$, $R_k$ is independent of the family $\{\boldsymbol{\beta}_{U_i}, i\in [k]\}$, 
	we deduce from Proposition~\ref{prop:alphatheta} that 
	$R_k \sim \besq_{r_k}(0)$ with $r_k:= \|\beta(0)\| - \sum_{i=1}^k\mathrm{Leb}(U_i)$. 
	By self-similarity and monotonicity, $\sup_{y\ge 0} R_k(y)\downarrow 0$ a.s., as $k\rightarrow\infty$;  on
	this a.s.\ event, we find, for each $\varepsilon>0$ a random $K_\varepsilon$ so that \eqref{eq:R_k} holds.

	For each $n\in \bN$, we similarly assume that $C^{(n)}(y) = \Concat_{U\in C^{(n)}(0)} C^{(n)}_U(y), y\ge 0$, 
	where each process $(C^{(n)}_U(y), y\ge 0)\sim \mathrm{PCRP}_{\mathrm{Leb}(U)}^{(\alpha)}(0,\alpha)$, independent of the others. 
	
	Due to the convergence of the initial state  $\frac{1}{n} C^{(n)}(0) \to \beta(0)$ and     
	by Lemma~\ref{lem:dH} we can find for each $i\ge 1$ a sequence 
	$U^{(n)}_i = (a^{(n)}_i, b^{(n)}_i) \in C^{(n)}(0)$, $n\in \bN$, 
	such that  $a^{(n)}_i/n\to \inf U_i$ and $b^{(n)}_i/n\to \sup U_i $. 
	In particular, we have ${\rm Leb}(U^{(n)}_i)/n \to {\rm Leb}(U_i)$ for every $i\le K_\varepsilon$. 
	Then we may assume by Lemma~\ref{lem:cv-clade} that, for all $i\ge 1$,
	\begin{equation}\label{eq:beta_i}
		\lim_{n\to \infty} \sup_{y\ge 0}d_{H} \left(\frac{1}{n} C_{U^{(n)}_i}^{(n)}(2 n y), \beta_{U_i}(y) \right)  = 0\quad  \text{a.s..}
	\end{equation}
	
	Moreover, it is easy to see that the total mass of a $\mathrm{PCRP}^{(\alpha)}(0,\alpha)$ is a Markov chain described by $\pi (0)$ in \eqref{eq:pi}. 
	By independence, for every $k\ge 1$, the rescaled process   
	\[
	R_k^{(n)}(y) := \frac{1}{n}\Big\|C^{(n)}(2 n y)\Big\| -\frac{1}{n} \sum_{i=1}^k \Big\|C^{(n)}_{U^{(n)}_i} (2 n y)\Big\|,\quad  y\ge 0,
	\] 
	has the law of $\pi^{(n)}_{r^{(n)}_k} (0)$ as in \eqref{eq:pi-n}, where 
	$r^{(n)}_k= \|C^{(n)}(0) \|- \sum_{i=1}^k {\rm Leb}(U^{(n)}_i)$.  
	By Lemma~\ref{lem:ud} and Skorokhod representation, we may assume,  as $n\to \infty$, $\sup_{y\ge 0} |R_k^{(n)}(y) - R_k(y)|\to 0$ a.s.\ and hence $\sup_{y\ge 0}|R_{K_\varepsilon}^{(n)}(y)-R_{K_\varepsilon}(y)|\rightarrow 0$ a.s..

	An easy estimate shows that
	\[
	d_{H} \left(\frac{1}{n} C^{(n)} (2 n y), \beta(y) \right)
	\le 2 R_{K_\varepsilon}^{(n)}(y) +2 R_{K_\varepsilon}(y) + \sum_{i=1}^{K_\varepsilon} d_{H} \left(\frac{1}{n} C^{(n)}_{U^{(n)}_i} (2 n y), \beta_{U_i}(y)\right). 
	\]
	As a result, combining \eqref{eq:R_k} and \eqref{eq:beta_i}, we have 
	\[
	\limsup_{n\to \infty} \sup_{y\ge 0}d_{H} \left(\frac{1}{n} C^{(n)}(2 n y), \beta(y) \right)  
	\le 4\varepsilon \quad \text{a.s.}.
	\]
	By the arbitrariness of $\varepsilon$ we deduce the claim. 
\end{proof}

\subsection{The scaling limit of a $\mathrm{PCRP}^{(\alpha)} (\theta_1,\alpha)$}\label{sec:cv-thetaalpha}

\begin{proposition}
	[Convergence of a $\mathrm{PCRP}^{(\alpha)}( \theta_1,\alpha)$] 
	\label{prop:crp-ip-theta}
	Let $\theta_1\ge 0$. 
	For $n\in \bN$, let  $(C^{(n)}(y), y\ge 0)$ be a $\mathrm{PCRP}^{(\alpha)}(\theta_1, \alpha)$
	starting from $C^{(n)}(0)\in \cC$.
	Suppose that  the interval partition $ \frac{1}{n}  C^{(n)}(0)$ converges in distribution to 
	$\beta(0)\in \cI_H$ as $n\to \infty$, under $d_H$.
	Then the process $(\frac{1}{n} C^{(n)}(2 n y), y\ge 0)$ converges in distribution to an $\mathrm{SSIP}^{(\alpha)}(\theta_1)$-evolution starting from $\beta(0)$, as $n\to \infty$, in $\bD(\bR_+,\cI_{H})$ under the Skorokhod topology. 
\end{proposition}

\begin{proof}
	We only need to prove the case when $\theta_1>0$ and $C^{(n)}(0)=\emptyset$ for every $n\in \bN$; then combining this special case and Theorem~\ref{thm:crp-ip-0} leads to the general result. 
	The arguments are very similar to those in the proof of Lemma~\ref{lem:cv-clade}; we only sketch the strategy here and omit the details.    
	
	Fix $j\in \bN$. Let $(\fN^{(n)},\xi^{(n)},\ell^{(n)})_{n\in \bN}$ be the sequence given in 
	Proposition~\ref{prop:cv-prm}.  
	For each $n\in \bN$, by using Theorem~\ref{thm:jccp}, we may write  
	\[
	\beta^{(n)} (y):=\frac{1}{n}  C^{(n)}(2 n y) = \skewer\left(y, \fN^{(n)}\big|_{[0,T^{(n)}_{-j\theta_1/\alpha})}, j+ \xi^{(n)}_{\theta_1 }\big|_{[0,T^{(n)}_{-j\theta_1/\alpha})}\right), \quad y\in [0,j],
	\]
	where 
	$
	\xi^{(n)}_{\theta_1}:= \xi^{(n)} + (1-\alpha/\theta_1) \ell^{(n)}
	$
	and $T^{(n)}_{-j\theta_1/\alpha}:=T_{-j\theta_1/\alpha}(\xi^{(n)})=T_{-j}(\xi_{\theta_1}^{(n)})$.
	By Proposition~\ref{prop:cv-prm} and Skorokhod representation, we may assume that 
	$(\fN^{(n)},\xi^{(n)} ,\xi^{(n)}_{\theta_1})$ converges  a.s.\@ to $(\fN,\mathbf{X}_\alpha,\mathbf{X}_{\theta_1})$. 
	Then it follows from Lemma~\ref{lem:T-} that $T^{(n)}_{-j\theta_1/\alpha}\to T_{-j\theta_1/\alpha}(\mathbf{X}_\alpha)=T_{-j}(\mathbf{X}_{\theta_1})$, cf. \eqref{eq:Talphatheta}. 
	Next, in the same way as in the proof of Lemma~\ref{lem:cv-clade}, 
	we consider for any $\rho>0$ the interval partition evolution $\boldsymbol{\beta}^{(n)}_{>\rho}$ associated with the spindles of $\boldsymbol{\beta}^{(n)}$ with lifetime longer than $\rho$. 
	By proving that for any $\rho>0$, 
	$(\beta^{(n)}_{>\rho}(y), y\in [0,j]) \to (\beta_{>\rho}(y), y\in [0,j])$ as $n\to \infty$, 
	and that $\|\boldsymbol{\beta}^{(n)}\|\!-\! \| \boldsymbol{\beta}^{(n)}_{>\rho}\| \to 0$ as $\rho\downarrow 0$ uniformly for all $n\in \bN$, 
	we deduce the convergence of $(\beta^{(n)}(y), y\in [0,j])$. 
	This leads to the desired statement.  
\end{proof}

\section{Convergence of the three-parameter family}\label{sec:pcrp-ssip}

In this section we consider the general three-parameter family $\mathrm{PCRP}^{(\alpha)}(\theta_1, \theta_2)$ with $\theta_1,\theta_2\ge 0$. 
In Section~\ref{sec:cvproof} we establish a related convergence result, Theorem~\ref{thm:crp-ip-bis}, for the processes killed upon hitting $\emptyset$, with the limiting diffusion being an $\mathrm{SSIP}_{\!\dagger}$-evolution introduced in \cite{ShiWinkel-1}. Using Theorem~\ref{thm:crp-ip-bis}, we obtain a pseudo-stationary distribution for an $\mathrm{SSIP}_{\!\dagger}$-evolution in Proposition~\ref{prop:ps-theta1theta2}, which enables us to introduce an excursion measure and thereby  construct an $\mathrm{SSIP}$-evolution from excursions, for suitable parameters, in Sections~\ref{sec:exc} and ~\ref{sec:rec} respectively.   
In Section~\ref{sec:results}, we finally complete the proofs of Theorem~\ref{thm:crp-ip} and the other results stated in Section \ref{sec:mainresult}.

\subsection{Convergence when $\emptyset$ is absorbing}\label{sec:cvproof}

If we choose any table in a $\mathrm{PCRP}^{(\alpha)}(\theta_1, \theta_2)$, then its size evolves as a  $\pi(-\alpha)$-process until the first hitting time of zero; before the deletion of this table, the tables to its left form a $\mathrm{PCRP}^{(\alpha)}(\theta_1, \alpha)$ and the tables to its right a $\mathrm{PCRP}^{(\alpha)}(\alpha, \theta_2)$. 
This observation suggests us to make such decompositions and to use the convergence results obtained in the previous section. 
A similar idea has been used in \cite{ShiWinkel-1} for the construction of an $\mathrm{SSIP}$-evolution with absorption in $\emptyset$, abbreviated as $\mathrm{SSIP}_{\!\dagger}$-evolution, which we shall now recall. 
Specifically, define a function $\phi\colon \cI_H \to \big(\cI_H \times  (0,\infty)\times \cI_H\big) \cup \{(\emptyset,0,\emptyset)\}$ by setting $\phi(\emptyset):= (\emptyset,0,\emptyset)$ and, for $\beta\ne \emptyset$, 
\begin{equation}\label{eq:phi}
	\phi(\beta) := \big( \beta \cap (0,\inf U), \mathrm{Leb}(U), \beta \cap (\sup U, \|\beta\|) - \sup U \big), 
\end{equation}
where  $U$ is the longest interval in $\beta$; we take $U$ to be the leftmost one if this is not unique. 

\begin{definition}[$\mathrm{SSIP}$-evolution with absorption in $\emptyset$, Definition 1.3 of \cite{ShiWinkel-1}] \label{defn:ipe}
	$\!\!$Consider $\theta_1\ge 0$, $\theta_2 \ge 0$ and $\gamma\in \cI_H$. 
	Set $T_0 := 0$ and $\beta(0):= \gamma$. 
	For $k\ge 0$, suppose by induction that we have obtained $(\beta(t), t\le T_k)$.
	\begin{itemize}
		\item If $\beta(T_k)= \emptyset$, then we stop and set 
		$T_{i}:= T_k$ for every $i\ge k$ and $\beta(t):=\emptyset $ for $t\ge T_k$. 
		\item If $\beta(T_k)\!\ne\! \emptyset$, denote $(\beta^{(k)}_1,m^{(k)},\beta^{(k)}_2):= \phi(\beta(T_k))$. Conditionally on the history, let $\ff^{(k)}\sim \besq_{m^{(k)}}(-2 \alpha)$ and $\boldsymbol{\gamma}^{(k)}_i=(\gamma_i^{(k)}(s),\,s\ge 0)$ an $\mathrm{SSIP}^{(\alpha)}( \theta_i)$-evolution starting from $\beta_i^{(k)}$, $i=1,2$,   
		with $\ff^{(k)}, \boldsymbol{\gamma}_1^{(k)},\boldsymbol{\gamma}_2^{(k)}$ independent. Set $T_{k+1}:= T_k + \zeta(\ff^{(k)})$.
		We define
		\[
		\beta(t) := \gamma^{(k)}_1(t\!-\! T_k) \concat\left\{\left( 0,\ff^{(k)}(t\!-\!T_k)\right)\right\} \concat {\rm rev}\big(\gamma^{(k)}_2(t\!-\!T_k)\big) , \qquad t\in (T_k, T_{k+1}]. 
		\]
	\end{itemize}
	We refer to $(T_k)_{k\ge1}$ as the \emph{renaissance levels} and $T_{\infty}:= \sup_{k\ge 1} T_k \in [0, \infty]$ as the \emph{degeneration level}. 
	If $T_{\infty}< \infty$, then by convention we set  
	$\beta(t) :=\emptyset$ for all $t\ge  T_{\infty}$. 
	Then the process $\boldsymbol{\beta}:=(\beta(t), t\ge 0)$ is called an \emph{$\mathrm{SSIP}_{\!\dagger}^{(\alpha)}( \theta_1, \theta_2)$-evolution} starting from $\gamma$. 
\end{definition}

Note that $\emptyset$ is an absorbing state of an $\mathrm{SSIP}_{\!\dagger}^{(\alpha)}( \theta_1, \theta_2)$-evolution by construction. Let us summarise a few results obtained in \cite[Theorem 1.4 and Corollary 3.6]{ShiWinkel-1}.

\begin{theorem}[\cite{ShiWinkel-1}]\label{thm:IvaluedMP}
	For $\theta_1,\theta_2\ge 0$, let $(\beta(t),t\ge 0)$ be an $\mathrm{SSIP}_{\!\dagger}^{(\alpha)}(\theta_1, \theta_2)$-evolution, with renaissance levels $(T_k,k\ge 0)$ and degeneration level $T_{\infty}$. Set $\theta= \theta_1+\theta_2-\alpha$. 
	\begin{longlist}
		\item (Hunt property)  $(\beta(t),t\ge 0)$ is a  Hunt process with continuous paths in $(\cI_H ,d_{H})$.
		\item \label{item:mass} (Total-mass)	
		$(\|\beta(t)\|,t\ge 0)$ is a ${\tt BESQ}_{\|\beta^0\|}(2\theta)$ killed at its first hitting time of zero.
		\item \label{item:levels} (Degeneration level) If $\theta\ge 1$ and $\beta(0)\ne \emptyset$, then a.s.\ $T_{\infty}=\infty$ and $\beta(t)\ne \emptyset$ for every $t\ge 0$;  
		if $\theta<1$, then a.s.\ $T_{\infty}<\infty$ and $\lim_{t\uparrow T_{\infty}} d_H(\beta(t), \emptyset) =0$. 
	\end{longlist} 
\end{theorem}

\begin{theorem}\label{thm:crp-ip-bis} 
	Let $\theta_1,\theta_2\ge 0$ and $\theta=\theta_1\!+\!\theta_2\!-\!\alpha$. 
	For $n\in \bN$, let  $(C^{(n)}(t), t\ge 0)$ be a $\mathrm{PCRP}^{(\alpha)}(\theta_1,\theta_2)$ starting from $C^{(n)}(0)= \gamma^{(n)}$ and killed at $\zeta^{(n)} =\inf\{ t\ge 0\colon C^{(n)}(t) =\emptyset \}$.  Let $(\beta(t),t\!\ge\! 0)$ be an $\mathrm{SSIP}_{\!\dagger}^{(\alpha)}(\theta_1,\theta_2)$-evolution starting from $\gamma$ and 
	$\zeta \!=\!\inf\{ t\!\ge\! 0\colon \beta(t)\!=\!\emptyset\}$. 
	Suppose that either $\gamma\ne \emptyset$ or $\theta<1$. 
	If $ \frac{1}{n}  \gamma^{(n)}$ converges in distribution to 
	$\gamma$ under $d_H$, as $n\to \infty$, then the following convergence holds in $\bD(\bR_+,\cI_H)$:  
	\[
	\Big(\frac{1}{n}  C^{(n)}\big((2 n t)\wedge\zeta^{(n)}\big), t\ge 0 \Big)
	\underset{n\to \infty}{\longrightarrow}  (\beta(t), t\ge 0) , \quad \text{in distribution}. 
	\]	
	Moreover, $\zeta^{(n)}/2n$ converges to $\zeta$ in distribution jointly.
\end{theorem}
To prove Theorem~\ref{thm:crp-ip-bis}, we shall construct a sequence of  $\mathrm{PCRP}^{(\alpha)} (\theta_1, \theta_2)$ on a sufficiently large probability space by using  $\mathrm{PCRP}^{(\alpha)} (\theta_1,\alpha)$, $\mathrm{PCRP}^{(\alpha)} (\alpha,\theta_2)$  and up-down chains of law $\pi_k(-\alpha)$ defined in \eqref{eq:pi}; the idea is similar to Definition~\ref{defn:ipe}. 
Then the convergences obtained in Proposition~\ref{prop:crp-ip-theta} and Lemmas~\ref{lem:ud}--\ref{lem:ud-bis} permit us to conclude. A detailed proof of Theorem~\ref{thm:crp-ip-bis} is postponed to Appendix~\ref{sec:proof-thm:crp-ip-bis}.

\subsection{Pseudo-stationarity of $\mathrm{SSIP}_{\!\dagger}$-evolutions}
\begin{proposition}[Pseudo-stationary distribution of an $\mathrm{SSIP}_{\!\dagger}^{(\alpha)}(\theta_1, \theta_2)$-evolution]\label{prop:ps-theta1theta2}
	For $\theta_1, \theta_2\ge 0$ and $\theta:=\theta_1+\theta_2-\alpha$, let $(Z(t),\,t\ge 0)$ be a $\besq (2 \theta)$ killed at zero with $Z(0)>0$, independent of $\bar\gamma\sim \mathtt{PDIP}^{(\alpha)}( \theta_1,\theta_2)$. 
	Let $(\beta(t),\, t\ge 0)$ be an $\mathrm{SSIP}_{\!\dagger}^{(\alpha)}(\theta_1, \theta_2)$-evolution starting from $Z(0)   \bar\gamma$. 
	Fix any $t\ge 0$, then $\beta(t)$ has the same distribution as $Z(t) \bar\gamma$. 
\end{proposition}
Analogous results for $\mathrm{SSIP}^{(\alpha)}(\theta_1)$-evolutions have been obtained in \cite{Paper1-2,IPPAT}, however, the strategy used in their proofs does not easily apply to our three-parameter model.  
We shall use a completely different method, which crucially relies on the discrete approximation by  $\mathrm{PCRP}^{(\alpha)}(\theta_1,\theta_2)$ in Theorem~\ref{thm:crp-ip-bis}. It is easy to see that the total mass of a $\mathrm{PCRP}^{(\alpha)}(\theta_1,\theta_2)$ evolves according to a Markov chain defined by $\pi(\theta)$ as in \eqref{eq:pi}, with $\theta=\theta_1+\theta_2-\alpha$. 
Conversely, given any $C (0)\in \cC$ and $Z\sim \pi_{\|C(0)\|} (\theta)$, we can embed a process $(C (t),\, t\ge 0)\sim \mathrm{PCRP}^{(\alpha)} (\theta_1, \theta_2)$, starting from $C (0)$, such that its total-mass evolution 
is $Z$. More precisely, in the language of the Chinese restaurant process, at each jump time when $Z$ increases by one, add a customer according to the seating rule in Definition~\ref{defn:oCRP}; and whenever $Z$ decreases by one, perform a down-step, i.e.\@ one uniformly chosen customer leaves. It is easy to check that this process indeed satisfies the definition of $\mathrm{PCRP}^{(\alpha)} (\theta_1, \theta_2)$ in the introduction. 
Recall the probability law $\mathtt{oCRP}^{(\alpha)}_{m}(\theta_1,\theta_2)$ from Definition~\ref{defn:oCRP}. 

\begin{lemma}[Marginal distribution of a $\mathrm{PCRP}^{(\alpha)} (\theta_1, \theta_2)$]\label{prop:crp-ps}
	Consider a $\mathrm{PCRP}^{(\alpha)} (\theta_1, \theta_2)$  $(C(t),\, t\ge 0)$ 
	starting from $C(0)\sim \mathtt{oCRP}^{(\alpha)}_{m}(\theta_1,\theta_2)$ with $m\in \bN_0$.
	Then, at any time $t \ge 0$, $C(t)$ has a mixture distribution $\mathtt{oCRP}^{(\alpha)}_{\|C(t)\|}(\theta_1,\theta_2)$, where the total number of customers has distribution $(\|C(t)\|,\,t\ge 0)\sim \pi_m(\theta)$ with $\theta:= \theta_1+\theta_2- \alpha$. 
\end{lemma}
\begin{proof} 
	Let $Z\sim \pi_{\|C(0)\|} (\theta)$ and we consider $(C (t),\, t\ge 0)\sim \mathrm{PCRP}^{(\alpha)} (\theta_1, \theta_2)$, starting from $C(0)$, as a process embedded in $Z\sim \pi_{\|C(0)\|} (\theta)$, in the way we just explained as above.  
	Before the first jump time $J_1$ of $Z$, $C(t) = C(0)\sim \mathtt{oCRP}^{(\alpha)}_{m}(\theta_1,\theta_2)$ by assumption. 
	At the first jump time $J_1$ of $Z$, it follows from Proposition~\ref{prop:down} that, given $Z(J_1)$, $C(J_1)$ has conditional distribution $\mathtt{oCRP}^{(\alpha)}_{Z(J_1)}(\theta_1,\theta_2)$. 
	The proof is completed by induction. 
\end{proof}

\begin{proof}[Proof of Proposition~\ref{prop:ps-theta1theta2}]
	For $n\in \bN$, consider a process $C^{(n)}\sim \mathrm{PCRP}^{(\alpha)}(\theta_1, \theta_2)$, 
	starting from $C^{(n)}(0) \sim\mathtt{oCRP}_{\lfloor n Z(0)\rfloor}^{(\alpha)}(\theta_1, \theta_2)$ and killed at $\emptyset$.  
	It follows from Lemma~\ref{prop:crp-ps}  that, for every $t\ge 0$, 
	$C^{(n)}(t)$ has the mixture distribution $\mathtt{oCRP}^{(\alpha)}_{N^{(n)}(t\wedge\zeta(N^{(n)}))}(\theta_1, \theta_2)$ with 
	$(N^{(n)}(t),t\ge 0)\sim\pi_{\lfloor n Z(0)\rfloor} (\theta)$. 
	By Proposition~\ref{prop:crp-pdip},
	$ \frac{1}{n}  C^{(n)}(0)$ converges in distribution to $Z(0)   \bar\gamma$ under $d_H$.
	For any fixed $t\ge 0$, it follows from Theorem~\ref{thm:crp-ip-bis} that 
	$\frac{1}{n} C^{(n)}(2 n t)$ converges in distribution to $\beta(t)$. 
	Using Lemmas~\ref{lem:ud}--\ref{lem:ud-bis} and Proposition~\ref{prop:crp-pdip} leads to the desired statement. 
\end{proof}

\subsection{SSIP-evolutions}\label{sec:SSIP}

Let $\alpha\in (0,1)$ and $\theta_1,\theta_2\ge 0$. 
Recall that the state $\emptyset$ has been defined to be a trap of an $\mathrm{SSIP}_{\!\dagger}^{(\alpha)}(\theta_1, \theta_2)$-evolution. 
In this section, we will show that, for certain cases, depending on the value of $\theta:= \theta_1+\theta_2-\alpha$, we can include $\emptyset$ as an initial state such that it leaves $\emptyset$ continuously. 

More precisely, consider independent $(Z(t),\, t\ge 0)\sim \besq_0(2 \theta)$
and $\bar\gamma\sim \pdip^{(\alpha)}(\theta_1, \theta_2)$. 
Define for every $t\ge 0$ a probability kernel $K_t$ on $\cI_H$: for $\beta_0\in \cI_H$ and measurable $A \subseteq \cI_H$, 
\begin{equation}\label{eq:kernel-ssip}
	K_t(\beta_0, A)= \bP\big(\beta(t)\in A,\, t< \zeta(\boldsymbol{\beta})\big)
	+ \int_0^t  \bP(Z({t\!-\!r})\bar\gamma\in A) \bP(\zeta(\boldsymbol{\beta})\in dr),
\end{equation}
where $\boldsymbol{\beta}=(\beta(t),\,t\ge 0)$ is an $\mathrm{SSIP}_{\!\dagger}^{(\alpha)}(\theta_1, \theta_2)$-evolution starting from $\beta_0$, and $\zeta(\boldsymbol{\beta})$ is the first hitting time of $\emptyset$ by $\boldsymbol{\beta}$. 
Note that \cite[Corollary XI.(1.4)]{RevuzYor} yields for fixed $s\ge 0$, that
\begin{equation}\label{eq:gamma}(Z(t),\, t\ge 0)\sim \besq_0(2 \theta),\quad\theta>0,\qquad\Rightarrow\qquad Z(s)\sim \mathtt{Gamma}(\theta,1/2s).
\end{equation}
When $\beta_0=\emptyset$, we have  by convention $\zeta(\boldsymbol{\beta})=0$ and the first term in \eqref{eq:kernel-ssip} vanishes.

\begin{theorem}\label{thm:hunt}
	Let $\theta_1,\theta_2\ge 0$. 
	The family $(K_t,\,t\ge 0)$ defined in \eqref{eq:kernel-ssip} is the transition semigroup of a path-continuous Hunt process on the Polish space $\cI_H$. 
\end{theorem}
\begin{definition}[$\mathrm{SSIP}^{(\alpha)}(\theta_1,\theta_2)$-evolutions]\label{defn:ssip}
	For $\theta_1,\theta_2\ge 0$, a path-continuous Markov process with transition semigroup	$(K_t,\,t\ge 0)$ is called an \emph{$\mathrm{SSIP}^{(\alpha)}(\theta_1,\theta_2)$-evolution}. 
\end{definition}

\begin{proposition}\label{prop:mass}
	For $\theta_1,\theta_2\ge 0$, let $(\beta(t) ,\,t\ge 0)$ be a Markov process with transition semigroup	$(K_t,t\!\ge \!0)$. 
	Then the total mass $(\|\beta(t)\| ,t\!\ge \!0)$ is a $\besq(2\theta)$ with $\theta =\theta_1\!+\!\theta_2\!-\!\alpha$. 
\end{proposition}
\begin{proof}
	We know from Theorem~\ref{thm:IvaluedMP} that the total mass of an $\mathrm{SSIP}_{\!\dagger}^{(\alpha)}(\theta_1,\theta_2)$-evolution evolves according to a $\besq(2\theta)$ killed at zero. Therefore, the description in \eqref{eq:kernel-ssip} implies that $(\|\beta(t)\| ,\,t\ge 0)$ has the semigroup of $\besq(2\theta)$. 
\end{proof}

The proof of Theorem~\ref{thm:hunt} is postponed to Section~\ref{sec:rec}. We distinguish three phases: 
\begin{itemize}
	\item $\theta\in  [-\alpha,0]$: by convention, $Z\sim \besq_0(2 \theta)$ is the constant zero process and thus the second term in \eqref{eq:kernel-ssip} vanishes; then $(K_t,\,t\ge 0)$ is just the semigroup of an $\mathrm{SSIP}_{\!\dagger}^{(\alpha)}(\theta_1, \theta_2)$-evolution. In this case Theorem~\ref{thm:hunt} is encompassed by Theorem~\ref{thm:IvaluedMP}. 
	\item $\theta\in (0,1)$: by Theorem~\ref{thm:IvaluedMP}~(\ref{item:mass}) and \cite[Equation~(13)]{GoinYor03} we deduce that $\zeta(\boldsymbol{\beta})$ is a.s.\ finite in \eqref{eq:kernel-ssip}, with $\zeta(\boldsymbol{\beta})\ed \|\beta(0)\|/2G$, where $G\sim \mathtt{Gamma}(1-\theta, 1)$. In this case, we will construct  an $\mathrm{SSIP}^{(\alpha)}(\theta_1,\theta_2)$-evolution as a \emph{recurrent extension} of $\mathrm{SSIP}_{\!\dagger}^{(\alpha)}(\theta_1, \theta_2)$-evolutions, by using an excursion measure that will be introduced in Section~\ref{sec:exc}. 
	\item $\theta\ge 1$: since $\zeta(\boldsymbol{\beta})=\infty$ a.s., the second term in \eqref{eq:kernel-ssip} vanishes unless $\beta(0)=\emptyset$ and $\emptyset$ is an entrance boundary with an entrance law  
	$K_t(\emptyset,\,\cdot\,) =\bP(Z(t)\bar\gamma\in \cdot\,)$, by Proposition \ref{prop:ps-theta1theta2}. See also \cite[Proposition 4.14]{ShiWinkel-1},
	where this was shown using a different construction and a different formulation of the entrance law, which is seen to be equivalent
	to \eqref{eq:kernel-ssip} by writing $\bar{\gamma}=B'\big(V^\prime\bar{\gamma}_1\concat\{(0,1\!-\!V^\prime)\}\big) \concat (1\!-\!B') \bar{\beta}\sim \mathtt{PDIP}^{(\alpha)}(\theta_1,\theta_2)$ as in Corollary \ref{cor:pdipdec}.\end{itemize}

\subsection{The excursion measure of an SSIP-evolution when $\theta\in (-\alpha,1)$}\label{sec:exc}
In this section, we fix $\theta_1,\theta_2\ge 0$ and suppose that $-\alpha< \theta \!=\! \theta_1 \!+\!\theta_2 \!-\!\alpha \!<\!1$. 
We shall construct an $\mathrm{SSIP}^{(\alpha)} (\theta_1, \theta_2)$ excursion measure $\Theta:= \Theta^{(\alpha)} (\theta_1, \theta_2)$, which is a $\sigma$-finite measure on the space $\bC([0,\infty), \cI_H)$ of  
continuous functions in $(\cI_H, d_{H})$, endowed with the uniform metric and the Borel $\sigma$-algebra. 
Our construction is in line with Pitman and Yor \cite[(3.2)]{PitmYor82}, by the following steps. 
\begin{itemize}
	\item 
	For each $t> 0$, define a measure $N_t$ on $\cI_H$ by
	\begin{align}\label{eq:entrance}
		N_t(A) &:= \bE\left[ (Z(t))^{\theta-1}  \mathbf{1}_{A} (Z(t)  \bar\gamma)\right],\quad \text{ measurable }A\subseteq \cI_H\setminus\{\emptyset\}, \\
		N_t (\{\emptyset\}) &:=\infty,\nonumber
	\end{align}
	where $Z=(Z(t),\, t\ge 0)\sim \besq_0(4 - 2 \theta)$
	and $\bar\gamma\sim \pdip^{(\alpha)}(\theta_1, \theta_2)$ are independent.  
	As $4-2\theta>2$, the process $Z$ never hits zero. 
	We have $N_t (\cI_H\setminus\{\emptyset\})=t^{\theta-1}/2^{1-\theta}\Gamma(2-\theta)$.
	
	Then $(N_t,\, t\ge 0)$ is an entrance law for an $\mathrm{SSIP}_{\!\dagger}^{(\alpha)} (\theta_1, \theta_2)$-evolution $(\beta(t),\, t\ge 0)$. Indeed, with notation above, we have by Proposition~\ref{prop:ps-theta1theta2} that, for every $s,t\ge 0$ and $f$ non-negative measurable, 
	\begin{align*}
		\int \bE\left[f(\beta(s)) \mid \beta(0) = \gamma\right]N_t (d\gamma)
		&= \bE\left[ (Z(t))^{\theta-1} \bE_{Z(t)\bar\gamma}\left[f(\beta(s))\right] \right]\\
		&=\bE\left[ (Z'(0))^{\theta-1}\bE_{Z'(0)}\left[f(Z'(s) \bar\gamma)\right] \right],
	\end{align*}
	where $(Z'(s),\, s\ge 0)$ is a $\besq(2 \theta)$ killed at zero with $Z'(0) = Z(t)$. 
	Since we know from the duality property of $\besq(2 \theta)$, see e.g.\ \cite[(3.b) and (3.5)]{PitmYor82}, that
	\[
	(Z'(0))^{\theta-1}  \bE_{Z'(0)} \big[g(Z'(s))\big]=	\bE_{Z'(0)} \big[g(\widetilde{Z}(s)) (\widetilde{Z}(s))^{\theta-1}\big] , \qquad \forall  s> 0,  
	\]
	where  $\widetilde{Z}\sim \besq(4 -2 \theta)$ starting from $Z'(0)$, 
	it follows from the Markov property that  
	\begin{align*}
		\bE\left[\! (Z'(0))^{\theta-1}\bE_{Z'(0)}[f(Z'({s}) \bar\gamma)] \right]\! 
		&=  \bE\left[ \bE_{Z'(0)}\left[(\widetilde{Z}(s))^{\theta-1} f(\widetilde{Z}({s}) \bar\gamma)\right]  \right] \\
		&= \bE\left[ (Z({t+s}))^{\theta-1} f(Z({t+s}) \bar\gamma) \right] 
		= \int f(\gamma) N_{t+s}(d\gamma). 
	\end{align*}
	We conclude that 
	\[
	\int \bE\big[f(\beta(s)) \mid\beta(0) = \gamma\big]N_t (d\gamma) = \int f(\gamma) N_{t+s} (d \gamma), \quad \forall s,t\ge 0. 
	\]
	\item 
	As a consequence, there exists a unique $\sigma$-finite measure $\Theta$ on $\bC((0,\infty), \cI_H)$ such that for all $t> 0$ and $F$ bounded measurable functional, we have the identity  
	\begin{equation}\label{eq:Theta:entrance}
		\Theta [F\circ L_t]  = \int \bE \left[ F( \beta(s),\, s\ge 0) \mid \beta(0) = \gamma\right] N_t (d \gamma) ,  
	\end{equation}
	where $(\beta(s),\, s\ge 0)$ is an $\mathrm{SSIP}_{\!\dagger}^{(\alpha)} (\theta_1, \theta_2)$-evolution and $L_t$ stands for the shift operator. 
	See \cite[VI.48]{RogersWilliams} for details. 
	In particular, for each $t>0$ and measurable $A\subseteq \cI_H\setminus\{\emptyset\}$, we have the identity
	$\Theta \{(\beta(s),\,s>0)\in\bC((0,\infty),\cI_H)\colon\beta(t) \in A \} =N_t(A)$. In particular, 
	\begin{equation}\label{eq:Theta-zeta}
		\Theta(\zeta>t)=\Theta\big\{(\beta(s),s>0)\in\bC((0,\infty),\cI_H)\colon\beta(t) \ne \emptyset \big\} = t^{\theta-1}/2^{1-\theta}\Gamma(2-\theta).
	\end{equation}
	\item The image of $\Theta$ by the mapping $(\beta(t),\,t>0)\mapsto(\|\beta(t)\|,\,t>0)$ is equal to the push-forward of $\Lambda_{\mathtt{BESQ}}^{(2 \theta)}$ from $\bC([0,\infty),\cI_H)$ to $\bC((0,\infty),\cI_H)$ under the restriction map, where $\Lambda_{\mathtt{BESQ}}^{(2 \theta)}$ is the excursion measure of 
	$\besq (2 \theta)$ as in \eqref{eq:besq-exc}.  
	In particular,  we have for $\Theta$-almost every $(\beta(t),\,t>0)\in\bC((0,\infty),\cI_H)$
	\begin{equation}\label{eq:Theta:tozero}
		\limsup_{t\downarrow 0} \|\beta(t)\|=0 \quad \Longrightarrow \quad \lim_{t\downarrow 0} d_H(\beta(t),\emptyset)=0.
	\end{equation}
	Therefore, we can ``extend'' $\Theta$ to $\bC([0,\infty), \cI_H)$, by defining $\beta(0)=\emptyset$ for $\Theta$-almost every $(\beta(t),t>0)\in\bC((0,\infty),\cI_H)$, and we also set 
	\begin{equation}\label{eq:Theta:zero}
		\Theta\big\{\boldsymbol{\beta}\in\bC([0,\infty),\cI_H)\colon\boldsymbol{\beta}\equiv \emptyset\big\} = 0.
	\end{equation} 
\end{itemize}

Summarising, we have the following statement. 

\begin{proposition}\label{prop:Theta-besq}
	Let $\theta_1,\theta_2\ge 0$ and suppose that $-\alpha< \theta \!=\! \theta_1 +\theta_2 -\alpha \!<\!1$. 
	Then there is a unique $\sigma$-finite measure $\Theta= \Theta^{(\alpha)} (\theta_1, \theta_2)$ on $\bC([0,\infty), \cI_H)$ that satisfies 
	\eqref{eq:Theta:entrance} and \eqref{eq:Theta:zero}. 
	Moreover,  the image of $\Theta$ by the mapping $(\beta(t),\,t\ge 0)\mapsto(\|\beta(t)\|,\,t\ge 0)$ is $\Lambda_{\mathtt{BESQ}}^{(2 \theta)}$, the excursion measure of $\besq (2 \theta)$. 
\end{proposition}

For the case $\theta_1=\theta_2 =0$, the law $\pdip^{(\alpha)}(\theta_1, \theta_2)$ coincides with the Dirac mass $\delta_{\{(0,1)\}}$. 
As a consequence, the $\mathrm{SSIP}^{(\alpha)} (0,0)$ excursion measure is just the pushforward of 
$\Lambda_{\mathtt{BESQ}}^{(2 \theta)}$, by the map $ x \mapsto \{(0,x)\}$ from $[0,\infty)$ to $\cI_H$. 
When $\theta_1=0$ and $\theta_2=\alpha$, it is easy to check using \cite[Proposition 2.12(i), Lemma 3.5(ii), Corollary 3.9]{IPPAT} that $2\alpha\Theta^{(\alpha)}(0,\alpha)$ is the push-forward via the mapping $\skewerbar$ in Definition~\ref{def:skewer} of the measure 
$\nu^{(\alpha)}_{\mathrm{\bot cld}}$ studied in \cite[Section 2.3]{IPPAT}. 

\subsection{Recurrent extension when $\theta\in (0,1)$}\label{sec:rec}
Consider the $\mathrm{SSIP}^{(\alpha)} (\theta_1, \theta_2)$ excursion measure $\Theta:= \Theta^{(\alpha)} (\theta_1, \theta_2)$ and suppose that $\theta= \theta_1 +\theta_2-\alpha \in (0,1)$.  
It is well-known \cite{Salisbury86} in the theory of Markov processes that excursion measures such as $\Theta$ can be used to construct a recurrent extension of a Markov process. To this end,  
let $\mathbf{G}\sim\mathtt{PRM}(\mathrm{Leb}\otimes b_\theta\Theta)$, where $b_\theta=2^{1-\theta}\Gamma(2-\theta)/\Gamma(\theta)$. 

For every $s\ge 0$, set $\sigma_s = \int_{[0,s]\times \cI_H} \zeta (\boldsymbol{\gamma}) \mathbf{G} (dr, d\boldsymbol{\gamma})$. As the total mass process under $\Theta$ is the $\mathtt{BESQ}(2\theta)$ excursion measure with $\theta\in (0,1)$, 
the process $(\sigma_s, s\ge 0)$ coincides with the inverse local time of a $\mathtt{BESQ}(2\theta)$, which is well-known to be a subordinator. 
We define 
\begin{equation}\label{eq:nonconcat}
	\beta(t)= \Concat_{\text{ points }(s,\boldsymbol{\gamma}_s)\text{ of } \mathbf{G},\, \sigma_{s-}< t\le \sigma_{s}}   \gamma_s(t- \sigma_{s-}), \qquad t\ge 0.    
\end{equation}
This ``concatentation'' consists of at most one interval partition since $(\sigma_s,\,s\ge 0)$ is increasing.

\begin{proposition}\label{prop:rec-extension}
	The process $(\beta(t),\, t\ge 0)$ of \eqref{eq:nonconcat} is a path-continuous Hunt process with transition semigroup $(K_t,\,t\ge 0)$.  Its total mass process $(\|\beta(t)\|,\, t\ge 0)$ is a $\besq(2\theta)$. 
\end{proposition}
\begin{proof}
	We can use \cite[Theorem 4.1]{Salisbury86}, since we have the following properties: 
	\begin{itemize}
		\item  an $\mathrm{SSIP}_{\!\dagger}^{(\alpha)} (\theta_1, \theta_2)$-evolution is a Hunt process; 
		\item $\Theta$ is concentrated on $\{\boldsymbol{\gamma}\in\bC([0,\infty),\cI_H)\colon 0<\zeta(\boldsymbol{\gamma})<\infty, \gamma(t) = \emptyset\mbox{ for all }t\ge \zeta(\boldsymbol{\gamma}) \}$;
		\item for any $a>0$, we have $\Theta\{\boldsymbol{\gamma}\in\bC([0,\infty),\cI_H)\colon\sup_{t\ge 0} \|\gamma(t)\|\ge a  \}<\infty$;
		\item $\int (1-e^{-\zeta(\boldsymbol{\gamma})}) b_\theta\Theta (d\boldsymbol{\gamma}) = 1$; 
		\item \eqref{eq:Theta:entrance} holds;
		\item $\Theta$ is infinite and $\Theta\{\boldsymbol{\gamma}\in\bC([0,\infty),\cI_H)\colon \gamma(0)\ne\emptyset \}=0$. 
	\end{itemize}
	It follows that $(\beta(t),\, t\ge 0)$ is a Borel right Markov process with transition semigroup $(K_t,\, t\ge 0)$. Moreover, the total mass process $(\|\beta(t)\|,\, t\ge 0)$ evolves according to a $\besq(2\theta)$ by Proposition~\ref{prop:Theta-besq}. 
	
	In fact, $(\beta(t),\,t\ge 0)$ a.s.\ has continuous paths. Fix any path on the almost sure event that the total mass process $(\|\beta(t)\|, t\ge 0)$ and all excursions $\boldsymbol{\gamma}_s$ are continuous. For any $t\ge 0$, if $\sigma_{s-}<t<\sigma_{s}$ for some $s\ge 0$, then the continuity at $t$ follows from that of $\boldsymbol{\gamma}_s$. 
	For any other $t$, we have $\beta(t)=\emptyset$ and the continuity at $t$ follows from the continuity of the ${\tt BESQ}(2\theta)$ total mass process. This completes the proof. 
\end{proof}

We are now ready to give the proof of Theorem \ref{thm:hunt}, which claims that $(K_t,\,t\ge 0)$ defined in \eqref{eq:kernel-ssip} is the transition semigroup of a path-continuous Hunt process.
\begin{proof}[Proof of Theorem~\ref{thm:hunt}] When $\theta\in (0,1)$, this is proved by Proposition~\ref{prop:rec-extension}.
	When $\theta\le 0$, the state $\emptyset$ is absorbing, and an $\mathrm{SSIP}^{(\alpha)}(\theta_1, \theta_2)$-evolution coincides with an $\mathrm{SSIP}_{\!\dagger}^{(\alpha)}(\theta_1, \theta_2)$-evolution.   
	For $\theta\ge 1$, the state $\emptyset$ is inaccessible, but an entrance boundary of the $\mathrm{SSIP}^{(\alpha)}(\theta_1,\theta_2)$-evolution, see
	also \cite[Proposition 4.14]{ShiWinkel-1}. For these cases, the proof is completed by Theorem~\ref{thm:IvaluedMP}, the only modification is when starting from $\emptyset$. Specifically, the modified semigroup is still measurable. Right-continuity starting from $\emptyset$ follows from the continuity of the total mass process, and this entails the strong Markov property by the usual approximation argument. 
\end{proof}

\subsection{Proofs of  results in Section~\ref{sec:mainresult} }\label{sec:results}
We will first prove  Theorem~\ref{thm:crp-ip}  and identify the limiting diffusion in Theorem~\ref{thm:crp-ip} with an $\mathrm{SSIP}$-evolution as defined in  Definition~\ref{defn:ssip}. Then we complete proofs of the other results in Section~\ref{sec:mainresult}. 

\begin{lemma}\label{lem:entrance}
	Let $\alpha\in(0,1)$, $\theta_1,\theta_2\ge 0$ and $\theta=\theta_1+\theta_2-\alpha$. 	For $n\in \bN$, let  $(C^{(n)}(t), t\ge 0)$ be a $\mathrm{PCRP}^{(\alpha)}(\theta_1,\theta_2)$ starting from $C^{(n)}(0)=\gamma^{(n)}$.
	If $ \frac{1}{n}  \gamma^{(n)}$ converges to $\emptyset$ under $d_H$, then for any $t\ge 0$, 
	\[
	\frac{1}{n} C^{(n)}(2nt) \to Z(t) \bar{\gamma}\quad \text{ in distribution},
	\]
	where $(Z(t),\,t\ge 0)\sim \mathtt{BESQ}_0(2\theta)$ and $\bar{\gamma}\sim \mathtt{PDIP}^{(\alpha)} (\theta_1,\theta_2)$ are independent. 
\end{lemma}

\begin{proof}
	We start with the case when $\theta<1$. 
	Let $\zeta^{(n)}$ be the hitting time of $\emptyset$ for $C^{(n)}$, then Theorem~\ref{thm:crp-ip-bis} yields that $\zeta^{(n)}/2n\to 0$ in probability as $n\to \infty$. 
	For any $t>0$ and any bounded continuous function $f$ on $\cI_H$, we have 
	\begin{align}\label{eq:lifesplit}
		\bE\left[f\left(\frac{1}{n}C^{(n)}(2 n t)\right) \right]
		=&	\bE\left[\ind\{\zeta^{(n)} \le 2nt \} f\left(\frac{1}{n}\widetilde{C}^{(n)}\big(2nt - \zeta^{(n)}\big)\right) \right]\\
		&+ 	\bE\left[\ind\{\zeta^{(n)} > 2nt \} f\left(\frac{1}{n}C^{(n)}(2 n t)\right) \right], \nonumber
	\end{align}
	where $\widetilde{C}^{(n)}(s)= C^{(n)}(s+\zeta^{(n)})$, $s\ge 0$. 
	As $n\to \infty$, since  $\zeta^{(n)}/2n\to 0$ in probability, the second term tends to zero. 
	By the strong Markov property  and Lemma~\ref{prop:crp-ps},   $\widetilde{C}^{(n)}(s)$ has the mixture distribution  $\mathtt{oCRP}_{\|\widetilde{C}^{(n)}(s)\|}^{(\alpha)}(\theta_1,\theta_2)$. 
	Since $\|C^{(n)}(2nt)\|/n \to Z(t)$ in distribution by Lemma~\ref{lem:ud}, we deduce by Proposition~\ref{prop:crp-pdip} that 
	the first term tends to 
	$\bE\left[f\left(Z(t)\bar{\gamma}\right) \right]$, as desired.

	For $\theta\ge 1$, at least one of $\theta_1\ge\alpha$ or $\theta_2\ge\alpha$. Say, $\theta_1\ge\alpha$. We may assume that
	$C^{(n)}(t)=C_1^{(n)}(t)\concat C_0^{(n)}(t)\concat C_2^{(n)}(t)$ for independent $(C_1^{(n)} (t),\,t\ge 0)\sim \mathrm{PCRP}^{(\alpha)} (\theta_1,0)$ starting from $\emptyset$, $(C_0^{(n)} (t),t\!\ge\! 0)\!\sim\! \mathrm{PCRP}^{(\alpha)} (\alpha,0)$ starting from $C^{(n)}(0)$, and $(C_2^{(n)} (t),t\!\ge\! 0)\linebreak \sim\mathrm{PCRP}^{(\alpha)} (\alpha,\theta_2)$ starting from $\emptyset$. 
	For the middle term $C_0^{(n)}$, the $\theta\!\le\! 0$ case yields that $\frac{1}{n} C_0^{(n)}(2nt) \to \emptyset$ in distribution.   
	For the other two, applying \eqref{eq:gamma} and Lemmas~\ref{prop:crp-ps}, \ref{lem:ud} and Proposition~\ref{prop:crp-pdip}  yields 
	$ \frac{1}{n} C_1^{(n)}(2nt) \to Z_1(t) \bar{\gamma}_1$ in distribution, with $Z_1(t)\sim \mathtt{Gamma}(\theta_1\!-\!\alpha,1/2t)$ and $\bar{\gamma}_1\sim \mathtt{PDIP}^{(\alpha)} (\theta_1,0)$, and  $ \frac{1}{n} C_2^{(n)}(2nt) \to Z_2(t) \bar{\gamma}_2$ in distribution, with $Z_2(t)\sim \mathtt{Gamma}(\theta_2,1/2t)$ and $\bar{\gamma}_2\sim \mathtt{PDIP}^{(\alpha)} (\alpha,\theta_2)$. 
	We complete the proof by applying the decomposition \eqref{eq:pdip-decomp}.
\end{proof}

\begin{proof}[Proof of Theorem~\ref{thm:crp-ip}] 
	We first consider the case $\theta\le 0$. Then the state $\emptyset$ is absorbing, and an $\mathrm{SSIP}^{(\alpha)}(\theta_1, \theta_2)$-evolution coincides with an $\mathrm{SSIP}_{\!\dagger}^{(\alpha)}(\theta_1, \theta_2)$-evolution.  
	For this case, the proof is completed by Theorem~\ref{thm:crp-ip-bis}.  

	We next consider $\theta>0$ and prove that the limiting diffusion is given by an $\mathrm{SSIP}^{(\alpha)}(\theta_1, \theta_2)$-evolution $\boldsymbol{\beta}=(\beta(t),\,t\ge 0)$ with $\zeta(\boldsymbol{\beta})=\inf\{t\ge 0\colon\beta(t)=\emptyset\}$ as defined in Definition~\ref{defn:ssip}. 
	It suffices to prove the convergence in $\bD([0,T], \cI_{H})$ for a fixed $T>0$. 	
	The convergence of finite-dimensional distributions follows readily from Theorem~\ref{thm:crp-ip-bis}, Lemma~\ref{lem:entrance} and the description in \eqref{eq:kernel-ssip}.
	Specifically, for $\theta\in(0,1)$, we proceed as in the proof of Lemma~\ref{lem:entrance} and see the first term in \eqref{eq:lifesplit} converge to $\bE[\ind\{\zeta(\boldsymbol{\beta})\le t\}f(Z_0(t-\zeta(\boldsymbol{\beta}))\bar{\gamma})]$ where $Z_0\sim{\tt BESQ}_0(2\theta)$ and $\bar{\gamma}\sim{\tt PDIP}^{(\alpha)}(\theta_1,\theta_2)$ are independent and jointly independent of $\boldsymbol{\beta}$, while the second term converges to $\bE[\ind\{\zeta(\boldsymbol{\beta})>t\}f(\beta(t))]$. For $\theta\ge 1$ and $\beta(0)=\emptyset$, convergence of marginals holds by Lemma \ref{lem:entrance} and \eqref{eq:kernel-ssip}. Theorem \ref{thm:crp-ip-bis} then establishes finite-dimensional convergence, also when $\beta(0)\neq\emptyset$.
	
	It remains to check tightness. Let $\boldsymbol{\beta}^{(n)}=(\beta^{(n)}(t),\,t\ge 0):= \frac{1}{n} C^{(n)}(2n\,\cdot\,)$. 
	Since we already know from Lemma~\ref{lem:ud} that the sequence of total mass processes $\|\boldsymbol{\beta}^{(n)}\|$, $n\ge 1$, converges in distribution, it is tight.
	For $h>0$, define the modulus of continuity by
	\[\omega \left(\|\boldsymbol{\beta}^{(n)}\|, h\right) = \sup \left\{ \big|\|\beta^{(n)}(s)\|- \|\beta^{(n)}(t)\|\big|\colon |s-t|\le h\right\}.\] 
	For any $\varepsilon>0$, the tightness implies that there exists $\Delta'$ such that for any $h\le 2\Delta'$, 
	\[
	\limsup_{n\to \infty} \bE\left[\omega \left(\|\boldsymbol{\beta}^{(n)}\|, h\right)\wedge 1\right] <\varepsilon; 
	\]
	this is an elementary consequence of \cite[Proposition VI.3.26]{JacodShiryaev}. See also \cite[Theorem~16.5]{Kallenberg}.
	
	For $1\le i\le \lfloor T/\Delta' \rfloor$, set $t_i = i \Delta'$ and let 
	$\boldsymbol{\beta}^{(n)}_i$ be the process obtained by shifting $\boldsymbol{\beta}^{(n)}$ to start from $t_i$, killed at $\emptyset$. 
	The convergence of the finite-dimensional distributions yields that each $\beta^{(n)}(t_i)$ converges weakly to $\beta(t_i)$. 
	Since $\beta(t_i)\ne \emptyset$ a.s., by Theorem~\ref{thm:crp-ip-bis} each sequence $\boldsymbol{\beta}^{(n)}_i$ converges in distribution as $n\to \infty$. 
	So the sequence $(\boldsymbol{\beta}^{(n)}_i, n\in \bN)$ is tight, as the space $(\cI_{H},d_H)$ is Polish.  
	By tightness there  
	exists $\Delta_i$ such that for any $h<\Delta_i$, 
	\[
	\limsup_{n\to \infty} \bE\left[\omega \left(\boldsymbol{\beta}^{(n)}_i, h\right)\wedge 1\right] < 2^{-i} \varepsilon. 
	\]
	Now let $\Delta = \min (\Delta', \Delta_0, \Delta_1, \ldots, \Delta_{\lfloor T/\Delta' \rfloor} )$. 
	For any $s\le t\le T$ with $t\!-\!s\le \Delta$, consider $i$ such that $t_i \le s < t_{i+1}$, then 
	$t\!-\!t_i < \Delta'\! +\! t\!-\!s \le 2 \Delta'$. 
	If $\zeta(\boldsymbol{\beta}^{(n)}_i) \le t\!-\!t_i$, then $\boldsymbol{\beta}^{(n)}$ touches $\emptyset$ during the time interval $[t_i, t]$ and thus 
	$\max (\|\beta^{(n)}(s)\|,\|\beta^{(n)}(t)\|) \le \omega \left(\|\boldsymbol{\beta}^{(n)}\|, 2\Delta'\right)$. 
	Therefore, we have 
	\[
	d_H\left(\beta^{(n)}(s), \beta^{(n)}(t)\right) \le d_H\left(\beta^{(n)}_i(s), \beta^{(n)}_i(t)\right)+ 2\omega \left(\|\boldsymbol{\beta}^{(n)}\|, 2\Delta'\right). 
	\] 
	It follows that for $h<\Delta$, 
	\[
	\bE\left[\omega \left(\boldsymbol{\beta}^{(n)}, h\right)\wedge 1\right]\le 
	2\bE\left[\omega \left(\|\boldsymbol{\beta}^{(n)}\|, 2\Delta'\right)\wedge 1\right] + 
	\sum_{i=0}^{\lfloor T/\Delta' \rfloor}\bE\left[\omega \left(\boldsymbol{\beta}^{(n)}_i, \Delta_i\right)\wedge 1\right] . 
	\]	
	So we have  $\limsup_{n\to \infty}\bE\left[\omega \left(\boldsymbol{\beta}^{(n)}, h\right)\wedge 1\right]\le 4\varepsilon$. 
	This leads to the tightness, e.g.\ via \cite[Theorem 16.10]{Kallenberg}.
\end{proof}

\begin{proof}[Proof of Proposition~\ref{prop:ps-theta1theta2-nokill}]
	This follows from Proposition~\ref{prop:ps-theta1theta2} and the semigroup description in \eqref{eq:kernel-ssip}.
\end{proof}

\begin{theorem}\label{thm:continst} Let $\alpha\in(0,1)$, $\theta_1,\theta_2\ge 0$ and $\gamma_n\in\cI_H$ with $\gamma_n\rightarrow\gamma\in\cI_H$. Let 
	$\boldsymbol{\beta}_n$, $n\ge 1$, and $\boldsymbol{\beta}$ be ${\rm SSIP}^{(\alpha)}(\theta_1,\theta_2)$-evolutions starting from $\gamma_n$, $n\ge 1$, and $\gamma$, respectively. Then 
	$\boldsymbol{\beta}_n\rightarrow\boldsymbol{\beta}$ in distribution in $\mathbb{C}(\mathbb{R}_+,\cI_H)$ equipped with the locally uniform 
	topology.
\end{theorem}
\begin{proof} 
	We first assume $\gamma\ne \emptyset$. 
	It follows easily from Lemma \ref{lem:dH} that we may assume that $\gamma_n=\beta_{n,1}^{(0)}\concat\{(0,m_n^{(0)})\}\concat\beta_{n,2}^{(0)}$ with $m_n^{(0)}\rightarrow m^{(0)}$, $\beta_{n,i}^{(0)}\rightarrow\beta_i^{(0)}$, $i=1,2$, and $\phi(\gamma)=(\beta_1^{(0)},m^{(0)},\beta_2^{(0)})$. We will now couple the constructions in Definition \ref{defn:ipe} and use the notation from there.  
	
	Given $(\beta_{n,1}^{(k)},m_n^{(k)},\beta_{n,2}^{(k)})\rightarrow(\beta_1^{(k)},m^{(k)},\beta_2^{(k)})$ a.s., for some $k\ge 0$, the Feller property of 
	\cite[Theorem 1.8]{IPPAT} allows us to apply \cite[Theorem 19.25]{Kallenberg} and, by Skorokhod representation, we obtain 
	$\boldsymbol{\gamma}_{n,i}^{(k)}\!\rightarrow\!\boldsymbol{\gamma}_i^{(k)}$ a.s.\ in $\mathbb{C}(\mathbb{R}_+,\cI_H)$, $i\!=\!1,2$, as $n\!\rightarrow\!\infty$. For 
	$\mathbf{f}^{(k)}\!\sim\!{\tt BESQ}_{m^{(k)}}(-2\alpha)$ and \linebreak
	$\mathbf{f}^{(k)}_n(s)\!:=\!(m_n^{(k)}\!/m^{(k)})\mathbf{f}^{(k)}((m^{(k)}\!/m_n^{(k)})s)$, $s\!\ge\! 0$, we find 
	$\mathbf{f}^{(k)}_n\!\!\sim\!{\tt BESQ}_{m^{(k)}_n}(-2\alpha)$. As $n\!\rightarrow\!\infty$, \linebreak
	$(\mathbf{f}^{(k)}_n,\zeta(\mathbf{f}^{(k)}_n))\rightarrow(\mathbf{f}^{(k)},\zeta(\mathbf{f}^{(k)}))$ a.s.. 
	And as $\gamma_1^{(k)}(\zeta(\mathbf{f}^{(k)}))\concat\gamma_2^{(k)}(\zeta(\mathbf{f}^{(k)}))$ has a.s.\ a unique longest interval,
	$\phi\big(\gamma_{n,1}^{(k)}(\zeta(\mathbf{f}^{(k)}_n))\concat\gamma_{n,2}^{(k)}(\zeta(\mathbf{f}^{(k)}_n))\big)\rightarrow\phi\big(\gamma_1^{(k)}(\zeta(\mathbf{f}^{(k)}))\concat\gamma_2^{(k)}(\zeta(\mathbf{f}^{(k)}))\big)$, a.s..
	
	Inductively, the convergences stated in this proof so far hold a.s.\ for all $k\ge 0$. When $\theta:=\theta_1+\theta_2-\alpha\ge 1$, by Theorem~\ref{thm:IvaluedMP}~(\ref{item:levels}) this construction extends to the entire time axis and thus gives rise to coupled 
	$\boldsymbol{\beta}_n$ and $\boldsymbol{\beta}$.  When $\theta<1$, arguments as at the end of the proof of Theorem \ref{thm:crp-ip-bis} allow us to prove the convergence until the first hitting time of $\emptyset$ jointly with the convergence of the hitting times. When $\theta\le 0$, we extend the constructions by absorption in $\emptyset$. When 
	$\theta\in(0,1)$, we extend by the same ${\rm SSIP}^{(\alpha)}(\theta_1,\theta_2)$ starting from $\emptyset$. In each case, we deduce
	that $\boldsymbol{\beta}_n\rightarrow\boldsymbol{\beta}$ a.s., locally uniformly. 
	
	We next deal with the case $\gamma =\emptyset$. 
	For $\theta<1$, we know from the total mass convergence (Lemmas~\ref{lem:ud} and ~\ref{lem:ud-bis}) that  the interval-partition valued process until its the first hitting time to $\emptyset$ converges to a constant zero process,  with the hitting times converging to zero jointly. Then  the same coupling arguments as in the previous paragraph complete the proof. 
	For $\theta \ge 1$, we first prove the convergence of the marginal distributions as in the proof of Lemma~\ref{lem:entrance}. 
	More precisely, since at least one of $\theta_1\ge\alpha$ and $\theta_2\ge\alpha$ holds, we may assume by the left-right symmetry (\cite[Proposition~4.4]{ShiWinkel-1}) that $\theta_2\ge\alpha$.  
	Also, by combining \cite[Definition~4.3(ii)]{ShiWinkel-1} and Definition \ref{defn:alphatheta}, we obtain a decomposition of the process, with 
	$\beta_n(t)=\beta_{1}(t)\concat \beta_{n,0}(t)\concat \beta_{2}(t)$, $t\ge 0$, for independent $(\beta_{1}(t),\,t\ge 0)\linebreak \sim \mathrm{SSIP}^{(\alpha)} (\theta_1,\alpha)$ starting from $\emptyset$, $(\beta_{n,0} (t),t\!\ge\! 0)\!\sim\! \mathrm{SSIP}^{(\alpha)} (0,\alpha)$ starting from $\gamma_n$, and $(\beta_{2} (t),t\!\ge\! 0) \sim\mathrm{SSIP}^{(\alpha)} (0,\theta_2)$ starting from $\emptyset$. 
	For the middle term, the convergence of the total mass yields that $\frac{1}{n} \beta_{n,0}(2nt) \to \emptyset$ in distribution.   
	For the other two, applying Proposition~\ref{prop:ps-theta1theta2-nokill}, we deduce that 
	$(\beta_{1}(t), \beta_{2}(t))$ has the same distribution as $(Z_1(t) \bar{\gamma}_1, Z_2(t) \bar{\gamma}_2)$, with $Z_1(t)\sim \mathtt{Gamma}(\theta_1,1/2t)$,  $\bar{\gamma}_1\sim \mathtt{PDIP}^{(\alpha)} (\theta_1,\alpha)$, $Z_2(t)\sim \mathtt{Gamma}(\theta_2-\alpha,1/2t)$ and $\bar{\gamma}_2\sim \mathtt{PDIP}^{(\alpha)} (0,\theta_2)$, independent of each other. 
	It follows from the decomposition \eqref{eq:pdip-decomp} that $\beta_n(t)\rightarrow\beta_1(t)\concat\beta_2(t)=\beta(t)$, where $\beta(t)$ has the same distribution as $Z(t) \bar{\gamma}$ with independent $\mathtt{Gamma}(\theta,1/2t)$ and $\bar{\gamma}\sim \mathtt{PDIP}^{(\alpha)} (\theta_1,\theta_2)$. 
	This gives the convergence of the marginal distributions and therefore the finite-dimensional distributions, due to the non-empty initial state case that we have already proved. We further extend this to the convergence of processes, by the same tightness arguments as in the proof of Theorem~\ref{thm:crp-ip}.  	
\end{proof}

\begin{proof}[Proof of Theorem~\ref{thm:dP}]
	For an $\mathrm{SSIP}$-evolution, we have established the pseudo-stationarity (Proposition~\ref{prop:ps-theta1theta2-nokill}), self-similarity,  
	path-continuity, Hunt property (Theorem~\ref{thm:crp-ip}) and the continuity in the initial state (Theorem \ref{thm:continst}).  	
	With these properties in hand, we can easily prove this theorem by following the same arguments as in \cite[proof of Theorem~1.6]{Paper1-2}. Details are left to the reader. 
\end{proof}

We now prove Theorem~\ref{thm:Theta}, showing that when $\theta=\theta_1+\theta_2-\alpha\in (-\alpha,1)$, the excursion measure $\Theta:=\Theta^{(\alpha)}(\theta_1,\theta_2)$ of Section \ref{sec:exc} is the limit of rescaled PCRP excursion measures.  
Recall that the total mass process of ${\rm PCRP}^{(\alpha)}(\theta_1,\theta_2)$ has distribution $\pi_1(\theta)$. 
We have already obtained the convergence of the total mass process from Proposition~\ref{prop:vague}.  

\begin{proof}[Proof of Theorem~\ref{thm:Theta}]
	Recall that $\zeta(\boldsymbol{\gamma})= \inf\{t>0\colon \gamma(t)=\emptyset \}$ denotes the lifetime of an excursion  $\boldsymbol{\gamma}\in \bD([0,\infty),\cI_H)$. To prove vague convergence, we proceed as in the proof of Proposition~\ref{prop:vague}. In the present setting, we work on the space of measures on $\bD([0,\infty),\cI_H)$  that are bounded on $\{\zeta>t\}$ for all $t>0$. We denote by $\mathrm{P}^{(n)}$ the distribution of $C^{(n)}$, a killed $\mathrm{PCRP}^{(\alpha)}(\theta_1,\theta_2)$ starting from $(1)$. It suffices to prove for fixed $t>0$, 
	\begin{enumerate}\item[1.] $\Theta(\zeta = t)=0$,
		\item[2.] $(\Gamma(1+\theta)/(1-\theta)) n^{1-\theta} \cdot \mathrm{P}^{(n)} (\zeta>t) \underset{n\to \infty}{\longrightarrow} \Theta(\zeta > t)$,
		\item[3.] $\mathrm{P}^{(n)}(\,\cdot\,|\,\zeta>t) \underset{n\to \infty}{\longrightarrow} \Theta(\,\cdot\,|\,\zeta>t)$ weakly.
	\end{enumerate}
	
	1. This follows from \eqref{eq:Theta-zeta}. 
	
	2. Since the total-mass process $\|C^{(n)}\|$ is an up-down chain of law $\widetilde{\pi}_1^{(n)}(\theta)$, Proposition~\ref{prop:vague} implies the following weak convergence of finite measures on $(0,\infty)$: 	
	\begin{align}\label{eq:cv-nu}
		&\frac{ \Gamma(1+\theta)}{1-\theta} n^{1-\theta} \bP \left(\|C^{(n)}(t)\|\in \cdot~;~ \zeta (C^{(n)} )> t\right) \\[-0.1cm]
		&\underset{n\to \infty}{\longrightarrow}
		\Lambda_{\mathtt{BESQ}}^{(2 \theta)}\big\{f\in\bC([0,\infty),[0,\infty))\colon f(t) \in \cdot~;~ \zeta(f)> t\big\} 
		= N_t \big\{\gamma\in\mathcal{I}_H\colon \|\gamma\|\in \cdot \big\}, \nonumber
	\end{align}
	where $N_t$ is the entrance law of $\Theta$ given in \eqref{eq:entrance}. This implies the desired convergence. 
	
	3. For any $t>0$, given $(\|C^{(n)}(r)\|, r\le 2 n t)$, we know from Lemma~\ref{prop:crp-ps} that the conditional distribution of $C^{(n)}(t) =\frac{1}{n} C(2 n t)$ is 
	the law of $\frac{1}{n} C_n$, where 
	$C_n$ is $\mathtt{oCRP}_{m}^{(\alpha)}(\theta_1,\theta_2)$ with $m= \|C(2 n t)\|$. 
	By Proposition~\ref{prop:crp-pdip}, we can strengthen \eqref{eq:cv-nu} to the following weak convergence on $\cI_H\setminus \{\emptyset\}$: 
	\[
	\bP \left(C^{(n)}(t)\in \cdot \,\middle|\, \zeta (C^{(n)} )> t\right) 
	\underset{n\to \infty}{\longrightarrow} N_t(\,\cdot \,|\, \cI_H\setminus\{\emptyset\}). 
	\]		
	Next, by the Markov property of a PCRP and the convergence result Theorem~\ref{thm:crp-ip-bis}, we deduce that, conditionally on $\{ \zeta (C^{(n)} )> t\}$, the process $(C^{(n)}(t+s), s\ge 0)$ converges weakly to an $\mathrm{SSIP}^{(\alpha)}(\theta_1,\theta_2)$-evolution $(\beta(s),\,s\ge 0)$ starting from  $\beta(0)\sim N_t(\,\cdot \,|\, \cI_H\setminus\{\emptyset\})$. By the description of $\Theta$ in  \eqref{eq:Theta:entrance}, this implies the convergence of finite-dimensional distributions for times 
	$t\le t_1<\cdots<t_k$. For $t>t_1$, this holds under $\mathrm{P}^{(n)}(\,\cdot\,|\,\zeta>t_1)$ and $\Theta(\,\cdot\,|\,\zeta>t_1)$ and can be further conditioned on $\{\zeta>t\}$, by 1.\ and 2.
	
	It remains to prove tightness. 
	For every $n\ge 1$, let $\tau_n$ be a stopping time with respect to the natural filtration of $C^{(n)}$ and $h_n$ a positive constant. Suppose that the sequence $\tau_n$ is bounded and $h_n\to 0$. 
	By Aldous's criterion \cite[Theorem 16.11]{Kallenberg}, it suffices to show that for any $\delta>0$,  
	\begin{equation}\label{eq:Aldous}
		\lim_{n\to\infty} \bP\left(d_H\left(C^{(n)}(\tau_n\!+\!h_n) , C^{(n)}(\tau_n) \right) >\delta \,\middle|\, \zeta (C^{(n)} )> t\right) =0. 
	\end{equation}
	By the total mass convergence in Proposition~\ref{prop:vague}, for any $\varepsilon>0$, there exists a constant $s>0$, such that 
	\begin{equation}\label{eq:Aldous-1}
		\limsup_{n\to\infty} \bP\left(\sup_{r\le 2s} \|C^{(n)}(r)\| >\delta/3 \,\middle|\, \zeta (C^{(n)} )> t\right)  \le \varepsilon. 
	\end{equation}
	Moreover, since $(C^{(n)}(s+z), z\ge 0)$ conditionally on $\{ \zeta (C^{(n)} )> s\}$ converges weakly to a continuous process, by \cite[Proposition~VI.3.26]{JacodShiryaev} we have for any $u>s$, 
	\begin{equation}\label{eq:Aldous-2}
		\lim_{n\to\infty} \bP\left(\sup_{r\in[s,u]} d_H\left(C^{(n)}(r\!+\!h_n) , C^{(n)}(r) \right) >\delta/3 \,\middle|\, \zeta (C^{(n)} )> t\right) =0.  
	\end{equation}
	Then \eqref{eq:Aldous} follows from \eqref{eq:Aldous-1} and \eqref{eq:Aldous-2}. This completes the proof. 	
\end{proof}

\subsection{The case $\alpha=0$}\label{sec:zero}

In a PCRP model with $\alpha=0$, the size of each table evolves according to an up-down chain $\pi(0)$ as in \eqref{eq:pi}, and new tables are only started to the left or to the right, but not between existing tables. 
We can hence build a $\mathrm{PCRP}^{(0)}(\theta_1,\theta_2)$ starting from  $(n_1, \ldots,n_k)\in \cC$ by a Poissonian construction. Specifically, consider independent $\mathbf{f}_i\sim \pi_{n_i}(0)$, $i\in [k]$, as size evolutions of the initial tables, $\mathbf{F}_1\sim \mathtt{PRM}(\theta_1\mathrm{Leb}\otimes \pi_1(0))$ whose atoms  describe the birth times and size evolutions of new tables added to the left, and $\mathbf{F}_2\sim \mathtt{PRM}(\theta_2\mathrm{Leb}\otimes \pi_1(0))$ for new tables added to the right. 
For $t\ge 0$, set
\[
C_1(t)= \Concat_{\text{atoms }(s,f)\text{ of }\mathbf{F}_1\downarrow: s\le t} \{(0, f(t-s))\},
\]
where $\downarrow$ means that the concatenation is from larger $s$ to smaller, 
\[
C_0(t)=  \Concat_{i\in [k]} \{(0,\mathbf{f}_i(t))\}, \text{ and } 
C_2(t)= \Concat_{\text{atoms }(s,f)\text{ of }\mathbf{F}_2:s\le t} \{(0, f(t-s))\}. 
\]
Then $(C(t)= C_1(t)\concat C_0(t)\concat C_2(t) ,t\ge 0)$ is  a  $\mathrm{PCRP}^{(0)}(\theta_1,\theta_2)$ starting from  $(n_1, \ldots,n_k)$.

\begin{proposition}
	The statement of Theorem~\ref{thm:crp-ip} still holds when $\alpha=0$. 
\end{proposition}
\begin{proof}
	We only prove this for the case when the initial state is $C^{(n)}(0)= \{(0,b^{(n)})\}$ with $\lim_{n\to \infty} b^{(n)}/n = b> 0$. 
	Then we can extend to a general initial state in the same way as we passed from Lemma~\ref{lem:cv-clade} to Theorem~\ref{thm:crp-ip-0}. 
	
	For each $n\ge 1$, we may assume $C^{(n)}$ is associated with $\mathbf{F}_1\sim \mathtt{PRM}(\theta_1\mathrm{Leb}\otimes \pi_1(0))$, $\mathbf{F}_2\sim \mathtt{PRM}(\theta_1\mathrm{Leb}\otimes \pi_1(0))$ and $\mathbf{f}^{(n)}_0\sim \pi_{b^{(n)}}(0)$. 
	Replacing each atom $\delta(s,f)$ of $\mathbf{F}_1$ and $\mathbf{F}_2$ by $\delta(s/2n,  f(2n\cdot)/n)$, we obtain $\mathbf{F}^{(n)}_1\sim \mathtt{PRM}(2n\theta_1\mathrm{Leb}\otimes \pi^{(n)}_1(0))$ and 
	$\mathbf{F}^{(n)}_2\sim \mathtt{PRM}(2n\theta_2\mathrm{Leb}\otimes \pi^{(n)}_1(0))$. Note that 
	$(\frac{1}{n}C^{(n)}(2nt),t\ge 0)$ is associated with $(\mathbf{F}^{(n)}_1, \mathbf{F}^{(n)}_2, \mathbf{f}^{(n)}(2n\cdot)/n)$. 
	
	Since Proposition~\ref{prop:vague} shows that $n\pi^{(n)}_1(0)\to \Lambda_{\mathtt{BESQ}}^{(0)}$, 
	by \cite[Theorem 4.11]{KallenbergRM}, we deduce that $\mathbf{F}^{(n)}_1$ and $\mathbf{F}^{(n)}_2$ converge in distribution respectively to  $\mathbf{F}^{(\infty)}_1\sim  \mathtt{PRM}(2\theta_1\mathrm{Leb}\otimes \Lambda_{\mathtt{BESQ}}^{(0)})$ and $\mathbf{F}^{(\infty)}_2\!\sim\!  \mathtt{PRM}(2\theta_2\mathrm{Leb}\otimes \Lambda_{\mathtt{BESQ}}^{(0)})$. 
	By Lemma~\ref{lem:ud}, $\mathbf{f}^{(n)}(2n\cdot)/n \!\to\! \mathbf{f}^{(\infty)}\!\sim\! \besq_b(0)$ in distribution. 
	
	As a result, we can deduce that $(\frac{1}{n}C^{(n)}(2nt),t\ge 0)$ converges to an $\cI_H$-valued process $(\beta(t),t\ge 0)$ defined by
	\begin{multline}\label{eq:ssip0}
		\beta(t)= 
		\bigg(\Concat_{\text{atoms }(s,f)\text{ of }\mathbf{F}^{(\infty)}_1\downarrow: s\le t}\!\! \{(0, f(t\!-\!s))\}\bigg)\\ 
		\concat \{(0,\mathbf{f}^{(\infty)}(t) )\} \concat
		\bigg( \Concat_{\text{atoms }(s,f)\text{ of }\mathbf{F}^{(\infty)}_2: s\le t}\!\! \{(0, f(t\!-\!s))\}\bigg). 
	\end{multline}	
	A rigorous argument can be made as in the proof of Lemma~\ref{lem:cv-clade}.
\end{proof}

The limiting process in \eqref{eq:ssip0} can be viewed as an $\mathrm{SSIP}^{(0)}(\theta_1,\theta_2)$-evolution, which is closely related to the construction of measure-valued processes in \cite{Shiga1990}. See also \cite[Section 7.1]{FVAT}.

\section{Applications}\label{sec:appl}

\subsection{Measure-valued processes}\label{sec:FV}
In \cite{FVAT}, we introduced a two-parameter family of superprocesses taking values in the space $(\cM^a,d_{\cM})$ of all purely atomic finite measures on a space of allelic types, say $[0,1]$. Here $d_{\cM}$ is the Prokhorov distance.   
Our construction is closely related to that of SSIP-evolutions, here extracting from scaffolding and spindles via the following \emph{superskewer} mapping. See Figure~\ref{fig:scaf-marks} on page \pageref{fig:scaf-marks} for an illustration.

\begin{definition}[Superskewer] \label{def:superskewer}
	Let $V= \sum_{i\in I} \delta(t_i,f_i,x_i)$ be a point measure on $\bR\times \cE\times [0,1]$ and $X$ a c\`adl\`ag process such that 
	\[\sum_{\Delta X(t)> 0} \delta(t, \Delta X(t)) = \sum_{i\in I} \delta(t_i, \zeta(f_i)).\] 
	The \emph{superskewer} of the pair $(V,X)$ at level $y$ is the atomic measure
	\begin{equation}
		\sskewer(y,V,X) :=    \sum_{i\in I\colon X(t_i-)\le y <X(t_i)}  f_i\big( y- X(t_i-) \big) \delta(x_i).  
	\end{equation}
\end{definition}

For $\alpha\in (0,1)$, $\theta\ge 0$, recall the scaffolding-and-spindles construction of an $\mathrm{SSIP}^{(\alpha)}(\theta)$-evolution starting from $\gamma\in \cI_H$; in particular, for each $U\in \gamma$, there is an initial spindle $\mathbf{f}_U\sim \besq_{\mathrm{Leb}(U)} (-2\alpha)$. 
For any collection $x_U\in[0,1]$, $U\in \gamma$, we can construct a self-similar superprocess $\mathrm{SSSP}^{(\alpha)}(\theta)$ starting from $\pi= \sum_{U\in \gamma} \mathrm{Leb}(U) \delta(x_U)$ as follows. 
We mark each initial spindle $\mathbf{f}_U$ by the allelic type $x_U$ and all other spindles in the construction by i.i.d.\ uniform random variables on $[0,1]$. 
Then we obtain the desired superprocess by repeating the construction of an $\mathrm{SSIP}^{(\alpha)}(\theta)$-evolution in Definitions \ref{defn:ip0}--\ref{defn:alphatheta}, with skewer replaced by superskewer, 
and concatenation replaced by addition. We refer to \cite{FVAT} for more details. 

We often write $\pi\in \cM^a$ in \em canonical representation \em $\pi = \sum_{i\ge 1} b_i\delta(x_i)$ with  
$b_1\ge b_2\ge \cdots$ and $x_i <x_{i+1}$ if $b_i=b_{i+1}$. We write $\|\pi\|:=\pi([0,1])=\sum_{i\ge 1}b_i$ for the \em total mass \em of $\pi$. 
\begin{definition}\label{defn:sssp}
	Let $\alpha\in (0,1)$ and $\theta\in [-\alpha,0)$. 
	We define a process $(\pi(t),\,t\ge 0)$ starting from $\pi(0)\in \cM^a$ by the following construction. 
	\begin{itemize}
		\item Set $T_0=0$. For $\pi(0) = \sum_{i\ge 1} b_i\delta(x_i)$ in canonical representation, 
		consider $\mathbf{x}^{(0)}:= x_1$ and independent 
		$\ff^{(0)}\sim \besq_{b_1}(-2\alpha)$ and 
		$\boldsymbol{\lambda}^{(0)}\sim \mathrm{SSSP}^{(\alpha)}(\theta\!+\!\alpha)$ starting from $\sum_{i\ge 2} b_i\delta(x_i)$. 
		\item For $k\ge 1$, suppose by induction we have obtained $(\boldsymbol{\lambda}^{(i)},\ff^{(i)},\mathbf{x}^{(i)}, T_i)_{0\le i\le k-1}$. 
		Then we set $T_{k}= T_{k-1} +\zeta(\ff^{(k-1)}) $ and 
		\[
		\pi(t) = \lambda^{(k-1)}(t\!-\!T_{k-1})+\ff^{(k-1)}(t\!-\!T_{k-1}) \delta (\mathbf{x}^{(k-1)}),\qquad t\in  [T_{k-1}, T_{k}]. 
		\]
		Write $\pi(T_{k}) = \sum_{i\ge 1} b^{(k)}_i\delta(x^{(k)}_i)$, with $b^{(k)}_1\ge b^{(k)}_2\ge \cdots$, for its canonical representation. 
		Conditionally on the history, construct independent 	$\boldsymbol{\lambda}^{(k)}\sim \mathrm{SSSP}^{(\alpha)}(\theta\!+\!\alpha)$ starting from $\sum_{i\ge 2} b^{(k)}_i\delta(x^{(k)}_i)$ and $\ff^{(k)}\sim \besq_{b^{(k)}_1}(-2\alpha)$. Let $\mathbf{x}^{(k)}= x^{(k)}_1$. 
		\item Let $T_{\infty}= \lim_{k\to\infty} T_k$ and $\pi(t) = 0$ for $t\ge T_{\infty}$. 
	\end{itemize}
	The process $\boldsymbol{\pi}:=(\pi(t),t\ge 0)$  is called an \emph{$(\alpha,\theta)$ self-similar superprocess}, $\mathrm{SSSP}^{(\alpha)}(\theta)$. 
\end{definition}

For any $\pi(0)=\sum_{i\ge 1}b_i\delta(x_i)\in\cM^a$ consider $\beta(0)=\{(s(i-1),s(i)),\,i\ge 1\}\in\mathcal{I}_H$, where 
$s(i)=b_1+\cdots+b_i$, $i\ge 0$. Consider an $\mathrm{SSIP}^{(\alpha)}(\theta\!+\!\alpha, 0)$-evolution starting from $\beta(0)$, built 
in Definition~\ref{defn:ipe} and use notation therein. 
As illustrated in Figure~\ref{fig:scaf-marks}, we may assume that each interval partition evolution is obtained from the skewer of marked spindles. Therefore, we can couple each $\mathrm{SSIP}^{(\alpha)}(\theta\!+\!\alpha)$-evolution $\boldsymbol{\gamma}_1^{(k)}$ with 
an  $\boldsymbol{\lambda}_1^{(k)}\sim \mathrm{SSSP}^{(\alpha)}(\theta\!+\!\alpha)$, such that the atom sizes of the latter correspond to the interval lengths of the former. 
Similarly, each $\mathrm{SSIP}^{(\alpha)}(0)$-evolution $\boldsymbol{\gamma}_2^{(k)}$ corresponds to a $\boldsymbol{\lambda}_2^{(k)}\sim \mathrm{SSSP}^{(\alpha)}(0)$.   
Then $\boldsymbol{\lambda}^{(k)}= \boldsymbol{\lambda}_1^{(k)}\!+\! \boldsymbol{\lambda}_2^{(k)}$ is an $\mathrm{SSSP}^{(\alpha)}(\theta\!+\!\alpha)$ by definition.  
Let $\ff^{(k)}$ be the \emph{middle} (marked) spindle in Definition~\ref{defn:ipe}, which is  a $\besq(-2\alpha)$, and $\mathbf{x}^{(k)}$ be its type. 
In this way, we obtain a sequence $(\boldsymbol{\lambda}^{(k)},\ff^{(k)},\mathbf{x}^{(k)})_{k\ge 0}$ and thus $\boldsymbol{\pi}=(\pi(t),\,t\ge 0)\sim\mathrm{SSSP}^{(\alpha)}(\theta)$ as in Definition~\ref{defn:sssp}. It is coupled with $\boldsymbol{\beta}=(\beta(t),\,t\ge 0)\sim\mathrm{SSIP}^{(\alpha)}(\theta\!+\!\alpha, 0)$ as in Definition \ref{defn:ipe}, such that atom sizes and interval lengths are matched, and the renaissance level $(T_k)_{k\ge 0}$ are exactly the same.

The next theorem extends \cite[Theorem 1.2]{FVAT} to $\theta\in [-\alpha,0)$. 
\begin{theorem}
	Let $\alpha\!\in\! (0,1)$, $\theta\!\in\! [-\alpha,0)$. An $\mathrm{SSSP}^{(\alpha)}(\theta)$ 
	is a Hunt process with $\besq(2\theta)$ total mass, paths that are total-variation continuous, and its finite-dimensional marginals are continuous along sequences of initial states that converge in total variation. 
\end{theorem}
\begin{proof} For any $\pi(0)=\sum_{i\ge 1}b_i\delta(x_i)\in\cM^a$ consider 
	$(\pi(t),\,t\ge 0)\sim\mathrm{SSSP}^{(\alpha)}(\theta)$ and $(\beta(t),\,t\ge 0)\sim\mathrm{SSIP}^{(\alpha)}(\theta\!+\!\alpha, 0)$
	coupled, as above. By \cite[Theorem 1.4]{ShiWinkel-1}, their (identical) total mass processes are $\besq(2\theta)$. 
	Moreover, by this coupling and Theorem~\ref{thm:IvaluedMP}~(\ref{item:levels}),
	\begin{equation}\label{eq:sssp-T-infty}
		T_{\infty}<\infty, ~\text{and}~\lim_{t\to T_\infty} \|\pi(t)\|=0 \quad \text{a.s.,}
	\end{equation}
	which implies the path-continuity at $T_{\infty}$. Since an $\mathrm{SSSP}^{(\alpha)}(\theta\!+\!\alpha)$ has continuous paths in total variation \cite[Theorem 1.2 and Corollary~5.5]{FVAT}, we conclude the path-continuity of $\mathrm{SSSP}^{(\alpha)}(\theta)$ by the construction in Definition~\ref{defn:sssp}, both in the Prokhorov sense and in the stronger total variation sense.  
	
	To prove the Hunt property, we adapt the proof of \cite[Theorem 1.4]{ShiWinkel-1} and apply Dynkin's criterion to a richer Markov process that records more information from the construction. Specifically, in the setting of Definition~\ref{defn:sssp}, let 
	\[
	(\lambda(t) ,\mathbf{f}(t),\mathbf{x}(t) ):=  \left(\lambda^{(k-1)}(t\!-\!T_{k-1}), \mathbf{f}^{(k-1)}(t\!-\!T_{k-1}), \mathbf{x}^{(k-1)}\right),\quad t\in  [T_{k-1}, T_{k}), k\ge 1.  
	\]
	and $(\lambda(t),\mathbf{f}(t),\mathbf{x}(t)) :=(0,0,0)$ for $t\ge T_{\infty}$. We shall refer to this process as a \emph{triple-valued 
		$\mathrm{SSSP}^{(\alpha)}(\theta)$} with values in $\widetilde{\mathcal{J}}:=(\cM^a\times(0,\infty)\times[0,1])\cup\{(0,0,0)\}$. 
	This process induces the $\cM^a$-valued $\mathrm{SSSP}^{(\alpha)}(\theta)$ as $\pi(t)=\lambda(t)+\mathbf{f}(t)\delta(\mathbf{x}(t))$. 
	Since each $(\boldsymbol{\lambda}^{(k)},\ff^{(k)},\mathbf{x}^{(k)})$ is Hunt and is built conditionally on the previous ones according to a probability kernel, then $(\lambda(t),\mathbf{f}(t), \mathbf{x}(t))$, $t\ge 0$, is a Borel right Markov process by 
	\cite[Th\'eor\`eme II 3.18]{Bec07}. 
	
	To deduce that $(\pi(t),t\ge 0)$ is Borel right Markovian, and hence Hunt by path-continuity, we use Dynkin's criterion \cite[Theorem~2 and the remark below]{RogersPitman}. Specifically, consider any 
	$(\lambda_1(0),\mathbf{f}_1(0),\mathbf{x}_1(0)), (\lambda_2(0),\mathbf{f}_2(0),\mathbf{x}_2(0))\in\widetilde{\mathcal{J}}$ with
	\[
	\lambda_1(0)+\mathbf{f}_1(0)\delta(\mathbf{x}_1(0))=\lambda_2(0)+\mathbf{f}_2(0)\delta(\mathbf{x}_2(0)).
	\]
	It suffices to couple 
	triple-valued $\mathrm{SSSP}^{(\alpha)}(\theta)$ from these two initial states whose induced $\cM^a$-valued $\mathrm{SSSP}^{(\alpha)}(\theta)$ coincide. 
	
	First note that (unless they are equal) the initial states are such that for $t=0$ and $i=1,2$, 
	\begin{equation}\label{eq:mvcoupling}\lambda_1(t)=\mu(t)+\mathbf{f}_2(t)\delta(\mathbf{x}_2(t))\quad\mbox{and}\quad
		\lambda_2(t)=\mu(t)+\mathbf{f}_1(t)\delta(\mathbf{x}_1(t))
	\end{equation} 
	for some $\mu(t)\in\mathcal{M}^a$. We follow similar arguments as in the proof of \cite[Lemma 3.8]{ShiWinkel-1}, via a quintuple-valued 
	process $(\mu(t),\mathbf{f}_1(t),\mathbf{x}_1(t),\mathbf{f}_2(t),\mathbf{x}_2(t))$, $0\le t<S_N$, that captures two marked types. Let $S_0:=0$. For $j\ge 0$, suppose we have constructed the process on $[0,S_j]$. 
	\begin{itemize}
		\item Conditionally on the history, consider an ${\rm SSSP}^{(\alpha)}(\theta\!+\!2\alpha)$-evolution 
		$\boldsymbol{\mu}^{(j)}$ starting from $\mu(S_j)$, and $\ff^{(j)}_i\sim{\tt BESQ}_{\mathbf{f}_i(S_j)}(-2\alpha)$, $i=1,2$, independent of each other. Let 
		$\Delta_j:=\min\{\zeta(\ff_1^{(j)}),\zeta(\ff_2^{(j)})\}$ and $S_{j+1}:=S_j+\Delta_j$. For $t\in[S_j,S_{j+1})$, define
		\[
		\left(\mu(t),\mathbf{f}_1(t),\mathbf{x}_1(t),\mathbf{f}_2(t),\mathbf{x}_2(t)\right)
		:=\Big(\!\mu^{(j)}(t\!-\!S_j),\ff_1^{(j)}(t\!-\!S_j),\mathbf{x}_1(S_j),\ff_2^{(j)}(t\!-\!S_j),\mathbf{x}_2(S_j)\!\Big).
		\]
		\item Say $\Delta_j=\zeta(\ff_1^{(j)})$. If $\ff_{2}^{(j)}(\Delta_j)$ exceeds the size of the largest atom in $\mu^{(j)}(\Delta_j)$, let $N=j+1$. The construction is complete.
		Otherwise, let $(\mathbf{f}_2(S_{j+1}),\mathbf{x}_2(S_{j+1})):= (\ff_2^{(j)}(\Delta_j),\mathbf{x}_2(S_j))$ and 
		decompose $\mu^{(j)}(\Delta_j)= \mu(S_{j+1})+\mathbf{f}_1(S_{j+1})\delta(\mathbf{x}_1(S_{j+1}))$ by identifying its largest atom, giving rise to the five components.   
		Similar operations apply when $\Delta_j=\zeta(\ff_2^{(j)})$.
	\end{itemize}
	For $t\in[0,S_N)$, define $\lambda_i(t)$, $i=1,2$, by \eqref{eq:mvcoupling}. In general, we may have $N\in\mathbb{N}\cup\{\infty\}$. On the 
	event $\{N<\infty\}$, we further continue with the same triple-valued ${\rm SSSP}^{(\alpha)}(\theta)$ starting from the terminal value  
	$(\mu^{(N)}(\Delta_{N-1}),\mathbf{f}_i^{(N)}(\Delta_{N-1}),\mathbf{x}_i(\Delta_N))$, with $i\in \{1,2\}$ being the index such that $\mathbf{f}_i^{(N)}(\Delta_{N-1})>0$.
	
	By \cite[Corollary~5.11 and remark below]{FVAT} and the 
	strong Markov property of these processes applied at the stopping times $S_j$, we obtain two coupled triple-valued ${\rm SSSP}^{(\alpha)}(\theta)$, which induce the same $\cM^a$-valued ${\rm SSSP}^{(\alpha)}(\theta)$, as required. 
	Indeed, the construction of these two processes is clearly complete on $\{N<\infty\}$. 
	On $\{N=\infty\}$, by \eqref{eq:sssp-T-infty} one has $\{S_\infty<\infty\}$ and the total mass tends to zero as $t\uparrow S_\infty$, and hence the construction is also finished. 
	
	For the continuity in the initial state, suppose that 
	$\pi_n(0)=\mathbf{f}_n^{(0)}(0)\delta(x_1)+\lambda_n^{(0)}(0)\rightarrow\pi(0)=\mathbf{f}^{(0)}(0)\delta(x_1)+\lambda^{(0)}(0)$ in total variation. 
	First note that a slight variation of the proof of \cite[Proposition 3.6]{FVAT} allows to couple $\boldsymbol{\lambda}_n^{(0)}$ and 
	$\boldsymbol{\lambda}^{(0)}$ so that $\lambda_n^{(0)}(t_n)\rightarrow\lambda^{(0)}(t)$ in total variation a.s., for any fixed sequence $t_n\rightarrow t$.
	Also coupling $\mathbf{f}_n^{(0)}$ and $\mathbf{f}^{(0)}$, we can apply this on $\{\zeta(\mathbf{f}^{(0)})>t\}$ to obtain $\pi_n(t)\rightarrow\pi(t)$ for any fixed $t$, and on
	$\{\zeta(\mathbf{f}^{(0)})<t\}$ to obtain $\zeta(\mathbf{f}^{(0)}_n)\rightarrow\zeta(\mathbf{f}^{(0)})$ and $\pi_n(\zeta(\mathbf{f}^{(0)}_n))\rightarrow\pi(\zeta(\mathbf{f}^{(0)}))$ in total variation a.s.. By induction, this establishes the convergence of one-dimensional marginals on
	$\{T_\infty\!>\!t\}$, and trivially on $\{T_\infty\!<\!t\}=\{\pi(t)\!=\!0\}$. A further induction extends this to finite-dimensional marginals. 
\end{proof}

For $\alpha\in (0,1)$, $\theta\in [-\alpha,0)$, let $(B_1,B_2,\ldots)\sim \mathtt{PD}^{(\alpha)}( \theta)$ be a Poisson--Dirichlet sequence in the Kingman simplex and  $(X_i, i\ge 1)$ i.i.d.\ uniform random variables on $[0,1]$, further independent of $(B_1,B_2,\ldots)$. Define $\mathtt{PDRM}^{(\alpha)}( \theta)$ to be the distribution of the random probability measure $\overline{\pi}:= \sum_{i\ge 1} B_i \delta (X_i)$ on $\cM^a_1:= \{\mu\in \cM^a\colon \|\mu\|=1 \}$.  If $\theta=-\alpha$, then  $\overline{\pi}=\delta(X_1)$. 
\begin{proposition}\label{prop:ps-SSSP}
	Let $(Z(t),t\ge 0)$ be a $\besq (2 \theta)$ killed at zero with $Z(0)>0$, independent of $\overline\pi\sim \mathtt{PDRM}^{(\alpha)}( \theta)$. 
	Let $(\pi(t), t\ge 0)$ be an $\mathrm{SSSP}^{(\alpha)}(\theta)$  starting from $Z(0)   \overline\pi$. 
	Fix any $t\ge 0$, then $\pi(t)$ has the same distribution as $Z(t) \overline\pi$. 
\end{proposition}
\begin{proof}
	By the coupling between an $\mathrm{SSIP}^{(\alpha)}(\theta\!+\!\alpha,0)$-evolution and an $\mathrm{SSSP}^{(\alpha)}(\theta)$ mentioned above, the claim follows from Proposition~\ref{prop:ps-theta1theta2} and the definition of $\mathtt{PDRM}^{(\alpha)}( \theta)$. 
\end{proof}

For $\theta\ge -\alpha$, 
let $\boldsymbol{\pi}:= (\pi(t),\, t\ge 0)$ be an $\mathrm{SSSP}^{(\alpha)} (\theta)$ starting from $\mu\in \cM^a_{1}$. Define an associated $\cM^a_{1}$-valued process via the de-Poissonisation as in Definition~\ref{defn:dePoi}:
\[
\overline{\pi}(u):= \big\| \pi(\tau_{\boldsymbol{\pi}}(u)) \big\|^{-1}  \pi(\tau_{\boldsymbol{\pi}}(u)),\qquad u\ge 0,
\]
where $\tau_{\boldsymbol{\pi}}(u):= \inf \big\{ t\ge 0\colon \int_0^t \|\pi(s)\|^{-1} d s>u \big\}$.  
The process $(\overline{\pi}(u),u\ge 0)$ on $\cM^a_{1}$  
is called a \emph{Fleming--Viot $(\alpha,\theta)$-process} starting from $\mu$, denoted by $\mathrm{FV}^{(\alpha)} (\theta)$. 

Using Proposition~\ref{prop:ps-SSSP}, we easily deduce the following statement by the same arguments as in the proof of \cite[Theorem~1.7]{FVAT}, 
extending \cite[Theorem~1.7]{FVAT} to the range $\theta\in [-\alpha,0)$.   
\begin{theorem}\label{thm:FV}
	Let $\alpha\in(0,1)$ and $\theta\ge -\alpha$. 
	A $\mathrm{FV}^{(\alpha)} (\theta)$-evolution is a total-variation path-continuous Hunt process on 
	$(\cM^a_{1},d_{\cM})$ and has a stationary distribution
	$\mathtt{PDRM}^{(\alpha)}( \theta)$. 
\end{theorem}

\subsection{Fragmenting interval partitions}\label{sec:nested1}
We define a fragmentation operator for interval partitions, which is associated with the random interval partition $\mathtt{PDIP}^{(\alpha)}(\theta_1, \theta_2)$ defined in Section~\ref{sec:PDIP}. Fragmentation theory has been extensively studied in the literature; see e.g.\ \cite{Bertoin06}.
\begin{definition}[A fragmentation operator]
	Let $\alpha\in (0,1)$ and $\theta_1,\theta_2\ge 0$. 
	We define a Markov transition kernel $\mathrm{Frag}:= \mathrm{Frag}^{(\alpha)}(\theta_1,\theta_2)$ on $\cI_{H}$ as follows. 
	Let $(\gamma_i)_{i\in\bN}$ be i.i.d.\@  with distribution 
	$\mathtt{PDIP}^{(\alpha)}(\theta_1, \theta_2)$.  
	For 
	$\beta = \{ (a_i,b_i), i\in \bN\}\!\in\! \cI_H$, with $(a_i,b_i)_{i\in \bN}$ enumerated in decreasing order of length, we define $\mathrm{Frag}(\beta, \cdot)$ to be the law of the interval partition obtained from $\beta$ by splitting each $(a_i,b_i)$ according to the interval partition $\gamma_i$, i.e.\@ 
	\[\{  (a_i + (b_i-a_i)l , a_i+ (b_i-a_i) r) \colon  i\in \bN, (a_i,b_i) \in \beta, (l,r)\in \gamma_i \}. 
	\]
\end{definition}

\begin{lemma}\label{lem:cf-pdip}
	For $\alpha, \bar\alpha\in (0,1)$ and $\theta_1,\theta_2, \bar\theta_1, \bar\theta_2\ge 0$, 
	suppose that  
	\[
	\theta_1+\theta_2 +\bar\alpha=\alpha .
	\]
	Let $\beta_c\sim \pdip^{(\bar\alpha)} (\bar\theta_1, \bar\theta_2)$ and $\beta_f$ a random interval partition whose regular conditional distribution given $\beta_c$ is $\mathrm{Frag}^{(\alpha)}(\theta_1,\theta_2)$. 
	Then $\beta_f\sim \pdip^{(\alpha)} (\theta_1\!+\!\bar\theta_1, \theta_2\!+\!\bar\theta_2)$.  
\end{lemma}
A similar result for the particular case with parameter $\bar{\alpha}=\bar{\theta}_2=0$, $\theta_1=0$ and $\theta_2=\alpha$ is included in \cite[Theorem~8.3]{GnedPitm05}. Lemma~\ref{lem:cf-pdip} is also an analogous result of \cite[Theorem~5.23]{CSP} for $\mathtt{PD}(\alpha,\theta)$ on the Kingman simplex. 
\begin{figure}[t]
	\centering
	\includegraphics[width=\linewidth]{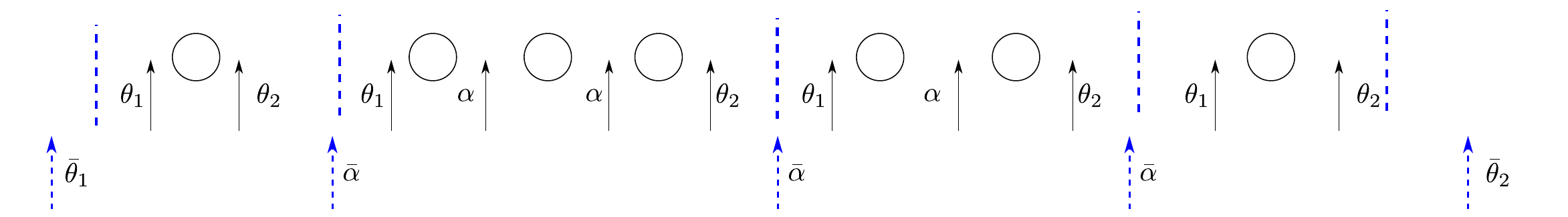}
	\caption{Clusters are divided by dashed lines. A new customer starts a new table in an existing cluster (solid arrow) or a new cluster (dashed arrow), with probability proportional to the indicated weights.
		In the continuous-time model studied in Section~\ref{sec:nestedPCRP}, the weights correspond to the rates ar which customers arrive. 
	}
	\label{fig:nested_CRP}
\end{figure}

To prove Lemma~\ref{lem:cf-pdip}, we now consider a pair of 
\emph{nested} ordered Chinese restaurant processes $(C_c(n), C_f(n))_{n\in \bN}$. 
The coarse one $C_c\sim \mathrm{oCRP}^{(\bar{\alpha})} (\bar{\theta}_1, \bar{\theta}_2)$ describes the arrangement of customers in a sequence of ordered \emph{clusters}. 
We next obtain a composition of each cluster of customers by further seating them at ordered \emph{tables}, according to the $(\alpha, \theta_1,\theta_2)$-seating rule. 
These compositions are concatenated according to the cluster order, forming the fine process $C_f$. 
Then, as illustrated in Figure~\ref{fig:nested_CRP}, we can easily to check that $C_f\sim \mathrm{oCRP}^{(\alpha)} (\theta_1\!+\!\bar{\theta}_1, \theta_2\!+\!\bar{\theta}_2)$, due to the identity $\theta_1\!+\!\theta_2\!+\!\bar\alpha = \alpha$. 
Nested (unordered) Chinese restaurant processes have been widely applied in nonparametric Bayesian analysis of the problem of learning topic hierarchies from data, see e.g.\ \cite{Blei10}.

Lemma~\ref{lem:cf-pdip} follows immediately from the following convergence result. 

\begin{lemma}[Convergence of nested oCRP]\label{lem:cv-oCRP}
	Let $(C_c(n),C_f(n))_{n\in \bN}$ be a pair of nested ordered Chinese restaurant processes constructed as above. 
	Then $(n^{-1} C_c(n), n^{-1}  C_f(n))$ converges a.s.\@ to a limit $(\beta_c, \beta_f)$ for the  metric $d_H$ as $n\to \infty$; furthermore, we have $\beta_c\sim \mathtt{PDIP}^{(\bar\alpha)}(\bar\theta_1, \bar\theta_2)$, $\beta_f\sim\pdip^{(\alpha)} (\theta_1\!+\!\bar\theta_1, \theta_2\!+\!\bar\theta_2)$, and the regular conditional distribution of $\beta_f$ given $\beta_c$ is  $\mathrm{Frag}^{(\alpha)}(\theta_1,\theta_2)$.
\end{lemma}

\begin{proof}
	By Proposition~\ref{prop:crp-pdip}, we immediately deduce that $(n^{-1} C_c(n), n^{-1}  C_f(n))$ converges a.s.\@ to a limit $(\beta_c, \beta_f)$. It remains to determine the joint distribution of the limit. 
	
	Consider a sequence of i.i.d.\@ $\gamma_i\sim \pdip^{(\alpha)}(\theta_1, \theta_2)$, $i\ge 1$ and an independent $\beta'_c:= \{(a_i,b_i), i\in \bN\} \sim \mathtt{PDIP}^{(\bar\alpha)}(\bar\theta_1, \bar\theta_2)$. Let $\beta'_f$ be obtained from $\beta'_c$ by splitting each $(a_i,b_i)$ according to  $\gamma_i$, then the regular conditional distribution of $\beta'_f$ given $\beta'_c$ is $\mathrm{Frag}^{(\alpha)}(\theta_1,\theta_2)$.
	We will show that $(\beta_c, \beta_f)\ed (\beta'_c, \beta'_f)$. 
	
	To this end, apply the paintbox construction described before Proposition~\ref{prop:ocrp-pdip} to the nested $\beta'_c$ and $\beta'_f$, by using the same sequence of i.i.d.\@ uniform random variables $(Z_j, j\in \bN)$ on $[0,1]$. 
	For each $n\in \bN$, let $C^*_c(n)$ and $C^*_f(n)$ be the compositions of the set $[n]$ obtained as in \eqref{eqn:C*}, associated with $\beta'_c$ and $\beta'_f$ respectively. 
	Write $(C'_c(n),C'_f(n))$ for the integer compositions associated with $(C^*_c(n),C^*_f(n))$, then $(n^{-1}C'_c(n),n^{-1}C'_f(n))$ converges a.s.\ to $(\beta'_c,\beta'_f)$ by \cite[Theorem 11]{Gnedin97}.

	Note that each $(a_i, b_i)\in \beta'_c$ corresponds to a cluster of customers in $C^*_c(n)$, which are further divided into ordered tables in $C^*_f(n)$. This procedure can be understood as a paintbox construction, independent of other clusters, by using $\gamma_i\sim \pdip^{(\alpha)}(\theta_1, \theta_2)$ and i.i.d.\@ uniform random variables $\{(Z_j - a_i)/(b_i-a_i)\colon  Z_j \in (a_i, b_i), j\in \bN\}$ on $[0,1]$. 
	By Proposition~\ref{prop:ocrp-pdip}, it has the same effect as an $\mathrm{oCRP}^{(\alpha)}(\theta_1,\theta_2)$. For each $n\in \bN$, as we readily have $C'_c(n)\ed C_c(n)$ by Proposition~\ref{prop:ocrp-pdip}, it follows that $(C'_c(n),C'_f(n))\ed (C_c(n),C_f(n))$. As a result, we deduce that the limits also have the same law, i.e.\ $(\beta_c, \beta_f)\ed (\beta'_c, \beta'_f)$.
\end{proof}

\subsection{Coarse-fine interval partition evolutions}\label{sec:nested2}
We consider the space 
\[
\cI^2_{\mathrm{nest}}:= 
\big\{(\gamma_c,\gamma_f) \in   \cI_{H}\times \cI_{H} \colon G(\gamma_c)\subseteq G(\gamma_f) 
,\| \gamma_c\| = \|\gamma_f \|  
\big\}. 
\]
In other words, for each element $(\gamma_c,\gamma_f)$ in this space, the interval partition $\gamma_f$ is a refinement of $\gamma_c$ such that each interval $U\in \gamma_c$ 
is further split into intervals in $\gamma_f$, forming an interval partition $\gamma_f\big|_U$ of $[\inf U, \sup U]$. We also define the shifted interval partition $\gamma_{f}\big|^{\leftarrow}_{U}:= \{ (a,b)\colon (a+\inf U, b+\inf U) \in \gamma_f\big|_U \}$ of  $[0,\mathrm{Leb}(U)]$ and note that $\gamma_{f}\big|^{\leftarrow}_U\in\cI_H$.  
We equip $\cI^2_{\mathrm{nest}}$ with the product metric
\[d_{H}^2 ((\gamma_c,\gamma_f), (\gamma'_c,\gamma'_f))= d_{H}(\gamma_c,\gamma'_c)+d_{H}(\gamma_f,\gamma'_f).\]

\begin{lemma}\label{lem:nested-cv}
	For each $n\ge 1$, let $(\beta_n, \gamma_n)\in \cI^2_{\mathrm{nest}}$. 
	Suppose that  $(\beta_n, \gamma_n)$ converges to $(\beta_{\infty},\gamma_{\infty})$ under the product metric $d^2_H$. Then $(\beta_{\infty}, \gamma_{\infty})\in \cI^2_{\mathrm{nest}}$. 
\end{lemma}
\begin{proof}
	This requires us to prove $G(\beta_{\infty})\subseteq G(\gamma_{\infty})$. 
	As $G(\gamma_{\infty})$ is closed, it is equivalent to show that, for any $x\in G(\beta_{\infty})$, 
	the distance $d(x, G(\gamma_{\infty}))$ from $x$ to the set $G(\gamma_{\infty})$ is zero. 
	For any $y_n\in G(\beta_n)\subseteq G(\gamma_{n})$, we have 
	$
	d(x, G(\gamma_{\infty}))\le d(x, y_n) 
	+ d_H(\gamma_{n} ,\gamma_{\infty}). 
	$
	It follows that                                                                                       
	$
	d(x, G(\gamma_{\infty})) \le \inf_{y_n\in G(\beta_n)} d(x, y_n) + d_H(\gamma_{n} , \gamma_{\infty})
	\le d_H(\beta_{\infty} , \beta_{n}) + d_H(\gamma_{n}, \gamma_{\infty}).
	$
	As $n\to \infty$, the right-hand side converges to zero. 
	So we conclude that $d(x, G(\gamma_{\infty}))=0$ for every $x\in G(\beta_{\infty})$, 
	completing the proof. 
\end{proof}

We shall now construct a coarse-fine interval partition evolution in the space $\cI^2_{\mathrm{nest}}$. 
To this end, let us first extend the scaffolding-and-spindles construction in Section~\ref{sec:pre} to the setting where each spindle is an interval-partition-valued excursion. 
Denote by $\cE_{\cI}$ the space of continuous $\cI_H$-valued excursions. 
Given a point measure $W$ on $\bR_+ \times \cE_{\cI}$ and a scafolding function 
$X \colon \bR_+ \to \bR$, we define the following variables, if they are well-defined. 
The \emph{coarse skewer} of $(W,X)$ at level $y\in \bR$ is the interval partition
\begin{equation*}
	\cskewer(y,W, X):= \left\{ \big(M_{W,X}^y (t-), M_{W,X}^y (t) \big) \colon
	t\ge 0, M_{W,X}^y (t-)< M_{W,X}^y (t)
	\right\},
\end{equation*}
where
$
M_{W,X}^y (t):= \int_{[0,t]\times \cE_{\cI}}  \big\|\gamma\big( y -X(s-) \big)\big\| W(ds,d\boldsymbol{\gamma})$,  $t\ge 0. 
$
Let $\cskewerbar(W,X):= (\cskewer(y,W,X), y\ge 0)$. 
The \emph{fine skewer} of $(W,X)$ at level $y\in \bR$ is the interval partition
\begin{equation*}
	\fskewer(y,W,X):= 
	\Concat_{\text{points }(t, \boldsymbol{\gamma}_t) \text{ of } W\colon M_{W,X}^y (t-)< M_{W,X}^y (t)} \gamma_t\big( y -X(t-) \big). 
\end{equation*}
Let $\fskewerbar(W,X):= (\fskewer(y,W, X), y\ge 0)$.

Let $\theta_1, \theta_2 \ge 0$. Suppose that $\theta=\theta_1+\theta_2-\alpha\in [-\alpha,0)$, then we have an  $\mathrm{SSIP}^{(\alpha)}(\theta_1, \theta_2)$-excursion measure $\Theta:= \Theta^{(\alpha)}(\theta_1, \theta_2)$ defined as in Section~\ref{sec:exc}. 
Write  $\bar\alpha:= -\theta \in (0,\alpha]$ and let $\fW$ be  a Poisson random measure on $\bR_+ \times \cE_{\cI}$ with intensity $c_{\bar\alpha}\Leb \otimes \Theta$, where 
$c_{\bar\alpha}:=2 \bar\alpha(1+\bar\alpha)/\Gamma(1-\bar\alpha)$. 
We pair $\fW$ with a (\emph{coarse}) scaffolding $(\xi^{(\bar\alpha)}_{\fW}(t), t\ge 0)$ defined by
\begin{equation}\label{eq:scaffoldingW}
	\xi^{(\bar\alpha)}_{\fW}(t):=  \lim_{z\downarrow 0} \left( 
	\int_{[0,t]\times \{\boldsymbol{\gamma}\in \cE_{\cI}\colon \zeta(\boldsymbol{\gamma})>z\}} \zeta(\boldsymbol{\gamma})  \fW(ds,d\boldsymbol{\gamma}) - \frac{(1+\bar\alpha)t}{(2z)^{\bar\alpha}\Gamma(1-\bar\alpha)\Gamma(1+\bar\alpha)}
	\right). 
\end{equation}
This is a spectrally positive stable L\'evy process of index $1+\bar\alpha$. 
Let  $\boldsymbol{\beta}$ be an $\mathrm{SSIP}_{\!\dagger}^{(\alpha)} (\theta_1,\theta_2)$-evolution starting from $\gamma_0\in \cI_{H}$ with first hitting time $\zeta(\boldsymbol{\beta})$ of $\emptyset$.  
We define by $\mathbf{Q}_{\gamma_0}^{(\alpha)}(\theta_1, \theta_2)$  the law of the following a random point measure on $[0,\infty)\times \cE_{\cI}$: 
\begin{equation}\label{eq:cladeW}
	\delta(0,\boldsymbol{\beta})  + \fW\big|_{(0, T_{-\zeta(\boldsymbol{\beta})}]\times \cE_{\cI}}, 
	\quad\text{where}~ T_{-y}:= \inf\big\{t\ge 0 \colon \xi^{(\bar\alpha)}_{\fW}(t)= -y\big\}. 
\end{equation}

\begin{definition}[Coarse-fine SSIP-evolutions]\label{defn:cfIP-0}
	$\!\!\!\!\!\!$ Let $\alpha\!\in\! (0,1)$, $\theta_1,\theta_2\!\ge\! 0$ with $\theta_1\!+\!\theta_2\!<\!\alpha$. Let $\bar\alpha:= \alpha -\theta_1-\theta_2 \in (0,\alpha]$. 
	For $(\gamma_c,\gamma_f) \in \cI^2_{\mathrm{nest}}$, let $\fW_U \sim \mathbf{Q}_{\gamma_{f}|^{\leftarrow}_{U}}^{(\alpha)}(\theta_1, \theta_2)$, $U\in \gamma_c$, be an independent family with scaffolding $\xi^{(\bar\alpha)}_{\fW_{U}}$ as in \eqref{eq:scaffoldingW}. 
	Then the pair-valued process 
	\[
	\Big(\Concat_{U\in \gamma_c}\cskewer\big(y,\fW_{U}, \xi^{(\bar\alpha)}_{\fW_{U}}\big),~
	\Concat_{U\in \gamma_c}\fskewer\big(y, \fW_{U}, \xi^{(\bar\alpha)}_{\fW_{U}}\big)\Big), \qquad y\ge 0,
	\]
	is called a \emph{coarse-fine $(\alpha,\theta_1,\theta_2,0)$-self-similar interval partition evolution}, starting from $(\gamma_c,\gamma_f)$, abbreviated as \emph{$\mathrm{cfSSIP}^{(\alpha,\theta_1,\theta_2)} (0)$-evolution}. 
\end{definition}

Roughly speaking, it is a random refinement of an $\mathrm{SSIP}^{(\bar\alpha)} (0)$-evolution according to $\mathrm{SSIP}^{(\alpha)} (\theta_1, \theta_2)$-excursions. 
To add immigration to this model,
let $\fW\sim\PRM(c_{\bar\alpha}\Leb\otimes \Theta)$ and consider its coarse scaffolding  $\xi^{(\bar\alpha)}_{\fW}$ as in \eqref{eq:scaffoldingW}. 
For $\bar\theta\ge 0$, as in \eqref{eq:X-alphatheta}, define the process 
\begin{equation} 
	\fX_{\bar\theta}(t) := \xi^{(\bar\alpha)}_{\fW}(t) + \left(1 - \bar\alpha/\bar\theta \right)\ell(t), \quad t\ge 0, \qquad \text{where} \quad \ell(t) := -\inf_{u\leq t}\xi^{(\bar\alpha)}_{\fW}(u).
\end{equation}
For each $j\in \bN$, set $T^{-j}_{\bar\theta} 
:=  \inf\{t\ge 0\colon \fX_{\bar\theta}(t)= - j \}$ and define nested processes 
\begin{align*}
	\cev{\beta}_{c,j}(y) &:= 
	\cskewer \big(y, \fW\big|_{[0,T_{\bar\theta}^{-j})}, j+\fX_{\bar\theta}\big|_{[0,T_{\bar\theta}^{-j})}\big), \qquad y\in [0,j], \\
	\cev{\beta}_{f,j}(y) &:= 
	\fskewer \big(y, \fW\big|_{[0,T_{\bar\theta}^{-j})}, j+\fX_{\bar\theta}\big|_{[0,T_{\bar\theta}^{-j})}\big), \qquad y\in [0,j].
\end{align*}
As in Section~\ref{sec:pre}, we find that 
$
\big(\!\big(\cev{\beta}_{c,j}(y),\cev{\beta}_{f,j}(y)\big), y\in[0,j] \big)  \stackrel{d}{=} \big(\!\big(\cev{\beta}_{c,k}(y),\cev{\beta}_{f,k}(y)\big), y\in[0,j] \big) 
$
for all $k\ge j$. 
Thus, by Kolmogorov's extension theorem, there exists a process  $\big(\cev{\boldsymbol{\beta}}_c,\cev{\boldsymbol{\beta}}_f\big)$ such that $\big(\!\big(\cev{\beta}_c(y),\cev{\beta}_f(y)\big), y\in[0,j]\big)  \stackrel{d}{=} \big(\!\big(\cev{\beta}_{c,j}(y),\cev{\beta}_{f,j}(y)\big), y\in[0,j] \big)$ for every $j\in \bN$. 

\begin{definition}[Coarse-fine SSIP-evolutions with immigration]\label{defn:cfIP-theta}
	Let $\bar\theta,\theta_1,\theta_2\!\ge\! 0$,  $\alpha\!\in \!(0,1)$, $\bar\alpha= \alpha \!-\!\theta_1\!-\!\theta_2 \in (0,\alpha]$ and $(\gamma_c,\gamma_f) \in \cI^2_{\mathrm{nest}}$. Let $(\cev{\boldsymbol{\beta}}_c,\cev{\boldsymbol{\beta}}_f)$ be defined as above and $(\vecc{\boldsymbol{\beta}}_c,\vecc{\boldsymbol{\beta}}_f)$ an independent $\mathrm{cfSSIP}^{(\alpha,\theta_1,\theta_2)} ( 0)$-evolution starting from $(\gamma_c,\gamma_f)$. 
	Then we call $(\cev{\beta}_c(y) \concat\vecc{\beta}_c(y),\cev{\beta}_f(y) \concat\vecc{\beta}_f(y))$, $y\ge 0$,
	a \emph{coarse-fine $(\alpha,\theta_1,\theta_2)$-self-similar interval partition evolution with immigration rate $\bar\theta$}, starting from $(\gamma_c,\gamma_f)$, or a \emph{$\mathrm{cfSSIP}^{(\alpha,\theta_1,\theta_2)} ( \bar\theta)$-evolution}. 
\end{definition}

By construction, the coarse process of a \emph{$\mathrm{cfSSIP}^{(\alpha,\theta_1,\theta_2)} (\bar\theta)$-evolution} is an $\mathrm{SSIP}^{(\bar\alpha)} (\bar\theta)$-evolution. 
For the special case $\theta_1=\theta_2 =0$, the fine process coincides with the coarse one. 

\begin{remark}
	By combining Definition~\ref{defn:cfIP-theta} and Definition~\ref{defn:ipe}, one can further construct a coarse-fine SSIP-evolution with the  coarse process being an $\mathrm{SSIP}^{(\bar\alpha)} (\bar{\theta}_1, \bar\theta_2)$-evolution.  
\end{remark}

\subsection{Convergence of nested PCRPs}\label{sec:nestedPCRP}

For $\bar\alpha\in (0, 1)$ and $\bar\theta\ge 0$, let $(C_c(t), t\ge 0)$ be a Poissonised Chinese restaurant process $\mathrm{PCRP}^{(\bar\alpha)} (\bar\theta,\bar\alpha)$ as in Section~\ref{sec:PCRP}.  
Recall that for each cluster of $C_c$, the mass evolves according to a Markov chain of law $\pi(-\bar\alpha)$ as in \eqref{eq:pi}. 
Let $\alpha\in (\bar\alpha, 1)$ and $\theta_1,\theta_2\ge 0$. Suppose that there is the identity
\[
\theta= \theta_1+\theta_2-\alpha = -\bar{\alpha}<0,  
\]
then the total mass evolution of a $\mathrm{PCRP}^{(\alpha)} (\theta_1,\theta_2)$ also has distribution $\pi(-\bar\alpha)$. 
Therefore, we can fragment each cluster of $C_c$ into $\mathrm{PCRP}^{(\alpha)} (\theta_1,\theta_2)$ as follows. 
In each \emph{cluster}, customers are further attributed into an ordered sequence of \emph{tables}: 
whenever a customer joins this cluster, they choose an existing table or add a new table according to the $(\alpha, \theta_1,\theta_2)$-seating rule; 
whenever the cluster size reduces by one, a customer is chosen uniformly to leave. 
As a result, we embed a $\mathrm{PCRP}^{(\alpha)} (\theta_1,\theta_2)$ into each cluster of $C_c$, independently of the others. 
The rates at which customers arrive are illustrated in Figure~\ref{fig:nested_CRP}. 
For each time $t\ge 0$, by concatenating the composition of ordered table size configuration of each cluster, from left to right according to the order of clusters, we obtain a composition $C_f(t)$, representing the numbers of customers at all tables. Then $C_f(t)$ is a refinement of $C_c(t)$. 
One can easily check that $(C_f(t),t\ge 0)$ is a $\mathrm{PCRP}^{(\alpha)} (\theta_1+\bar{\theta},\theta_2+ \bar{\alpha})$. We refer to the
pair $((C_c(t),C_f(t)),\,t\ge 0)$ as a \em pair of nested PRCPs\em.

\begin{theorem}[Convergence of nested PCRPs]\label{thm:cv-cfIP1}
	For each $n\in \bN$, let $(C^{(n)}_c,C^{(n)}_f)$ be a pair of nested $\mathrm{PCRP}$s as defined above, starting from $(\gamma^{(n)}_c, \gamma^{(n)}_{f})\in \cI^2_{\mathrm{nest}}$. Suppose that $\frac{1}{n}(\gamma^{(n)}_c, \gamma^{(n)}_{f})$ converges to $(\gamma_c, \gamma_{f})\in \cI^2_{\mathrm{nest}}$ under the product metric $d^2_H$. 
	Then the following convergence holds in distribution in the space of c\`adl\`ag functions on $\cI_{H}\times \cI_{H}$ endowed with the Skorokhod topology, 
	\[
	\left( \frac{1}{n}\Big(C^{(n)}_c(2nt),C^{(n)}_f(2nt)\Big), t\ge 0 \right) \underset{n\to \infty}{\longrightarrow} \left(\Big(\beta_c(t),\beta_f(t)\Big),\, t\ge 0\right),
	\]
	where the limit $(\boldsymbol{\beta}_c,\boldsymbol{\beta}_f)= \big(\!\big(\beta_c(t),\beta_f(t)\big),\, t\ge 0\big)$ 
	is a  $\mathrm{cfSSIP}^{(\alpha,\theta_1,\theta_2)}(\bar{\theta})$ starting from $(\gamma_c, \gamma_{f})$.  
\end{theorem}

\begin{proof}
	The arguments are very similar to those in the proof of Theorem~\ref{thm:crp-ip-0} and Proposition~\ref{prop:crp-ip-theta}, with an application of  Theorem~\ref{thm:Theta}, which replaces the role of Proposition~\ref{prop:vague}. 
	Let us sketch the main steps: 
	
	\begin{itemize}
		\item  Let $\fW^{(n)}$ be a Poisson random measure of rescaled excursions of $\mathrm{PCRP}^{(\alpha)} (\theta_1,\theta_2)$ with intensity $2\bar{\alpha}n^{1+\bar{\alpha}}\mathrm{P}^{(n)}$, where $\mathrm{P}^{(n)}$ is as in Theorem~\ref{thm:Theta}. Write $\xi^{(n)}$ for the associated scaffolding of $\fW^{(n)}$ defined as in \eqref{eq:scaffolding-D} and $M^{(n)}$ the total mass of the coarse skewer. 
		Since by Theorem~\ref{thm:Theta} the intensity measure converges vaguely to the $\mathrm{SSIP}^{(\alpha)} (\theta_1,\theta_2)$-excursion measure $c_{\bar{\alpha}}\Theta$, in analogy with Proposition~\ref{prop:cv-prm}, the sequence  
		$(\fW^{(n)}, \xi^{(n)},M^{(n)})$ can be constructed such that it converges a.s.\ to $(\fW, \xi, M)$, where $\xi$ and $M$ are the scaffolding defined as in \eqref{eq:scaffoldingW} and the coarse skewer total mass of $\fW\sim{\tt PRM}(c_{\bar{\alpha}}{\rm Leb}\otimes\Theta)$,  respectively. 
		\item Using similar methods as in Theorem~\ref{thm:crp-ip-0} proves the convergence when $\bar{\theta}=0$. More precisely, using the sequence $\fW^{(n)}$ obtained in the previous step, we give a scaffolding-and-spindles construction for each rescaled nested pair $(\frac{1}{n}C^{(n)}_c(2n\,\cdot\,), \frac{1}{n}C^{(n)}_f(2n\,\cdot\,))$, as in the description below Lemma~\ref{lem:cv-clade} and in  Section~\ref{sec:PCRP}. 
		We first study the case when the initial state of the coarse process is a single interval as in Lemma~\ref{lem:cv-clade},
		and then extend to any initial state by coupling the large clades and controlling the total mass of the remainder. 	
		\item When $\bar\theta>0$, we proceed as in the proof of Proposition~\ref{prop:crp-ip-theta}: we prove that the modified scaffolding converges and then the skewer process also converges. 
	\end{itemize}
	Summarising, we deduce the convergence of nested PCRPs to the coarse-fine skewer processes, as desired. 	
\end{proof}

Having Theorem~\ref{thm:cv-cfIP1}, we can now identify the fine process by Theorem~\ref{thm:crp-ip}. 
\begin{proposition}[Nested SSIP-evolutions]\label{prop:nested}
	Let $\alpha\in (0,1)$, $\theta_1,\theta_2,\bar{\theta}\ge 0$ and suppose that $\theta= \theta_1+\theta_2-\alpha<0$. 
	In a $\mathrm{cfSSIP}^{(\alpha,\theta_1,\theta_2)} ( \bar\theta)$-evolution, the coarse and fine processes 
	are $\mathrm{SSIP}^{(\bar\alpha)}(\bar\theta)$- and $\mathrm{SSIP}^{(\alpha)}(\theta_1\!+\! \bar{\theta} , \theta_2\!+\! \bar\alpha)$-evolutions respectively, where $\bar{\alpha} = -\theta$. 
\end{proposition}
\begin{proof}
	We may assume that this cfSSIP is the limit of a sequence of nested PCRPs, with coarse $\mathrm{PCRP}^{(\bar\alpha)}(\bar\theta)$ and fine $\mathrm{PCRP}^{(\alpha)}(\theta_1\!+\! \bar{\theta} , \theta_2\!+\! \bar\alpha)$. 
	Since the coarse sequence of $\mathrm{PCRP}^{(\bar\alpha)}(\bar\theta)$ converges in its own right, by Theorem~\ref{thm:crp-ip} the limit is an $\mathrm{SSIP}^{(\bar\alpha)}(\bar{\theta})$-evolution. 
	Similarly, the limit of  those  $\mathrm{PCRP}^{(\alpha)}(\theta_1\!+\! \bar{\theta} , \theta_2\!+\! \bar\alpha)$ is an $\mathrm{SSIP}^{(\alpha)}(\theta_1\!+\! \bar{\theta} , \theta_2\!+\! \bar\alpha)$-evolution. 
\end{proof}

\begin{proposition}[Pseudo-stationarity]\label{prop:cf-ps-theta1theta2}
	Let $\alpha\!\in\! (0,1)$, $\theta_1,\theta_2\!\ge\! 0$ with  
	$\bar\alpha:= \alpha\!-\!\theta_1\!-\!\theta_2$ $\in (0,\alpha]$, and $\bar\theta\ge 0$. 
	Let $Z\sim \besq (2 \bar\theta ) $ and $\bar\gamma_c\sim \pdip^{(\bar\alpha)} (\bar\theta,\bar\alpha)$ be independent and $\bar\gamma_f\sim \mathrm{Frag}^{(\alpha)}(\theta_1,\theta_2)(\gamma_c,\,\cdot\,)$. 
	Let $((\beta_c(t),\beta_f(t)),\, t\ge 0)$ be a $\mathrm{cfSSIP}^{(\alpha,\theta_1, \theta_2)}(\bar{\theta})$-evolution starting from $ (Z(0)\bar\gamma_c,Z(0)\bar\gamma_f)$. 
	Then $(\beta_c(t), \beta_f(t))\overset{d}{=}(Z(t)\bar\gamma_c,Z(t)\bar\gamma_f)$ for each $t\ge 0$. 
\end{proposition}
\begin{proof}
	We may assume this cfSSIP-evolution is the limit of a sequence of nested PCRPs $(C^{(n)}_c,C^{(n)}_f)$, with $(C^{(n)}_c,C^{(n)}_f)$ starting from nested compositions of $[n]$ with distribution as in Lemma~\ref{lem:cv-oCRP}. 
	By similar arguments as in Lemma~\ref{prop:crp-ps}, we deduce that, given the total number of customers $m:=\|C^{(n)}_c(t)\|=\|C^{(n)}_f(t)\|$ at time $t\ge 0$, the conditional distribution of $(C^{(n)}_c(t), C^{(n)}_f(t))$ is given by nested $\mathtt{oCRP}_m^{(\bar\alpha)}(\bar\theta, \bar\alpha)$ and $\mathtt{oCRP}_m^{(\alpha)}(\theta_1+\bar{\theta}, \theta_2+\bar{\alpha})$ described above  Lemma~\ref{lem:cv-oCRP}. 
	The claim then follows from Lemma~\ref{lem:cv-oCRP} and Theorem~\ref{thm:cv-cfIP1}. 
\end{proof}

\begin{proposition}[Markov property]\label{prop:nest-Markov}
	A $\mathrm{cfSSIP}^{(\alpha,\theta_1,\theta_2)} ( \bar\theta)$-evolution
	is a Markov process on $(\cI_{\mathrm{nest}}^2, d_H^2)$ with continuous paths. 
\end{proposition}

The proof of Proposition~\ref{prop:nest-Markov} is postponed to Appendix~\ref{sec:nested-Markov}. 

\begin{theorem}\label{thm:nested-ssip}
	For any $\theta\ge 0$ and pairwise nested $\gamma_\alpha\in \cI_H$, $\alpha\in(0,1)$, there exists a nested family $(\boldsymbol{\beta}_{\alpha}, \alpha\in (0,1))$ 
	such that each $\boldsymbol{\beta}_{\alpha}$ is an $\mathrm{SSIP}^{(\alpha)}(\theta)$-evolution starting from $\gamma_\alpha$, and for any $0<\bar\alpha< \alpha<1$, $(\boldsymbol{\beta}_{\bar\alpha}, \boldsymbol{\beta}_{\alpha})$ almost surely takes values in $\cI^2_{\mathrm{nest}}$. 
\end{theorem}
\begin{proof}
	For $0<\bar\alpha<\alpha<1$, let $(\boldsymbol{\beta}_c,\boldsymbol{\beta}_f)\!=\!((\beta_c(y),\beta_f(y)),y\!\ge\! 0)$ be a  $\mathrm{cfSSIP}^{(\alpha,0,\alpha\!-\!\bar\alpha)} ( \theta)$-evolution starting from  $(\gamma_{\bar\alpha},\gamma_\alpha) \in \cI^2_{\mathrm{nest}}$. 
	Then by Proposition~\ref{prop:nested}, the coarse process $\boldsymbol{\beta}_c$ 
	is an  $\mathrm{SSIP}^{(\bar\alpha)}(\theta)$-evolution and  the fine process $\boldsymbol{\beta}_f$ 
	is an  $\mathrm{SSIP}^{(\alpha)}(\theta)$-evolution. 	
	This induces a kernel $\kappa_{\bar\alpha,\alpha}$ from the coarse process to the fine process. Arguing by approximation as in Theorem \ref{thm:cv-cfIP1}, we can prove that $\kappa_{\alpha_1,\alpha_2}\circ\kappa_{\alpha_2,\alpha_3}=\kappa_{\alpha_1,\alpha_3}$ for all $0<\alpha_1<\alpha_2<\alpha_3<1$. More generally, for any finitely many $0<\alpha_1 < \alpha_2 <\cdots < \alpha_n<1$, we can find nested $(\boldsymbol{\beta}_{\alpha_i}, 1\le i\le n)$ that are consistently related by these kernels. 
	We can thus construct the full family by using Kolmogorov's extension theorem.
\end{proof}

Let $\cI^2_{\mathrm{nest},1}:= 
\{ (\gamma_c, \gamma_{f})\!\in\!\cI^2_{\mathrm{nest}} \colon \|\gamma_c\|\!=\!\|\gamma_f\|\!=\!1\}$ be the space of nested partitions of $[0,1]$.
\begin{theorem}
	For any $\theta\ge 0$ and pairwise nested $\bar\gamma_\alpha\in \cI_{H,1}$, $\alpha\in(0,1)$, there exists a family of processes $(\overline{\boldsymbol{\beta}}_{\alpha}, \alpha\in (0,1))$ on $\cI_{H,1}$, such that each $\overline{\boldsymbol{\beta}}_{\alpha}$ is an $\mathrm{IP}^{(\alpha)}(\theta)$-evolution starting from $\bar\gamma_\alpha$, and for any $0<\bar\alpha< \alpha<1$, $(\overline{\boldsymbol{\beta}}_{\bar\alpha}, \overline{\boldsymbol{\beta}}_{\alpha})$ almost surely takes values in $\cI^2_{\mathrm{nest},1}$. 
\end{theorem}
\begin{proof}
	Build a family of SSIP-evolutions $(\boldsymbol{\beta}_{\alpha}, \alpha\in (0,1))$ as in Theorem~\ref{thm:nested-ssip} on the same probability space. 
	In particular, they have the same total mass process and thus the same de-Poissonisation. 
	So the de-Poissonised family $(\overline{\boldsymbol{\beta}}_{\alpha}, \alpha\in (0,1))$ is still nested. 
\end{proof}

\subsection{An application to alpha-gamma trees}\label{sec:trees}
For $n\ge 1$, let $\mathbb{T}_n$ be the space of all (non-planar) trees without degree-2 vertices, a root vertex of degree 1, and exactly $n$ further degree-1 vertices, \em leaves \em labelled by $[n] = \{1,\ldots, n\}$. 
For $\alpha\in (0,1)$ and $\gamma\in [0,\alpha]$, we construct random trees $T_n$ by using the following \emph{$(\alpha,\gamma)$-growth rule} \cite{CFW}: 
$T_1$ and $T_2$ are the unique elements in $\mathbb{T}_1$ and $\mathbb{T}_2$.   
Given $T_k$ with $k\ge 2$, assign weight $1\!-\!\alpha$ to each of the $k$ edges adjacent to a leaf, weight $\gamma$ to each of the other edges, and weight $(d\!-\!2)\alpha-\gamma$ to each branch point with degree $d\ge 3$.  
To create $T_{k+1}$ from $T_k$, choose an edge or a branch point proportional to the weight, and insert the leaf $k\!+\!1$ to the chosen edge or branch point. 
This generalises R\'emy's algorithm \cite{Remy85} of the uniform tree (when $\alpha=\gamma= 1/2$) and Marchal's recursive construction \cite{Mar08} of $\rho$-stable trees with $\rho\in (1,2]$ (when $\alpha =1/\rho$ and $\gamma=1-1/\rho$). 

For each $T_n$, consider its spinal decomposition as discussed in the introduction, the spine being the path connecting the leaf $1$ and the root. 
Let $C_c(n)$ be the sizes of bushes at the spinal branch points, ordered from left to right in decreasing order of their distances to the root. Then the $(\alpha,\gamma)$-growth rule implies that 
the $(C_c(n), n\in \bN)$ is an  $\mathrm{oCRP}^{(\gamma)}(1\!-\!\alpha,\gamma)$.  
Similar as the \emph{semi-planar $(\alpha,\gamma)$-growth trees} in \cite{Soerensen}, we further equip each spinal branch point with a left-to-right ordering of its subtrees, such that the sizes of the sub-trees in each bush follow the $(\alpha,0,\alpha\!-\!\gamma)$-seating rule. 
By concatenating the sub-tree-configurations of all bushes according to the order of bushes, we obtain the composition $C_f(n)$ of sizes of subtrees. 
Then $(C_f(n), n\in \bN)$ is an $\mathrm{oCRP}^{(\alpha)}(1\!-\!\alpha,\alpha)$ nested to $(C_c(n), n\in \bN)$, as in Figure~\ref{fig:nested_CRP}. 

Let us introduce a continuous-time Markov chain $(\mathbf{T}(s), s\ge 0)$ on $\mathbb{T}=\bigcup_{n\ge 1} \mathbb{T}_n$, the space of labelled rooted trees without degree-2 vertices. 
Given $\mathbf{T}(s)$, assign weights to its branch points and edges as in the $(\alpha,\gamma)$-growth model, such that for each branch point or edge, a new leaf arrives and is attaches to this position at the rate given by its weight. Moreover, fix the root and the leaf $1$, and delete any other leaf at rate one, together with the edge attached to it; in this operation, if a branching point degree is reduced to two, we also delete it and merge 
the two edges attached to it. 

For each $n\ge 1$, consider such a continuous-time up-down Markov chain  $(\mathbf{T}^{(n)}(s), s\!\ge\! 0)$ starting from a random tree $T_n$ built by the $(\alpha,\gamma)$-growth rule. 
At each time $s\ge 0$, with the spine being the path connecting the leaf $1$ and the root, we similarly obtain a nested pair $(C_c^{(n)}(s), C_f^{(n)}(s))$, representing the sizes of spinal bushes and subtrees respectively. Then it is clear that $\big(C_c^{(n)}(s), s\ge 0\big)$ is a $\mathrm{PCRP}^{(\gamma)}(1\!-\!\alpha)$ and that $\big( C_f^{(n)}(s), s\ge 0\big)$ is a  $\mathrm{PCRP}^{(\alpha)}(1\!-\!\alpha,\alpha)$ nested within $C^{(n)}_c$, such that the size evolution of the subtrees in each bush gives a  $\mathrm{PCRP}^{(\alpha)}(0, \alpha\!-\!\gamma)$.

\begin{proposition}
	For each $n\ge 1$, let $\big((C_c^{(n)}(t), C_f^{(n)}(t)), t\ge 0\big)$ be a pair of nested PCRPs defined as above, associated with a tree-valued process $(\mathbf{T}^{(n)}(s), s\ge 0)$ starting from  $T_n$. As $n\to \infty$,  $\big(\frac{1}{n}\big(C_c^{(n)}(2nt), C_f^{(n)}(2nt)\big), t\ge 0\big)$ converges in distribution to a $\mathrm{cfSSIP}^{(\alpha,0, \alpha-\gamma)}(1\!-\!\alpha)$-evolution $(\boldsymbol{\beta}_c, \boldsymbol{\beta}_f)$ starting from $ (\gamma_c,\gamma_f)$, where $\gamma_c\sim \pdip^{(\gamma)} (1\!-\!\alpha, \gamma)$ and $\gamma_f \sim \mathrm{Frag}^{(\alpha)}(0,\alpha\!-\!\gamma)(\gamma_c,\,\cdot\,)$.  
\end{proposition}
\begin{proof}
	We characterised the limiting initial distribution in Lemma~\ref{lem:cv-oCRP} and deduce the convergence of the rescaled process by Theorem~\ref{thm:cv-cfIP1}.
\end{proof}

For $\rho\in (1,2]$, with $\alpha= 1-1/\rho$ and $\gamma = 1-\alpha$, the nested evolution $(\boldsymbol{\beta}_c, \boldsymbol{\beta}_f)$, transformed by de-Poissonisation as in Definition~\ref{defn:dePoi}, would be stationary with the law of nested $\pdip^{(1/\rho)} (1/\rho, 1/\rho)$ and $\pdip^{(1-1/\rho)} (1/\rho, 1-1/\rho)$. 
This stationary distribution corresponds to the coarse and fine spinal decompositions in a $\rho$-stable L\'evy tree \cite[Corollary 10]{HPW}. 
%


\begin{appendix}
	\section{Birth-death Processes}\label{sec:ud-proof}
	
	\begin{proof}[Proof of Lemma~\ref{lem:ud-bis}] We adapt the proof of \cite[Theorem 3(i)]{BertKort16}, which establishes such convergence of hitting times in a general context of discrete-time Markov chains 
		converging to positive self-similar Markov processes. This relies on Lamperti's representation for $Z\sim \besq_a (2 \theta)$
		$$Z(t)=\exp(\xi(\sigma(t))),\qquad\mbox{where }\sigma(t)=\inf\left\{s\ge 0\colon\int_0^se^{\xi(r)}dr>t\right\},$$
		for a Brownian motion with drift $\xi(t)=\log(a)+2B(t)-2(1-\theta)t$, and corresponding representations 
		$$Z_n(t)=\exp(\xi_n(\sigma_n(t))),\qquad\mbox{where }\sigma_n(t)=\inf\left\{s\ge 0\colon\int_0^se^{\xi_n(r)}dr>t\right\},$$
		for continuous-time Markov chains $\xi_n$, $n\ge 1$, with increment kernels
		$$L^n(x,dy)=2ne^x\Big((ne^x+\theta)\delta_{\log(1+1/ne^x)}(dy)+ne^x\delta_{\log(1-1/ne^x)}(dy)\Big),\ \ x\ge -\log n.$$ 
		We easily check that for all $x\in\mathbb{R}$
		$$\int_{y\in\mathbb{R}}yL^n(x,dy)\rightarrow -2(1-\theta)\quad\mbox{and}\quad\int_{y\in\mathbb{R}}y^2L^n(x,dy)\rightarrow 4,$$
		as well as
		$$\sup_{\{x\colon|x|\le r\}}\int_{y\in\mathbb{R}}y^21_{\{|y|>\varepsilon\}}L^n(x,dy)=0\quad\mbox{for $n$ sufficiently large.}$$
		To apply \cite[Theorem IX.4.21]{JacodShiryaev}, we further note that all convergences are locally uniform, and we extend the increment kernel $\widetilde{L}^n(x,dy):=L^n(x,dy)$, $x\ge\log(2)-\log(n)$, by setting
		$$\widetilde{L}^n(x,dy):=L^n(\log(2)-\log(n),dy),\quad x<\log(2)-\log(n),$$
		to be definite. With this extension of the increment kernel, we obtain $\widetilde{\xi}_n\rightarrow\xi$ in distribution on 
		$\mathbb{D}([0,\infty),\mathbb{R})$. This implies $\xi_n\rightarrow\xi$ in distribution also for the process $\xi_n$ that jumps from $-\log(n)$ to 
		$-\infty$, but only if we stop the processes the first time they exceed any fixed negative level. 
		
		Turning to extinction times $\tau_n$ of $Z_n$, we use Skorokhod's representation $\xi_n\rightarrow\xi$ almost surely. Then we want to show 
		that also
		$$\tau_n=\int_0^\infty e^{\xi_n(s)}ds\rightarrow\tau=\int_0^\infty e^{\xi(s)}ds\qquad\mbox{in probability.}$$
		We first establish some uniform bounds on the extinction times when $Z_n$, $n\ge 1$, are started from 
		sufficiently small initial states. To achieve this, we consider the generator $\mathcal{L}^n$ of $Z_n$ and note that for $g(x)=x^\beta$, we have
		$$\mathcal{L}^ng(x)=2n((nx+\theta)g(x+1/n)+nx g(x-1/n)-(2nx+\theta)g(x))\le -C g(x)/x$$
		for all $n\ge 1$ and $x\ge K/n$ if and only if
		$$\frac{g(1+h)-2g(1)+g(1-h)}{h^2}+\theta\frac{g(1+h)-g(1)}{h}\le -C/2\quad\mbox{for all $h\le 1/K$.}$$
		But since $g^{\prime\prime}(1)+\theta g^\prime(1)=\beta(\beta-1+\theta)<0$ for $\beta\in(0,1-\theta)$, the function $g$ is a Foster-
		Lyapunov function, and \cite[Corollary 2.7]{MensPetr14}, applied with $q=p/2=\beta$ and $f(x)=x^{1/2}$, yields
		$$\exists C^\prime>0\ \forall n\ge 1,\ \forall x\ge K/n\quad\mathbb{E}_x((\tau_n^{(K)})^q)<C^\prime x^\beta,$$
		where $\tau_n^{(K)}=\inf\{t\ge 0\colon Z_n(t)\le K/n\}$. An application of Markov's inequality yields 
		$\mathbb{P}_x(\tau_n^{(K)}>t)\le C^\prime x^\beta t^{-\beta}$. In particular,
		$$\forall\varepsilon>0\ \forall t>0\ \exists\eta>0\ \forall n\ge 1\ \forall K+1\le i\le\eta n\quad \mathbb{P}_{i/n}(\tau_n^{(K)}>t/6)\le\varepsilon/8.$$
		Furthermore, there is $n_0$ such that for $n\ge n_0$, the probability that $Z_n$ starting from $K/n$ takes more than time $t/6$ to get 
		from $K/n$ to $0$ is smaller than $\varepsilon/8$. Now choose $R$ large enough so that 
		$$\mathbb{P}(\exp(\xi(R))<\eta/2)\ge 1-\varepsilon/8\quad\mbox{and}\quad\mathbb{P}\left(\int_R^\infty e^{\xi(s)}ds>t/3\right)\le\varepsilon/4.$$
		We can also take $n_1\ge n_0$ large enough so that 
		$$\mathbb{P}(|\exp(\xi_n(R))-\exp(\xi(R))|<\eta/2)\ge 1-\varepsilon/8\quad\mbox{for all $n\ge n_1$.}$$
		Then, considering $\exp(\xi_n(R))$ and applying the Markov property at time $R$, 
		$$\mathbb{P}(\exp(\xi_n(R))>\eta)<\varepsilon/4\quad\mbox{and}\quad\mathbb{P}\left(\int_R^\infty e^{-\xi_n(s)}ds>t/3\right)\le\varepsilon/2,\quad\mbox{for all $n\ge n_1$.}$$
		But since $\xi_n\rightarrow\xi$ almost surely, uniformly on compact sets, we already have 
		$$\int_0^R e^{\xi_n(s)}ds\rightarrow\int_0^Re^{\xi(s)}ds\quad\mbox{almost surely.}$$
		Hence, we can find $n_2\ge n_1$ so that for all $n\ge n_2$
		$$\mathbb{P}\left(\left|\int_0^R e^{\xi_n(s)}ds-\int_0^Re^{\xi(s)}ds\right|>t/3\right)<\varepsilon/4.$$
		We conclude that, for any given $t>0$ and any given $\varepsilon$, we found $n_2\ge 1$ such that for all $n\ge n_2$
		$$\mathbb{P}\left(\left|\int_0^\infty e^{\xi_n(s)}ds-\int_0^\infty e^{\xi(s)}ds\right|>t\right)<\varepsilon,
		$$
		as required. 
	\end{proof}
	
	\begin{proof}[Proof of Proposition~\ref{prop:vague}] Denote by $A(f)=\sup|f|$ the supremum of a c\`adl\`ag excursion $f$. In using the term ``vague convergence'' on spaces that 
		are not locally compact, but are bounded away from a point (here bounded on $\{A>a\}$ for all $a>0$), we follow 
		Kallenberg \cite[Section 4.1]{KallenbergRM}. Specifically, it follows from his Lemma 4.1 that it suffices to show for all $a>0$
		\begin{enumerate}\item[1.] $\Lambda^{(2\theta)}_{\mathtt{BESQ}}(A=a)=0$,
			\item[2.] $(\Gamma(1+\theta)/(1-\theta)) n^{1-\theta} \cdot \widetilde{\pi}_1^{(n)}(\theta)(A>a) \underset{n\to \infty}{\longrightarrow} \Lambda^{(2\theta)}_{\mathtt{BESQ}}(A>a)$,
			\item[3.] $\widetilde{\pi}_1^{(n)}(\theta)(\,\cdot\,|\,A>a) \underset{n\to \infty}{\longrightarrow} \Lambda^{(2\theta)}_{\mathtt{BESQ}}(\,\cdot\,|\,A>a)$ weakly.
		\end{enumerate}
		See also \cite[Proposition A2.6.II]{DaleyVereJones1}.
		
		1. is well-known. Indeed, we have chosen to normalise $\Lambda^{(2\theta)}_{\mathtt{BESQ}}$ so that 
		$\Lambda^{(2\theta)}_{\mathtt{BESQ}}(A>a)=a^{\theta-1}.$
		See e.g. \cite[Section~3]{PitmYor82}. Cf. \cite[Lemma 2.8]{Paper1-1}, where a different normalisation was chosen. 
		
		2. can be proved using scale functions. Let us compute a scale function $s$ for the birth-death chain with up-rates $i+\theta$ and down-rates
		$i$ from state $i\ge 1$. Set $s(0)=0$ and $s(1)=1$. For $s$ to be a scale function, we need 
		$$(k+\theta)(s(k+1)-s(k))+k(s(k-1)-s(k))=0\qquad\mbox{for all }k\ge 1.$$
		Let $d(k)=s(k)-s(k-1)$, $k\ge 1$. Then 
		$$d(k+1)=\frac{k}{k+\theta}d(k)=\frac{\Gamma(k+1)\Gamma(1+\theta)}{\Gamma(k+1+\theta)}\sim\Gamma(1+\theta)k^{-\theta}\qquad\mbox{as } k\rightarrow\infty,$$
		and therefore
		$$s(k)=\sum_{i=1}^kd(i)=\sum_{i=1}^k\frac{\Gamma(i+1)\Gamma(1+\theta)}{\Gamma(i+1+\theta)}\sim\frac{\Gamma(1+\theta)}{1-\theta}k^{1-\theta}.$$
		Then the scale function applied to the birth-death chain is a martingale. Now let $p(k)$ be the probability of hitting $k$ before absorption in 0, 
		when starting from 1. Applying the optional stopping theorem at the first hitting time of $\{0,k\}$, we find $p(k)s(k)=1$, and hence
		$$\frac{\Gamma(1\!+\!\theta)}{1\!-\!\theta} n^{1-\theta}p(\lceil na\rceil)=\frac{\Gamma(1\!+\!\theta) n^{1- \theta}}{(1-\theta)s(\lceil na\rceil)}\underset{n\to\infty}{\longrightarrow}a^{\theta-1},$$
		as required.
		
		3. can be proved by using the First Description of $\Lambda^{(2\theta)}_{\mathtt{BESQ}}$ given in \cite[(3.1)]{PitmYor82}, which states, in particular, that the
		excursion under $\Lambda^{(2\theta)}_{\mathtt{BESQ}}(\,\cdot\,|A>a)$ is a concatenation of two independent processes, an $\uparrow$-diffusion
		starting from 0 and stopped when reaching $a$ followed by a $0$-diffusion starting from $a$ and run until absorption in 0. In our case, the 
		$0$-diffusion is ${\tt BESQ}(2\theta)$, while \cite[(3.5)]{PitmYor82} identifies the $\uparrow$-diffusion as ${\tt BESQ}(4-2\theta)$. 
		Specifically, straightforward Skorokhod topology arguments adjusting the time-changes around the concatenation times, see 
		\cite[VI.1.15]{JacodShiryaev}, imply that it suffices to show: 
		\begin{enumerate}\item[(a)] The birth-death chain starting from 1 and conditioned to reach $\lceil na\rceil$ before 0, rescaled,   
			converges to a ${\tt BESQ}(4\!-\!2\theta)$ stopped at $a$, jointly with the hitting times of $\lceil na\rceil$.
			\item[(b)] The birth-death chain starting from $\lceil na\rceil$ run until first hitting 0, rescaled, converges to a ${\tt BESQ}(2\theta)$
			starting from $a$ stopped when first hitting $0$, jointly with the hitting times.
		\end{enumerate}
		(b) was shown in \cite[Theorem 1.3]{RogeWink20}. See also Lemmas \ref{lem:ud}--\ref{lem:ud-bis} here, completing the convergence of the hitting time. For (a), we adapt that proof. But first we need to identify the conditioned birth-death 
		process. Note that the holding rates are not affected by the conditioning. An elementary argument based purely on the jump chain shows that
		the conditioned jump chain is Markovian, and its transition probabilities are adjusted by factors $s(i\pm 1)/s(i)$ so that the conditioned birth-death
		process has up-rates $(i+\theta)s(i+1)/s(i)$ and down-rates $is(i-1)/s(i)$ from state $i\ge1$.  Rescaling, our processes are instances of $\mathbb{R}$-valued  
		pure-jump Markov process with jump intensity kernels $\widetilde{K}^n(x,dy)=0$ for $x\le 0$ and, for $x>0$, 
		\[\widetilde{K}^n(x,dy)=2n\!\left(\!\!\left(\lceil nx\rceil+\theta\right)\frac{s(\lceil nx\!+\!1\rceil)}{s(\lceil nx\rceil)}\delta_{1/n}(dy)
		+\lceil nx\rceil\frac{s(\lceil nx\!-\!1\rceil)}{s(\lceil nx\rceil)}\delta_{-1/n}(dy)\!\right)\!.
		\]
		We now check the drift, diffusion and jump criteria of \cite[Theorem IX.4.21]{JacodShiryaev}: for $x>0$
		\begin{align*}
			\int_{\mathbb{R}}y\widetilde{K}^n(x,dy)
			&=2\lceil nx\rceil\frac{s(\lceil nx\!+\!1\rceil)-s(\lceil nx\!-\!1\rceil)}{s(\lceil nx\rceil)}+2\theta\frac{s(\lceil nx\!+\!1\rceil)}{s(\lceil nx\rceil)}\\
			&\rightarrow 4-4\theta+2\theta=4-2\theta,\\
			\int_{\mathbb{R}}y^2\widetilde{K}^n(x,dy)
			&=\frac{2\lceil nx\rceil}{n}\frac{s(\lceil nx\!+\!1\rceil)+s(\lceil nx\!-\!1\rceil)}{s(\lceil nx\rceil)}+\frac{2\theta}{n}\frac{s(\lceil nx\!+\!1\rceil)}{s(\lceil nx\rceil)}\\
			&\rightarrow 4x+0=4x,\\
			\int_{\mathbb{R}}y^21_{\{|y|\ge\varepsilon\}}\widetilde{K}^n(x,dy)&=0\qquad\mbox{for $n$ sufficiently large,}
		\end{align*}
		all locally uniformly in $x\in(0,\infty)$, as required for the limiting $(0,\infty)$-valued ${\tt BESQ}(4-2\theta)$ diffusion with infinitesimal drift $4-2\theta$ and diffusion coefficient $4x$. The convergence of hitting times of $\lfloor na\rfloor$, which is the first passage time above level $\lfloor na\rfloor$, follows from the regularity of the limiting diffusion after the first passage time above level $a$. See e.g. 
		\cite[Lemma 3.3]{RogeWink20}. 
	\end{proof}
	
	\section{Proof of Theorem~\ref{thm:crp-ip-bis}}\label{sec:proof-thm:crp-ip-bis}
	
	\begin{proof}[Proof of Theorem~\ref{thm:crp-ip-bis}]
		We start with the case $\gamma = \emptyset$ and $\theta<1$. 
		For any $\varepsilon>0$, for all $n$ large enough, 
		the total mass process $\|C^{(n)}\|$ is stochastically dominated  by an up-down chain $Z^{(n)}\sim \pi(2\theta)$ starting from $\lfloor n\varepsilon \rfloor$, killed at  the hitting time of zero denoted by $\zeta_Z^{(n)}$. 
		Using Lemmas~\ref{lem:ud} and \ref{lem:ud-bis}, we deduce that, as $n\to \infty$, $(\frac{1}{n}Z^{(n)}(2nt), t\ge 0)$ converges in distribution to $\besq_{\varepsilon}(2\theta)$ killed at zero, and	$\zeta_Z^{(n)}/2n$  converges jointly in distribution to $\varepsilon/2G$ with $G$ a Gamma variable. 
		By the arbitrariness of $\varepsilon$, we conclude that 
		$(\frac{1}{n}\|C^{(n)}((2nt)\wedge \zeta^{(n)})\| ,t\ge 0) \to 0$ and $\zeta^{(n)}/2n\to 0$ in probability jointly.   
		Then the convergence of the PCRP follows. 
		
				We next deal with the case $\gamma\ne \emptyset$.  
		By assumption, $\frac{1}{n}  C^{(n)}(0)$ converges in distribution to 
		$\gamma\in \cI_H$ under the metric $d_H$. 
		As $\gamma\ne \emptyset$, 
		we may use Skorokhod representation and Lemma~\ref{lem:dH} to find $\big(C^{(n)}_1(0), m^{(n)}(0),C^{(n)}_2(0)\big)$ for all $n$ sufficiently large, with $m^{(n)}(0)\ge 1$ and $C^{(n)}_1(0),C^{(n)}_2(0)\in \cC$, such that $C^{(n)}_1(0)\concat \{(0, m^{(n)}(0))\}\concat C^{(n)}_2(0)= C^{(n)}(0)$, and that, as $n\to \infty$, 
		\begin{equation}\label{eq:intial}
			\Big(\frac{1}{n}  C_1^{(n)}(0), \frac{1}{n} m^{(n)}(0), \frac{1}{n}  C_2^{(n)}(0)\Big)\to (\gamma_1, m,\gamma_2):=\phi(\gamma), \quad \text{a.s.,}
		\end{equation} 
		where $\phi$ is the function defined by \eqref{eq:phi}.	
		
		For every $n\in \bN$, let $\ff^{(n,0)}\sim \pi_{m^{(n)}(0)}(-\alpha)$ be as in \eqref{eq:pi}, $\boldsymbol{\gamma}^{(n,0)}_1$ a $\mathrm{PCRP}^{(\alpha)} (\theta_1, \alpha)$ starting from $ C_1^{(n)}(0)$
		and $\boldsymbol{\gamma}^{(n,0)}_2$ a $\mathrm{PCRP}^{(\alpha)} (\alpha, \theta_2)$ starting from $ C_2^{(n)}(0)$;  
		the three processes $\boldsymbol{\gamma}_1^{(n,0)},\ff^{(n,0)}$ and $\boldsymbol{\gamma}_2^{(n,0)}$ are independent. 
		By Proposition~\ref{prop:crp-ip-theta}, Lemma~\ref{lem:ud} and Skorokhod representation, we may assume that a.s.\ 
		\begin{equation}\label{eq:cvT1}
			\Big(\frac{1}{n} \gamma_1^{(n,0)}(2n~\cdot), \frac{1}{n}\ff^{(n,0)}(2n~\cdot),\frac{\zeta(\ff^{(n,0)})}{2n}, \frac{1}{n} \gamma_2^{(n,0)}(2n~\cdot)\Big)\to 
			\Big(\boldsymbol{\gamma}_1^{(0)},\ff^{(0)}, \zeta(\ff^{(0)}),\boldsymbol{\gamma}_2^{(0)}\Big).
		\end{equation}	
		The limiting triple process $(\boldsymbol{\gamma}_1^{(0)},\ff^{(0)}, \boldsymbol{\gamma}_2^{(0)})$ starting from $(\gamma_1, m,\gamma_2)$ can serve as that in the construction of $\boldsymbol{\beta}$ in Definition~\ref{defn:ipe}. 
		Write $T_1=\zeta(\ff^{(0)})$ and $T_{n,1}= \zeta(\ff^{(n,0)})$, then 
		\[
		\frac{1}{n} \gamma_1^{(n,0)}( T_{n,1})\concat \frac{1}{n} \gamma_2^{(n,0)}( T_{n,1}) \to \gamma_1^{(0)}(T_{1})\concat \gamma_2^{(0)}(T_{1})=:\beta(T_1), \quad \text{a.s..}.
		\]
		With $\phi$ the function defined in \eqref{eq:phi}, set 
		\[
		( C_1^{(n,1)}, m^{(n,1)}, C_2^{(n,1)})
		:=\phi\left(\gamma^{(0)}_1(T_{n,1})\concat \gamma^{(0)}_2(T_{n,1})\right).	
		\]
		Since $T_1$ is independent of $(\boldsymbol{\gamma}_1^{(0)},\boldsymbol{\gamma}_2^{(0)})$, $\beta(T_1)$ a.s.\@ has a unique largest block. 
		By this observation and \eqref{eq:cvT1} we have
		$ \frac{1}{n}	( C_1^{(n,1)}, m^{(n,1)}, C_2^{(n,1)}) \to  \phi(\beta(T_{1}))$, since $\phi$ is continuous at any interval partition whose longest block is unique. 	
		
		For each $n\ge 1$, 	if $(  C_1^{(n,1)}, m^{(n,1)}, C_2^{(n,1)}) =(\emptyset, 0,\emptyset)$, then for every $i\ge 1$, we set $T_{n,i}:= T_{n,1}$  and 
		$
		\Big(  \boldsymbol{\gamma}^{(n,i)}_1 ,\ff^{(n,i)} ,\boldsymbol{\gamma}^{(n,i)}_2\Big):\equiv (\emptyset,0, \emptyset).$				
		If $( C_1^{(n,1)}, m^{(n,1)}, C_2^{(n,1)}) \ne (\emptyset, 0,\emptyset)$, then conditionally on the history, let $\ff^{(n,1)}\sim \pi_{m^{(n,1)}}(-\alpha)$, and consider $\boldsymbol{\boldsymbol{\gamma}}^{(n,1)}_1$, a $\mathrm{PCRP}^{(\alpha)} (\theta_1,\alpha)$ starting from $C_1^{(n,1)}$, and $\boldsymbol{\gamma}^{(n,1)}_2$, a $\mathrm{PCRP}^{(\alpha)} (\alpha,\theta_2)$ starting from $C_2^{(n,1)}$, independent of each other. Set $T_{n,2} = T_{n,1} +\zeta(\ff^{(n,1)})$. 
		Again, by Proposition~\ref{prop:crp-ip-theta}, Lemma~\ref{lem:ud} and Skorokhod representation, we may assume that a similar a.s.\ convergence as in \eqref{eq:cvT1} holds for $(\boldsymbol{\gamma}_1^{(n,1)},\ff^{(n,1)},\zeta(\ff^{(n,1)}),\boldsymbol{\gamma}_2^{(n,1)})$.

		By iterating arguments above, we finally obtain for every $n\ge 1$ a sequence of processes $(\boldsymbol{\gamma}_1^{(n,i)},\ff^{(n,i)}, \boldsymbol{\gamma}_2^{(n,i)})_{i\ge 0}$ with renaissance levels $(T_{n,i})_{i\ge 0}$, such that, inductively, for every $k\ge 0$, the following a.s.\@ convergence holds: 
		\begin{equation}\label{eq:cv-ik}
			\Big(\frac{1}{n} \gamma_1^{(n,k)}\!(2n~\cdot), \frac{1}{n}\ff^{(n,k)}\!(2n~\cdot), \frac{\zeta(\ff^{(n,k)})}{2n}, \frac{1}{n} \gamma_2^{(n,k)}\!(2n~\cdot)\Big) 
			\to \Big(\boldsymbol{\gamma}_1^{(k)},\ff^{(k)},\zeta(\ff^{(k)}), \boldsymbol{\gamma}_2^{(k)}\Big).  
		\end{equation}
		Using the limiting processes $\Big(\boldsymbol{\gamma}_1^{(k)},\ff^{(k)}, \boldsymbol{\gamma}_2^{(k)}\Big)_{k\ge 0}$, we build according to Definition~\ref{defn:ipe} an $\mathrm{SSIP}_{\!\dagger}^{(\alpha)}(\theta_1, \theta_2)$-evolution $\boldsymbol{\beta}=(\beta(t), t\ge 0)$, starting from $\gamma$, with renaissance levels $T_k = \sum_{i=0}^{k-1}\zeta(\ff^{(i)} )$ and $T_{\infty}=\lim_{k\to \infty} T_k$.  
		
		Then for every $t\ge 0$ and $k\in \bN$, on the event $\{T_{k} >t \}$ we have by \eqref{eq:cv-ik} the a.s.\@ convergence 
		of the process $(   \frac{1}{n}  C^{(n)} (2n s) , s\le t ) \to (    \beta (s) , s\le t )$. 
		When $\theta\ge 1$, since the event 
		$\{T_{\infty}=\infty\}=\bigcap_{t\in \bN} \bigcup_{k\in \bN} \{T_{k} >t \}$  has probability one by Theorem~\ref{thm:IvaluedMP}~(\ref{item:levels}), the convergence in Theorem~\ref{thm:crp-ip-bis} holds a.s.. 
		
		We now turn to the case $\theta<1$, where we have by Theorem~\ref{thm:IvaluedMP}~(\ref{item:levels}) that $T_{\infty}<\infty$ a.s.\ 
		and that, for any $\varepsilon >0$, there exists $K\in \bN$ such that 
		\begin{equation}\label{eq:cv-1}
			\bP \Big(\sup_{t\ge T_K} \|\beta(t)\|>\varepsilon \Big) <\varepsilon\quad\mbox{and}\quad\bP\big(T_\infty>T_K+\varepsilon\big)<\varepsilon.
		\end{equation}
		For each $n\in \bN$, consider the  concatenation
		\begin{equation*}
			C^{(n)}\!(t)\!=\! \begin{cases}
				\!\gamma^{\!(n,i)}_1\!(t\!-\! T_{n,i})\!\concat\! \big\{\!\big(0, \ff^{(n,i)}\!(t\!-\!T_{n,i})\big)\!\big\} \!\concat\!\gamma^{\!(n,i)}_2\!(t\!-\!T_{n,i}) ,\!  &t\!\in\! [T_{n,i}, T_{n,i+1}), i\!\le\! K\!\!-\!1,\\ 
				\!\widetilde{C}^{(n)}(t\!-\!T_{n,K}),& t\ge T_{n,K}, 
			\end{cases}
		\end{equation*}
		where  $\widetilde{C}^{(n)}$ is a $\mathrm{PCRP}^{(\alpha)} (\theta_1, \theta_2)$ starting from $ C^{(n)}(T_{n,K}-)$ and killed at $\emptyset$, independent of the history. 
		Then $C^{(n)}$ is a $\mathrm{PCRP}^{(\alpha)} (\theta_1, \theta_2)$ starting from $C^{(n)}(0)$  and killed at $\emptyset$. 
		We shall next prove that its rescaled process converges to $(\beta(t),t\ge 0)$ in probability, which completes the proof. 
		
		By the convergence \eqref{eq:cv-ik}, there exists $N\in \bN$ such that for every $n>N$, we have
		\begin{equation}\label{eq:cv-2}
			\bP\bigg(\sup_{s\in [0, T_K]} d_{H} \Big(\frac{1}{n}C^{(n)} (2ns),  \beta(s)\Big) >\varepsilon \bigg)<\varepsilon\quad\mbox{and}\quad
			\bP\bigg(\Big|\frac{1}{2n}T_{n,K}-T_K\Big|>\varepsilon\bigg)<\varepsilon.
		\end{equation}
		Furthermore, by Lemmas~\ref{lem:ud}--\ref{lem:ud-bis}, under the locally uniform topology
	\[
	\Big( \frac{1}{n}\|\widetilde{C}^{(n)} (2n\cdot)\|, \frac{1}{2n} \zeta (\widetilde{C}^{(n)}) \Big)
	\underset{n\to\infty}{\to} 
	\Big(\|\beta(\cdot+ T_K)\|, \zeta\big(\beta(\cdot+ T_K)\big) \Big) \quad \text{in distribution}.
	\]
		By the convergence of $\frac{1}{n}\|\widetilde{C}^{(n)}\|$, there exists $\widetilde{N}\in \bN$
		such that for every $n>\widetilde{N}$, 
		\begin{equation}\label{eq:cv-3}
			\bP\bigg( \sup_{s\ge 0}\frac{1}{n}\|\widetilde{C}^{(n)}(s)\|> \varepsilon\bigg)  <\varepsilon\quad\mbox{and}\quad
			\bP\bigg(\Big|\frac{1}{2n}\zeta(\widetilde{C}^{(n)})-\zeta\big(\beta(\cdot+T_K)\big)\Big|>\varepsilon\bigg)<\varepsilon.
		\end{equation}
		Summarising \eqref{eq:cv-1} and \eqref{eq:cv-3}, for every $n>\widetilde{N}$, we have
		\[
		\bP \bigg(\! \sup_{s\in [0, \infty)} d_{H} \Big(\frac{1}{n}\widetilde{C}^{(n)} (2ns),  \beta(s+T_K)\Big)>3\varepsilon\bigg)\le 3\varepsilon. 
		\]
		Together with \eqref{eq:cv-2}, this leads to the desired convergence in probability. 
	\end{proof}
	\section{The Markov property of nested SSIP-evolutions}\label{sec:nested-Markov}
	
	To prove Proposition~\ref{prop:nest-Markov}, we first give a property of the excursion measure $\Theta^{(\alpha)}(\theta_1,\theta_2)$. For any $\cI_{H}$-valued process $\boldsymbol{\gamma}\!=\!(\gamma(y),y\!\ge\!0)$ and $a\!>\!0$, let $H^a(\boldsymbol{\gamma}):= \inf\{y\!\ge\! 0\colon \|\gamma(y)\|\!>\!a \}$. 
	\begin{lemma}\label{lem:Theta-Ha}
		For $a>0$, let $\boldsymbol{\beta}=(\beta(y),\,y\ge 0)\sim\Theta^{(\alpha)}(\theta_1,\theta_2)(\,\cdot\,|\, H^a<\infty)$. Conditionally on $(\beta(r),\, r\le H^a(\boldsymbol{\beta}))$, the process $(\beta(H^a(\boldsymbol{\beta})+z),\, z\ge 0)$ is an $\mathrm{SSIP}^{(\alpha)}(\theta_1,\theta_2)$-evolution starting from $\beta(H^a(\boldsymbol{\beta}))$. 
	\end{lemma} 
	\begin{proof}
		For $k\in \bN$, let $H^a_k:= 2^{-k}\lceil2^k H^a\rceil\wedge 2^k$. Then $H^a_k$ is a stopping time that a.s.\ only takes a finite number of possible values and eventually decreases to $H^a$. 
		By \eqref{eq:Theta:entrance}, the desired property is satisfied by each $H^a_k$. Then we deduce the result for $H^a$ by approximation, using the path-continuity and Hunt property of $\mathrm{SSIP}^{(\alpha)}(\theta_1,\theta_2)$-evolutions of Theorem~\ref{thm:hunt}. 
	\end{proof}
	
	For $(\gamma_c, \gamma_f)\in \cI^2_{\mathrm{nest}}$, let $(\fW_U, U\in \gamma_c)$ be a family of independent clades, with each $\fW_U \sim \mathbf{Q}^{(\alpha)}_{\gamma_{f}|^{\leftarrow}_{U}} (\theta_1,\theta_2)$. Let $\xi_{U}^{(\bar{\alpha})}$ be the scaffolding associated with $\fW_U$ as in \eqref{eq:scaffoldingW} and write $\mathrm{len}(\fW_U):= \inf \{s\ge 0\colon \xi_{U}^{(\bar{\alpha})} (s)=0\}$ for its length, which is a.s.\@ finite. 
	Then we define the concatenation of $(\fW_U, U\in \gamma_c)$ by
	\begin{equation*}\label{eq:concatenation-clade}
		\Concat_{U\in \gamma_c} \fW_U
		:= \sum_{U\in\gamma_c} \int \delta (g(U)\!+\!t, \boldsymbol{\beta}) \fW_{U} (dt,d\boldsymbol{\beta}), 
		\, \text{where}~ g(U)=\!\! \sum_{V\in\gamma_c, \sup V \le  \inf U}\!\! \mathrm{len} (\fW_{V}).
	\end{equation*}
	Write $\mathbf{Q}^{(\alpha)}_{(\gamma_c,\gamma_{f})} (\theta_1,\theta_2)$ for the law of $\Concat_{U\in \gamma_c} \fW_U$. 
	We next present a Markov-like property for such point measures of interval partition excursions, analogous to \cite[Proposition 6.6]{Paper1-1}. 
	
	\begin{lemma}\label{lem:nest-Markov-like}
		For $(\gamma_c, \gamma_f)\!\in\! \cI^2_{\mathrm{nest}}$, let $\fW\sim \mathbf{Q}^{(\alpha)}_{(\gamma_c,\gamma_{f})} (\theta_1,\theta_2)$ and  $\mathbf{X}\!=\! \xi_{\fW}^{(\bar{\alpha})}$. 
		For $y\!\ge\! 0$, set 
		\[
		\mathrm{cutoff}^{\ge y}_{\fW}=\!\!\sum_{\text{points }(t,\boldsymbol{\gamma}_t) \text{ of } \mathbf{W}} \!\! \ind\{ \mathbf{X}(t-)\ge  y\}\delta(\sigma^y(t), \boldsymbol{\gamma}_t) + \ind\{y\in (\mathbf{X}(t-), \mathbf{X}(t)) \}\delta(\sigma^y(t), \widehat{\boldsymbol{\gamma}}^y_t), 
		\]
		where $\sigma^y(t)=\mathrm{Leb} \{ u\!\le\! t\colon\mathbf{X}(u)\!>\! y\}$ and 
		$\widehat{\boldsymbol{\gamma}}^y_t = (\gamma_t(y\!-\!\mathbf{X}(t-)\!+\!z), z\!\ge\! 0)$. Similarly define $\mathrm{cutoff}^{\le y}_{\fW}$. 
		Given 
		$(\beta_c(y),\beta_f(y))=(\cskewer(y,\fW, \mathbf{X}),\fskewer(y,\fW, \mathbf{X}))$, $\mathrm{cutoff}^{\ge y}_{\fW}$ is conditionally independent of $\mathrm{cutoff}^{\le y}_{\fW}$ and has conditional distribution $\mathbf{Q}^{(\alpha)}_{(\beta_c(y),\beta_f(y))} (\theta_1,\theta_2)$. 
	\end{lemma}
	\begin{proof}
		Recall that the construction of the nested processes is a modification of  the scaffolding-and-spindles construction of the coarse component, with the same scaffolding and the $\Lambda_{\mathtt{BESQ}}^{(-2\bar\alpha)}$-excursions being replaced by the interval-partition excursions under $\Theta$. 
		In view of this, we can follow the same arguments as in the proof of \cite[Proposition 6.6]{Paper1-1}, with an application of Lemma~\ref{lem:Theta-Ha}.  
	\end{proof}
	\begin{proof}[Proof of Proposition~\ref{prop:nest-Markov}]
		The path-continuity follows directly from that of an SSIP-evolution. 
		As in \cite[Corollary 6.7]{Paper1-1}, Lemma~\ref{lem:nest-Markov-like} can be translated to the skewer process under $\mathbf{Q}^{(\alpha)}_{(\gamma_c,\gamma_{f})} (\theta_1,\theta_2)$, thus giving the Markov property for  $\mathrm{cfSSIP}^{(\alpha,\theta_1,\theta_2)} (0)$-evolutions. 
		
		When  the immigration rate is $\bar{\theta}>0$, we introduce an excursion measure $\Theta_{\mathrm{nest}}$ on the space of continuous $\cI_{\mathrm{nest}}^2$-excursions, such that the coarse excursion is a $\Theta^{(\bar{\alpha})}(0,\bar{\alpha})$, and each of its $\besq(-2\bar{\alpha})$-excursions is split into a $\Theta^{(\alpha)}(\theta_1,\theta_2)$-excursion. 
		More precisely, for $y>0$, it has the following properties: 
		\begin{enumerate}
			\item  $\Theta_{\mathrm{nest}}(\zeta > y) = \Theta^{(\bar{\alpha})}(0,\bar{\alpha})(\zeta>y) = (2y)^{-1}$. 
			\item If $(\boldsymbol{\beta}_c,\boldsymbol{\beta}_f)\sim\Theta_{\mathrm{nest}}(\,\cdot\,|\,\zeta > y)$, then $(\beta_c(y),\beta_f(y))
			\ed \mathtt{Exponential} ( 1/2y) (\bar\gamma_c,\bar\gamma_f)$, where $\bar\gamma_c\sim \mathtt{PDIP}^{(\bar{\alpha})}(0)$ and the conditional distribution of $\bar\gamma_f$ given $\bar\gamma_c$ is $\mathrm{Frag}^{(\alpha)}(\theta_1,\theta_2)$. Moreover, conditionally on $(\beta_c(y),\beta_f(y))$, the process 
			$((\beta_c(y+z),\beta_f(y+z)),\, z\ge 0)$ is a 
			$\mathrm{cfSSIP}^{(\alpha,\theta_1, \theta_2)}(0)$-evolution. 
		\end{enumerate} 
		Having obtained the pseudo-stationarity (Proposition~\ref{prop:cf-ps-theta1theta2}) and the Markov property of $\mathrm{cfSSIP}^{(\alpha,\theta_1,\theta_2)} (0)$-evolutions, the construction of $\Theta_{\mathrm{nest}}$ can be made by a similar approach as in Section~\ref{sec:exc}. 
		
		Using $\mathbf{F}\sim \mathtt{PRM} (\bar{\theta}\mathrm{Leb}\otimes \Theta_{\mathrm{nest}})$, by the construction in \cite[Section 3]{IPPAT}, the following process has the same law as a $\mathrm{cfSSIP}^{(\alpha,\theta_1, \theta_2)}(\bar{\theta})$-evolution starting from $(\emptyset,\emptyset)$, for $y\ge 0$, 
		\[
		\beta_{c}(y)= \Concat_{\text{points }(s,\boldsymbol{\gamma}_c,\boldsymbol{\gamma_f})\text{ of }\mathbf{F}\colon s\in [0,y]\downarrow}\gamma_c(y-s), \quad 
		\beta_{f}(y)= \Concat_{\text{points }(s,\boldsymbol{\gamma}_c,\boldsymbol{\gamma_f})\text{ of }\mathbf{F}\colon s\in [0,y]\downarrow}\gamma_f(y-s).
		\] 
		The Markov property of $\mathrm{cfSSIP}^{(\alpha,\theta_1,\theta_2)} (\bar{\theta})$-evolutions is now a consequence of this Poissonian construction and the form of $\Theta_{\mathrm{nest}}$; see the proof of \cite[Lemma 3.10]{IPPAT} for details. 
	\end{proof}
	
\end{appendix}
%
%

 {\bf Acknowledgments. } The authors would like to thank the referees for their valuable feedback and constructive comments. 
 
\smallskip 
 
	    {\bf Funding. } This work was partially supported by the National Key R\&D Program of China (grant 2022YFA1006500), National Natural Science Foundation of China (grants 12288201 and 12301169) and SNSF grant P2ZHP2\_171955. 
\bibliographystyle{imsart-number} 
\bibliography{AldousDiffusion4}       


\end{document}